\documentclass[a4paper,12pt,reqno]{amsart}

\usepackage[margin=1in]{geometry}
\usepackage{hyperref}
\usepackage{amsmath,mathtools}
\usepackage{enumerate}
\usepackage{tikz}
\usepackage[all]{xy}
\usepackage{verbatim}
\usepackage{setspace}
\usepackage{enumitem}
\usepackage{mathdots}
\usetikzlibrary{patterns}

\xyoption{rotate}

\DeclareMathOperator{\Hom}{Hom}

\DeclareMathOperator{\End}{End}

\DeclareMathOperator{\im}{Im}
\DeclareMathOperator{\Ker}{Ker}
\DeclareMathOperator{\spn}{span}

\DeclareMathOperator{\tp}{top}
\DeclareMathOperator{\rad}{rad}
\DeclareMathOperator{\soc}{soc}

\DeclareMathOperator{\Mod*}{mod}
\DeclareMathOperator{\MOD}{Mod}

\DeclareMathOperator{\ind}{ind}
\DeclareMathOperator{\uind}{\underline{ind}}

\DeclareMathOperator{\Coker}{Coker}
\DeclareMathOperator{\val}{val}

\DeclareMathOperator{\dimbf}{\textbf{dim}}
\DeclareMathOperator{\id}{id}
\DeclareMathOperator{\proj}{proj}
\DeclareMathOperator{\add}{add}
\DeclareMathOperator{\fin}{fin}

\DeclareMathOperator{\Ob}{Ob}

\theoremstyle{plain}
\newtheorem{thm}{Theorem}[section]
\newtheorem*{thm*}{Theorem}
\newtheorem{lem}[thm]{Lemma}

\newtheorem{prop}[thm]{Proposition}
\newtheorem*{prop*}{Proposition}

\theoremstyle{definition}
\newtheorem{defn}[thm]{Definition}

\newtheorem{exam}[thm]{Example}

\newtheorem*{exam*}{Example}

\theoremstyle{remark}
\newtheorem{rem}[thm]{Remark}
\newtheorem{assumption}[thm]{Assumption}

\numberwithin{equation}{section}
\allowdisplaybreaks

\begin{document}
	\title{Tame and Wild Symmetric Special Multiserial Algebras}
	\author{Drew Duffield}
	\maketitle
	\begin{abstract}
		We provide a complete classification of all tame and wild symmetric special multiserial algebras in terms of the underlying Brauer configuration. Our classification contains the symmetric special multiserial algebras of finite representation type.
	\end{abstract}
	
	\section{Introduction}
	The representation type of a finite dimensional algebra is of fundamental importance to representation theorists. Drozd's famous dichotomy (\cite{Drozd}) shows that an algebra can either be of tame or wild representation type (which becomes a trichotomy if one distinguishes between algebras that are of finite or infinite representation type). For tame algebras, the indecomposable modules in each dimension occur in a finite number of one-parameter families, meaning that there is at least some hope of a classification of the indecomposable modules. On the other hand, the representation theory of any wild algebra is at least as complicated as the representation theory of all finite dimensional algebras. Thus, a classification of the indecomposable modules of a wild algebra is often considered to be hopeless. It is therefore of tremendous importance to know whether an algebra is tame or wild if one aims to develop a detailed understanding of its representation theory.
	
	The class of special biserial algebras have been of great interest and study, and their representation theory is well-understood. For example, special biserial algebras are of tame representation type (\cite{TameBiserial}, \cite{WaldWasch}). The indecomposable modules of special biserial algebras have been classified using the functorial filtration method due to Gel'fand and Ponomarev (\cite{FuncFilt}) and the morphisms between them have been studied in \cite{CBMaps} and \cite{TreeMaps}. The Auslander-Reiten quiver of special biserial algebras is also well-understood (\cite{butlerRingel}), particularly for those that are self-injective (\cite{ErdmannAR}) and symmetric (\cite{Chinburg}, \cite{DuffieldBGA}, \cite{GroupsWOGroups}).
	
	Special biserial algebras have been instrumental to many classification problems regarding the representation type of an algebra. For example, they are used in the derived equivalence classification of tame self-injective algebras that are either periodic, standard, or of polynomial growth (\cite{FiniteSelfInj}, \cite{TameSelfInj}). Special biserial algebras also play a role in the classification of blocks of group algebras of tame (and finite) representation type (see for example \cite{ErdmannString}).
	
	The focus of this paper is on the broader class of special multiserial algebras, which were first introduced in \cite{VonHohne} and later investigated in \cite{AlmostGentle}, \cite{BCA} and \cite{Multiserial}. They both generalise and contain the class of special biserial algebras, and similar to special biserial algebras, the representation theory is generally controlled by the uniserial modules over the algebra. Unlike biserial algebras though, most special multiserial algebras are wild. Symmetric radical cube zero algebras are a subclass of special multiserial algebras (as shown in \cite{Multiserial}), and for these, a classification of finite, tame and wild algebras exists in \cite{BensonRad}. However, such a classification does not currently exist for symmetric special multiserial algebras in general. It is precisely the aim of this paper to provide a classification of all tame and wild symmetric special multiserial algebras.
	
	A helpful tool in understanding the representation theory of symmetric special multiserial algebras is the notion of a Brauer configuration. A Brauer configuration is a decorated hypergraph with orientation -- that is, a collection of vertices and connected polygons, with a cyclic ordering of the polygons at each vertex. Every Brauer configuration gives rise to a symmetric special multiserial algebra (\cite{BCA}). Conversely, every symmetric special multiserial algebra can be associated to a Brauer configuration (\cite{Multiserial}). Brauer configurations are particularly useful, as the representation theory of the algebra is encoded in the combinatorial data of the hypergraph. If every polygon in the Brauer configuration is a 2-gon (or edge), then one obtains a Brauer graph, and hence a symmetric special biserial algebra.
	
	This paper is centred around proving the main result given in Section~\ref{MainResult}, which uses a variety of mathematical techniques. For example, one test for the wildness of an algebra, detailed in \cite{Boevey}, is to show that there exist infinitely many indecomposable modules of dimension $d$ which lie in an Auslander-Reiten component that is not a homogeneous tube. In Sections~\ref{Examples} and \ref{nSerialSection}, we give examples of one-parameter families of indecomposable modules in symmetric special multiserial algebras, some of which are the well known class of band modules. We then show that for certain symmetric special multiserial algebras, these modules do not belong to homogeneous tubes in the Auslander-Reiten quiver, thus showing that these algebras are wild. We use this technique in particular to show that there is precisely one tame symmetric special quadserial algebra, and that every symmetric special $n$-serial algebra with $n>4$ is wild.
	
	The most involved cases in the proof of the main theorem concern the subclass of symmetric special triserial algebras. This forms the majority of the paper. In Section~\ref{QCSection}, we prove the existence of a class of tame symmetric special triserial algebras which are closely related to the (tame) classes of clannish algebras and skewed gentle algebras in \cite{Clans} and \cite{Pena}. The proof involves an indirect classification of the indecomposable modules.
	
	In Section~\ref{TriserialResults}, we show which symmetric special triserial algebras (with an underlying Brauer configuration that is a tree) are derived equivalent to the trivial extension of a hereditary algebra. The techniques we use are similar to those used by Rickard for Brauer tree algebras in \cite{Rickard}. In particular, we determine which symmetric special triserial algebras are derived equivalent to the trivial extension of some wild hereditary algebra, and which are derived equivalent to the trivial extension of some hereditary algebra of type $\mathbb{E}$ or $\widetilde{\mathbb{E}}$.
	
	The aim of Sections~\ref{CycleMultSubsection} and \ref{Multiple3GonSubsection} is to prove that any symmetric special triserial algebra which is not one of the algebras from Sections~\ref{QCSection} and \ref{TriserialResults} is wild. The proof involves the construction of functors analogous to representation embeddings (c.f. \cite{WildBook}) from a wild hereditary algebra to the symmetric special triserial algebra. In the case of Section~\ref{CycleMultSubsection}, this provides an explicit description of a class of two parameter families of indecomposable modules.

	Finally, in Section~\ref{MainProof}, we bring together all of the results in Sections~\ref{Examples}-\ref{Multiple3GonSubsection} to prove the main theorem in Section~\ref{MainResult}.

	\section*{Acknowledgements}
	This paper is an outcome of my PhD, so I would like to thank EPSRC for providing the funding for my research project. I would also like to give a special thank you to my PhD supervisor, Prof Sibylle Schroll, for taking the time to proofread the work for this paper and for the many helpful discussions along the way. Finally, I would like to thank Edward Green and {\O}yvind Solberg for the development of the QPA package in GAP, which proved to be enormously useful in the various calculations behind this paper.
	
	\section{The Classification Theorem} \label{MainResult}
	The aim of this paper is to prove the following result.
	
	\begin{thm} \label{Result1}
		Let $A$ be a symmetric special multiserial algebra corresponding to a Brauer configuration $\chi$. Then $A$ is tame if and only if $\chi$ satisfies one of the following.
		\begin{enumerate}[label=(\roman*)]
			\item $\chi$ is a Brauer graph.
			\item $\chi$ is of the form
			
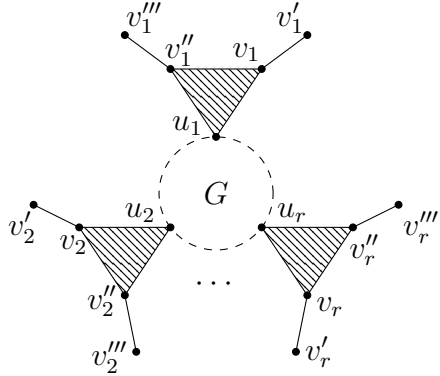
\begin{center}
				\begin{tikzpicture}[scale=1.5, ]
				\draw [dashed] (0,0) ellipse (0.5 and 0.5);
				\draw (0,0) node {$G$};
				\draw [fill=black] (0,0.5) ellipse (0.03 and 0.03);
				\draw (-0.25,0.6) node {$u_1$};
				\draw [fill=black] (-0.4,-0.3) ellipse (0.03 and 0.03);
				\draw (-0.667,-0.1665) node {$u_2$};
				\draw [fill=black] (0.4,-0.3) ellipse (0.03 and 0.03);
				\draw (0.6665,-0.1665) node {$u_r$};
				
				\draw [fill=black] (0.4,1.1) ellipse (0.03 and 0.03);
				\draw (0.2665,1.25) node {$v_1$};
				\draw [fill=black] (0.8,1.4) ellipse (0.03 and 0.03);
				\draw (0.65,1.55) node {$v'_1$};
				\draw [fill=black] (-0.4,1.1) ellipse (0.03 and 0.03);
				\draw (-0.3165,1.2634) node {$v''_1$};
				\draw [fill=black] (-0.8,1.4) ellipse (0.03 and 0.03);
				\draw (-0.6299,1.55) node (o) {$v'''_1$};
				\draw [fill=black] (-1.2,-0.3) ellipse (0.03 and 0.03);
				\draw (-1.25,-0.45) node {$v_2$};
				\draw [fill=black] (-1.6,-0.1) ellipse (0.03 and 0.03);
				\draw (-1.7,-0.3) node {$v'_2$};
				\draw [fill=black] (-0.8,-0.9) ellipse (0.03 and 0.03);
				\draw (-1,-0.95) node {$v''_2$};
				\draw [fill=black] (-0.7,-1.4) ellipse (0.03 and 0.03);
				\draw (-0.9335,-1.45) node {$v'''_2$};
				\draw [fill=black] (0.8,-0.9) ellipse (0.03 and 0.03);
				\draw (1,-1) node {$v_r$};
				\draw [fill=black] (0.7,-1.4) ellipse (0.03 and 0.03);
				\draw (0.9,-1.4) node {$v'_r$};
				\draw [fill=black] (1.2,-0.3) ellipse (0.03 and 0.03);
				\draw (1.3,-0.5) node {$v''_r$};
				\draw [fill=black] (1.6,-0.1) ellipse (0.03 and 0.03);
				\draw (1.8,-0.3) node {$v'''_r$};
				
				\draw [pattern=north west lines](0,0.5) -- (-0.4,1.1) -- (0.4,1.1) -- (0,0.5);
				\draw (-0.8,1.4) -- (-0.4,1.1);
				\draw (0.8,1.4) -- (0.4,1.1);
				\draw [pattern=north west lines](-1.2,-0.3) -- (-0.4,-0.3) -- (-0.8,-0.9) -- (-1.2,-0.3);
				\draw (-1.6,-0.1) -- (-1.2,-0.3);
				\draw (-0.7,-1.4) -- (-0.8,-0.9);
				\draw [pattern=north west lines](0.4,-0.3) -- (1.2,-0.3) -- (0.8,-0.9) -- (0.4,-0.3);
				\draw (1.6,-0.1) -- (1.2,-0.3);
				\draw (0.7,-1.4) -- (0.8,-0.9);
				\draw (0,-0.8) node {$\cdots$};
				\end{tikzpicture}
			\end{center}
			where $G$ is a Brauer graph connecting the (not necessarily distinct) vertices $u_1,\ldots,u_r$ and $\mathfrak{e}_{v_i}=\mathfrak{e}_{v'_i}=\mathfrak{e}_{v''_i}=\mathfrak{e}_{v'''_i}=1$ for all $i$.
			\item $\chi$ is of the form
			
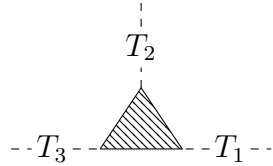
\begin{center}
				\begin{tikzpicture}[baseline = (o.base),scale=0.9]
					\draw[pattern = north west lines] (0.2,-1) -- (1.4,-1) -- (0.8,-0.1) -- (0.2,-1);
					\draw[dashed] (1.8,-1) -- (1.4,-1);
					\draw[dashed] (2.4,-1) -- (2.7,-1);
					\draw (2.1,-1) node{$T_1$};
					\draw[dashed] (0.8,-0.1) -- (0.8,0.2);
					\draw[dashed] (0.8,0.8) -- (0.8,1.2) node (o) {};
					\draw (0.8,0.5) node{$T_2$};
					\draw[dashed] (-1.1,-1) -- (-0.8,-1);
					\draw[dashed] (-0.2,-1) -- (0.2,-1);
					\draw (-0.5,-1) node{$T_3$};
				\end{tikzpicture}
			\end{center}
			where $T_1$, $T_2$ and $T_3$ are distinct multiplicity-free Brauer trees containing $m_1$, $m_2$ and $m_3$ polygons respectively such that the values of the triple $(m_1,m_2,m_3)$ conform to a column of the following table.
			\begin{center}
				\begin{tabular}{c | c c c c c c}
					$m_1$	&	$1$	&	$1$	&	$1$	&	$1$	&	$1$	&	$2$	\\	\hline
					$m_2$	&	$2$	&	$2$	&	$2$	&	$2$	&	$3$	&	$2$	\\	\hline
					$m_3$	&	$2$	&	$3$	&	$4$	&	$5$	&	$3$	&	$2$
				\end{tabular}
			\end{center}
			\item $\chi$ is of the form
			
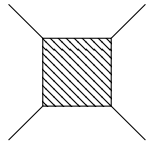
\begin{center}
				\begin{tikzpicture}[scale=0.9]
					\draw[pattern = north west lines] (0.5,-0.5) -- (0.5,0.5) -- (1.5,0.5) -- (1.5,-0.5) -- (0.5,-0.5);
					\draw (0,1) -- (0.5,0.5);
					\draw (0,-1) -- (0.5,-0.5);
					\draw (1.5,0.5) -- (2,1);
					\draw (1.5,-0.5) -- (2,-1);
				\end{tikzpicture}
			\end{center}
			where every vertex has multiplicity one.
		\end{enumerate}
	\end{thm}
	
	\section{Preliminaries} \label{preliminaries}
	Throughout, we let $K$ be an algebraically closed field and $Q=(Q_0,Q_1)$ be a finite connected quiver with vertex set $Q_0$ and arrow set $Q_1$. We let $I$ be an admissible ideal of the path algebra $KQ$ such that $KQ/I$ is a basic finite dimensional algebra. We denote by $\Mod* A$ the category of finitely generated right $A$-modules. For any vertex $x \in Q_0$, we denote by $S(x)$ and $P(x)$ the simple module and indecomposable projective module respectively corresponding to $x$.
	
	\subsection{Representation Type}
	The definitions of this section are taken from \cite[XIX]{WildBook}. Let $A$ and $B$ be arbitrary $K$-algebras (that are not necessarily finite-dimensional). Let $\mathcal{A}\subseteq \MOD A$ and $\mathcal{B}\subseteq \MOD B$ be full exact additive subcategories that are closed under direct summands. Let $F:\mathcal{B} \rightarrow \mathcal{A}$ be a $K$-linear functor. We say $F$ \emph{respects isomorphism classes} if for any modules $M,M' \in \mathcal{B}$, we have $FM \cong FM' \Rightarrow M \cong M'$. We say $F$ is a \emph{representation embedding} if it is exact, respects isomorphism classes, and maps indecomposable modules in $\mathcal{B}$ to indecomposable modules in $\mathcal{A}$. A $K$-linear functor $F:\mathcal{B} \rightarrow \mathcal{A}$ is said to be a \emph{strict representation embedding} if is exact and fully faithful. It is known that strict representation embeddings are representation embeddings (\cite[XIX, Lemma 1.2]{WildBook}).
	
	There are two equivalent definitions for a finite-dimensional algebra to be of wild representation type. A finite dimensional $K$-algebra $A$ is said to be of \emph{wild representation type} (or shortly, is said to be \emph{wild}) if for every finite dimensional $K$-algebra $B$, there exists a representation embedding $F:\Mod*B \rightarrow \Mod*A$. Equivalently, there exists a representation embedding $F:\fin K\langle a_1, a_2 \rangle \rightarrow \Mod*A$. Note that if for some algebra $A$ of unknown representation type and for some algebra $B$ of wild representation type, it follows that if there exists a representation embedding $F:\Mod*B \rightarrow \Mod*A$, then $A$ is also wild.
	
	On the other hand, a finite dimensional algebra $A$ is said to be of \emph{tame representation type} (or shortly, is said to be \emph{tame}) if for each strictly positive integer $d$, there exists a finite number of $K[a]$-$A$-bimodules $M_1, \ldots, M_{n_d}$ that are finitely generated and free as left $K[a]$-modules such that all (except perhaps finitely many) indecomposable modules of dimension $d$ are isomorphic to a module of the form $S \otimes_{K[a]} M_i$ for some simple right $K[a]$-module $S$ and some $1 \leq i \leq n_d$.
	
	\subsection{Tilting Complexes}
	Denote by $\proj A$ the full subcategory of $\Mod* A$ consisting of projective $A$-modules. By $K^b(\proj A)$, we mean the bounded homotopy category of chain complexes over $\proj A$. We call an object in $K^b(\proj A)$ that has a non-zero term in at most one degree a \emph{stalk complex}. Given an object $T \in K^b(\proj A)$, denote by $\add(T)$ the full subcategory of $K^b(\proj A)$ consisting of direct summands of direct sums of copies of $T$. We call an object $T \in K^b(\proj A)$ a \emph{tilting complex} if $\Hom(T, T[n])=0$ for all $n \neq 0$ and $\add(T)$ generates $K^b(\proj A)$ as a triangulated category. We use the following result of Rickard in Section~\ref{TriserialResults}, which is particularly useful to us, as derived equivalent self-injective algebras have the same representation type.
	
	\begin{thm}[\cite{Rickard}, Theorem 1.1] \label{RickardTheorem}
		Let $A$ and $B$ be finite dimensional algebras. Then $A$ and $B$ are derived equivalent if and only if $B$ is isomorphic to the endomorphism ring of a tilting complex.
	\end{thm}
	
	An example of a tilting complex in a symmetric algebra $A$ is an Okuyama-Rickard complex (c.f. \cite{Okuyama}, \cite{TiltingHandbook}). Let $\varepsilon_1, \ldots, \varepsilon_n$ be complete set of primitive orthogonal idempotents of $A$. Let $E'$ be a subset of $E = \{1,\ldots, n\}$ and let $\varepsilon=\sum_{i \in E'}\varepsilon_i$. Define $T_i$ to be either the stalk complex with degree zero term $\varepsilon_i A$ if $i \in E'$ or the complex
	\begin{equation*}
		\xymatrix@1{0 \ar[r] & P(\varepsilon_i A\varepsilon A) \ar[r]^-f & \varepsilon_i A \ar[r] & 0}
	\end{equation*}
	if $i \not\in E'$, where $P(\varepsilon_i A\varepsilon A)$ is in degree zero and $\xymatrix@1{P(\varepsilon_i A\varepsilon A) \ar[r]^-f & \varepsilon_i A}$ is the minimal projective presentation of $\varepsilon_i A / \varepsilon_i A\varepsilon A$. Then the complex $T=\bigoplus_{i \in E} T_i$ is called the \emph{Okuyama-Rickard} tilting complex with respect to $E'$.
	
	\subsection{Special Multiserial Algebras}
	We follow the definitions in \cite{Multiserial}. Let $A$ be a finite dimensional algebra. We say a left or right $A$-module is \emph{multiserial} (or \emph{$n$-serial}) if $\rad(M)$ can be written as a sum of uniserial modules $U_1, \ldots, U_n$ such that $U_i \cap U_j$ is simple or zero for all $i \neq j$. We say that an algebra is \emph{multiserial} (or \emph{$n$-serial}) if $A$ is multiserial (resp. $n$-serial) as a left or right $A$-module. In particular, if $A$ is an $n$-serial algebra for $n\in\{1,2,3,4\}$, then we say that $A$ is uniserial, biserial, triserial or quadserial respectively.
	
	We say that a finite dimensional algebra $A$ is \emph{special multiserial} if it is Morita equivalent to a quotient $KQ/I$ of a path algebra $KQ$ by an admissible ideal $I$ such that the following property holds. For any arrow $\alpha \in Q_1$, there exists at most one arrow $\beta \in Q_1$ and at most one arrow $\gamma \in Q_1$ such that $\alpha\beta \not\in I$ and $\gamma \alpha \not\in I$. Note that special multiserial algebras are multiserial algebras (\cite[Corollary 2.4]{Multiserial}).
	
	\subsection{Configurations}
	We present an alternative (but equivalent) definition of a hypergraph, which we call a configuration.
	\begin{defn}
		A \emph{configuration} is a tuple $\chi=(\chi_0,\chi_1, \kappa)$ where the following hold.
		\begin{enumerate}[label=(\roman*)]
			\item $\chi_0$ is a finite set whose elements are called \emph{vertices}.
			\item $\chi_1$ is a finite collection of finite sets. We call a set $x \in \chi_1$ a \emph{polygon} of $\chi$ and require that $|x| \geq 2$ for all $x \in \chi_1$. Specifically, we call $x \in \chi_1$ an $n$-gon if $|x|=n$. For each polygon $x \in \chi_1$, we call the elements of $x$ the \emph{germs of the polygon $x$}. By definition, we say that $x \cap y = \emptyset$ for any distinct $x,y\in \chi_1$. We denote the set of all germs of polygons in $\chi$ by $\mathcal{G}_\chi=\bigcup_{x \in \chi_1} x$.
			\item $\kappa\colon\mathcal{G}_\chi \rightarrow \chi_0$ is a function, which may be considered as a map taking each germ of a polygon to an incident vertex.
		\end{enumerate}
	\end{defn}
	
	A polygon $x$ is said to be \emph{incident} (or \emph{connected}) to a vertex $v$ if there exists $g \in x$ such that $\kappa(g)=v$. We define the \emph{valency} of a vertex $v$ in a configuration to be the integer $\val(v)=|\{g \in \mathcal{G}_\chi : \kappa(g)=v\}|$. Informally, one can realise a configuration as a generalisation of a graph, where instead of vertices and connected edges, we have vertices and connected polygons. We typically realise $2$-gons as edges in the configuration. If a configuration $\chi$ consists entirely of 2-gons, then $\chi$ is indeed equivalent to a graph. 
	
	We say a polygon $x$ in $\chi$ is \emph{self-folded} (at the vertex $v$) if there exist at least two distinct germs $g_1, g_2 \in x$ such that $\kappa(g_1)=v=\kappa(g_2)$. Specifically, if there are precisely $m$ distinct germs $g_1,\ldots, g_m \in x$ such that $\kappa(g_i)=v$ for all $i$, we say that $x$ is $m$-self-folded at $v$. Self-folded polygons in a configuration generalise the notion of loops in a graph, and indeed a self-folded 2-gon is considered as a \emph{loop}.
	
	For the purposes of readability, we will now establish a notation for germs of polygons in a configuration $\chi$. If $g$ is a unique germ of a polygon $x\in\chi_1$ such that $\kappa(g)=v\in\chi_0$, then we will write $g$ as $x^v$. It is therefore implicitly assumed that for a germ of a polygon $x^v \in \mathcal{G}_\chi$, we have $x^v \in x$ and $\kappa(x^v)=v$ for some polygon $x$ and vertex $v$ in $\chi$. On the other hand, if a polygon $x$ in $\chi$ is $m$-self-folded at a vertex $v$, then we will write the germs $g_1,\ldots,g_m\in x$ satisfying $\kappa(g_i)=v$ for each $i$ as $x^{v,1},\ldots,x^{v,m}$.
	
	A \emph{path of length} $n$ in a configuration $\chi$ is a sequence 
	\begin{equation*}
		(v_0,x^{v_0}_1, x^{v_1}_1, v_1, x^{v_1}_2, x^{v_2}_2, v_2, \ldots, v_{n-1}, x^{v_{n-1}}_n, x^{v_n}_n, v_n)
	\end{equation*}
	of vertices and germs of polygons such that $x^{v_i}_i \in x_i \setminus \{x^{v_{i-1}}_i\}$ for all $i$. We say a polygon $x$ is in $p$ if there exists a germ of a polygon $g$ in $p$ such that $g \in x$. Where the context is clear and there are no ambiguities arising from self-folded polygons and multiple polygons between vertices, we will often write the above sequence as
	\begin{equation*}
		\xymatrix@1{v_0 \ar@{-}[r]^-{x_1} & v_1 \ar@{-}[r]^-{x_2} & v_2 \ar@{-}[r] & \cdots \ar@{-}[r] & v_{n-1} \ar@{-}[r]^-{x_n} & v_n}.
	\end{equation*}
	We say a path in $\chi$ is \emph{simple} if it is non-crossing at polygons and vertices. That is, each vertex in $p$ occurs precisely once and for any polygon $x$ in $p$, there are precisely two germs of $x$ in $p$. We say a path in $\chi$ is a \emph{cycle} if the starting and ending vertices are the same. A cycle is said to be \emph{simple} if it is non-crossing at polygons and vertices (except at the starting and ending vertices). We say a configuration $\chi$ is \emph{connected} if there exists a path between any two vertices of $\chi$. We say $\chi$ is a tree if there exists a unique simple path between any two vertices of $\chi$. We say $\chi'=(\chi'_0,\chi'_1, \kappa')$ is a \emph{subconfiguration} of $\chi=(\chi_0,\chi_1, \kappa)$ if $\chi'$ is a connected configuration such that $\chi'_0 \subseteq \chi_0$, $\chi'_1 \subseteq \chi_1$ and $\kappa'(g)=\kappa(g)$ for all $g \in \mathcal{G}_{\chi'}$.
	
	\subsection{Brauer Configuration Algebras}
	The definitions presented here are based on the work of \cite{BCA}. A non-empty connected configuration $\chi=(\chi_0,\chi_1, \kappa)$ is called a \emph{Brauer configuration} if the we have the following additional properties.
	\begin{enumerate}[label=(\roman*)]
		\item To each $v \in \chi_0$, we equip a cyclic ordering $\mathfrak{o}_v$ of the germs of polygons incident to $v$.
		\item To each $v \in \chi_0$, we assign a strictly positive integer $\mathfrak{e}_v$ called the \emph{multiplicity} of the vertex.
		\item For any $x \in \chi_1$ such that $|x|>2$, there exists no $g \in x$ such that $\val(\kappa(g))=1$ and $\mathfrak{e}_{\kappa(g)}=1$.
	\end{enumerate}
	A Brauer configuration generalises the notion of a ribbon graph with weighted vertices, and may be realised geometrically as a local embedding of the polygons around each vertex in the oriented plane. We shall use an anticlockwise cyclic ordering throughout. If every polygon of $\chi$ is a 2-gon (and thus, $\chi$ is a graph), then $\chi$ is a \emph{Brauer graph}. We call a Brauer graph $\chi$ a \emph{Brauer tree} if $\chi$ is a tree and at most one vertex $v$ in $\chi$ has multiplicity $\mathfrak{e}_v>1$. If every vertex of a Brauer configuration has multiplicity one, then we say that $\chi$ is \emph{multiplicity-free}.
	
	We say a vertex $v$ of a Brauer configuration $\chi$ is \emph{truncated} if $\val(v)=1$ and $\mathfrak{e}_v=1$. It follows from condition (iii) of the definition of a Brauer configuration that any such vertex is connected to a unique 2-gon. We call such a polygon a \emph{truncated edge} of $\chi$.
	
	Let $x_1^v$ and $x_2^v$ be germs of polygons at the same vertex in a Brauer configuration. We say $x_2^v$ is the \emph{successor} to $x_1^v$ if $x_2^v$ directly follows $x_1^v$ in the cyclic ordering at $v$. We then say that the polygon $x_2$ is the successor to $x_1$ at $v$. From this we obtain a sequence $x_1, x_2, \ldots, x_{\val(v)}$, where each $x_i$ is the successor to $x_{i-1}$. We call this the \emph{successor sequence} of $x_1$ at $v$. Similarly, we say $x_2^v$ is the \emph{predecessor} to $x_1^v$ if $x_1^v$ directly follows $x_2^v$ in the cyclic ordering at $v$, and we say the polygon $x_2$ is the predecessor to $x_1$ at $v$. From this we obtain a (descending) sequence $x_{\val(v)}, \ldots, x_2, x_1$, where each $x_i$ is the predecessor to $x_{i-1}$. We call this the \emph{predecessor sequence} of $x_1$ at $v$.
	
	Given a Brauer configuration $\chi$, we construct an algebra $A$ as follows. If $\chi$ is a Brauer tree consisting of a single edge and two distinct connected vertices of multiplicity one, then we let $A=KQ/I$, where $Q$ is the quiver consisting of a loop $\alpha$ at a single vertex and $I$ is generated by the relation $\alpha^2$. Otherwise, we define $Q$ to be the quiver whose vertices are in bijective correspondence with the distinct polygons of $\chi$. If $x_2^v$ is the successor to $x_1^v$ at some non-truncated vertex $v$ in $\chi$, then there exists an arrow $x_1 \rightarrow x_2$ in $Q$. If $x$ is connected to a vertex $v$ such that $\val(v) = 1$ and $\mathfrak{e}_v>1$, then there exists a loop at $x$ in $Q$. If $\mathfrak{e}_v=1$ then no such loop exists. Each non-truncated vertex of $\chi$ therefore induces a cycle in $Q$, and no two such cycles share a common arrow. We denote by $\mathfrak{C}_v$ the cycle of $Q$ up to cyclic permutation generated by the non-truncated vertex $v$ in $\chi$. By $\mathfrak{C}_{v,\alpha}$, we denote the permutation of the cycle $\mathfrak{C}_{v}$ such that the first arrow is $\alpha$.
	
	We define a set of relations $\rho$ on $Q$ as follows. If $x$ is a truncated edge of $\chi$ and $\mathfrak{C}_{v,\gamma_1}=\gamma_1\ldots\gamma_n$ is the cycle induced by the non-truncated vertex $v$ connected to $x$ with $\gamma_1$ of source $x$, then $(\mathfrak{C}_{v,\gamma_1})^{\mathfrak{e}_v}\gamma_1 \in \rho$. If $u$ and $v$ are (possibly equal) non-truncated vertices connected to the same polygon $x$ and $\mathfrak{C}_{u,\gamma}$ and $\mathfrak{C}_{v,\delta}$ are cycles of source $x$ generated by the respective vertices $u$ and $v$ (for some arrows $\gamma$ and $\delta$ of source $x$), then $(\mathfrak{C}_{u,\gamma})^{\mathfrak{e}_u}-(\mathfrak{C}_{v,\delta})^{\mathfrak{e}_v} \in \rho$. Finally, if $\alpha\beta$ is a path of length two in $Q$ such that $\alpha\beta$ is not a subpath of any cycle $\mathfrak{C}_{v}$ of any non-truncated vertex $v$ of $\chi$, then $\alpha\beta \in \rho$. The algebra $A=KQ/I$, where $I$ is the ideal generated by $\rho$, is called the Brauer configuration algebra associated to $\chi$.
	
	\begin{exam}
		Let $\chi=(\chi_0, \chi_1, \kappa)$ be a configuration defined as follows. We let $\chi_0=\{v_1,v_2,v_3,v_4,v_5\}$ and
		\begin{equation*}
			\chi_1=\{x=\{x^{v_1}, x^{v_2}, x^{v_4, 1},x^{v_4, 2}\}, y=\{y^{v_1,1}, y^{v_1,2}\}, z=\{z^{v_2}, z^{v_3}\}, w=\{w^{v_4}, w^{v_5}\}\}.
		\end{equation*}
		The definition of $\kappa$ is implicit from the notation used for the germs of polygons. We give $\chi$ the structure of a Brauer configuration by setting
		\begin{align*}
			\mathfrak{o}_{v_1}&=[x^{v_1}, y^{v_1,1}, y^{v_1,2}]	&
			\mathfrak{o}_{v_2}&=[x^{v_2}, z^{v_2}]					&
			\mathfrak{o}_{v_3}&=[z^{v_3}]								\\
			\mathfrak{o}_{v_4}&=[x^{v_4, 1}, w^{v_4}, x^{v_4, 2}]	&
			\mathfrak{o}_{v_5}&=[w^{v_5}]
		\end{align*}
		and setting $\mathfrak{e}_{v_3}=2$ and $\mathfrak{e}_{v_i}=1$ for all $i \neq 3$. The Brauer configuration may be presented pictorially, as shown in Figure~\ref{BCAExample}. Here, $\chi$ contains a self-folded 4-gon $x$ and three 2-gons (edges), one of which is self-folded (a loop). The edge $w$ and connected vertex $v_5$ are both truncated, with all other polygons and vertices being non-truncated. From $\chi$, we obtain the quiver $Q$ illustrated in Figure~\ref{BCAExample}. We also obtain the set of relations
		\begin{multline*}
			\rho = \{ \gamma_2\gamma_3\gamma_1\gamma_2, \\
			\alpha_1\alpha_2\alpha_3 - \beta_1\beta_2, \beta_1\beta_2 - \gamma_1\gamma_2\gamma_3, \gamma_3\gamma_1\gamma_2- \gamma_1\gamma_2\gamma_3, \alpha_2\alpha_3\alpha_1 - \alpha_3\alpha_1\alpha_2, \beta_2\beta_1 - \delta^2, \\
			 \alpha_3\beta_1, \alpha_3\gamma_1, \alpha_3\gamma_3, \beta_2\alpha_1, \beta_2\gamma_1, \beta_2\gamma_3, \gamma_2\alpha_1, \gamma_2\beta_1, \gamma_2\gamma_1, \gamma_3\alpha_1, \gamma_3\beta_1,  \gamma_3^2, \alpha_2^2, \alpha_1\alpha_3, \beta_1\delta, \delta\beta_2\}.
		\end{multline*}
		The Brauer configuration algebra associated to $\chi$ is therefore $KQ/\langle \rho\rangle$.
	\end{exam}
	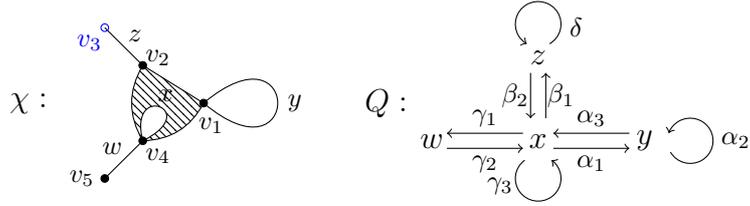
\begin{figure} 
		\begin{tikzpicture}
			\draw (0.3,-3.5) node {$\chi:$};
			\draw[pattern = north west lines] (1.8,-4) to[bend right] (2.6,-3.5) -- (1.8,-3) to[bend right] (1.8,-4);
			\draw (1.3,-2.5) -- (1.8,-3);
			\draw (1.3,-4.5) -- (1.8,-4);
			\draw (2.6,-3.5) .. controls (3.9,-4.6) and (3.9,-2.4) .. (2.6,-3.5);
			\draw [fill=white] (1.8,-4) .. controls (1.6,-3.2) and (2.6,-3.6) .. (1.8,-4);
		
			\draw (2.1,-3.4) node{\footnotesize$x$};
			\draw (3.8,-3.5) node{\footnotesize$y$};
			\draw (1.7,-2.6) node{\footnotesize$z$};
			\draw (1.4,-4.1) node{\footnotesize$w$};
			\draw (2.7,-3.8) node{\footnotesize$v_1$};
			\draw (2,-2.9) node{\footnotesize$v_2$};
			\draw [blue] (1.1,-2.7) node{\footnotesize$v_3$};
			\draw (2,-4.2) node{\footnotesize$v_4$};
			\draw (1,-4.5) node{\footnotesize$v_5$};
			
			\draw [blue] (1.3,-2.5) ellipse (0.05 and 0.05);
			\draw [fill=black] (2.6,-3.5) ellipse (0.05 and 0.05);
			\draw [fill=black] (1.3,-4.5) ellipse (0.05 and 0.05);
			\draw [fill=black] (1.8,-4) ellipse (0.05 and 0.05);
			\draw [fill=black] (1.8,-3) ellipse (0.05 and 0.05);
			
			\draw (5,-3.5) node {$Q:$};
			\draw (7,-4) node {$x$};
			\draw (8.4,-4) node {$y$};
			\draw (7,-2.9) node {$z$};
			\draw (5.6,-4) node {$w$};
			
			\draw [->](7.1,-3.7) -- (7.1,-3.1);
			\draw [->](6.9,-3.1) -- (6.9,-3.7);
			\draw [->](6.8,-3.9) -- (5.8,-3.9);
			\draw [->](5.8,-4.1) -- (6.8,-4.1);
			\draw [->](8.2,-3.9) -- (7.2,-3.9);
			\draw [->](7.2,-4.1) -- (8.2,-4.1);
			\draw [->](8.7402,-4.15) arc (-149.9993:160:0.3);
			
			\draw (7.7,-4.3) node {\footnotesize$\alpha_1$};
			\draw (9.6,-4) node {\footnotesize$\alpha_2$};
			\draw (7.7,-3.7) node {\footnotesize$\alpha_3$};
			\draw (7.3,-3.4) node {\footnotesize$\beta_1$};
			\draw (6.7,-3.4) node {\footnotesize$\beta_2$};
			\draw (6.3,-3.7) node {\footnotesize$\gamma_1$};
			\draw (6.3,-4.3) node {\footnotesize$\gamma_2$};
			\draw (6.5,-4.6) node {\footnotesize$\gamma_3$};
			\draw (7.5,-2.5) node {\footnotesize$\delta$};
			\draw [->](7.15,-2.7134) arc (-59.9993:240:0.3);
		
			\draw [->](6.8237,-4.2573) arc (125.9951:414:0.3);
		\end{tikzpicture}
		\caption{The above is an example of a Brauer configuration and its corresponding quiver $Q$. Here, we have $\mathfrak{e}_{v_3}>1$.} \label{BCAExample}
	\end{figure}

	The class of Brauer configuration algebras and the class of symmetric special multiserial algebras coincide (\cite[Theorem 4.1]{Multiserial}). Henceforth, we assume that $A=KQ/I$ is a Brauer configuration algebra constructed from a Brauer configuration $\chi$. By the stationary path at a vertex $x \in Q_0$, we mean the path of length zero of source (and target) $x$, which we denote by $\varepsilon_x$.
	
	\subsection{Strings in Brauer Configuration Algebras}
	We follow the definitions of \cite{butlerRingel}, but in the context of Brauer configuration algebras. To each arrow $\alpha \in Q_1$, we denote by $s(\alpha)$ the vertex $x \in Q_0$ at the source of $\alpha$, and by $e(\alpha)$ the vertex $y \in Q_0$ at the target of $\alpha$. Recall that the vertices of $Q_0$ are in correspondence with the polygons of $\chi$. Thus for each arrow $\alpha \in Q_1$, we can consider $s(\alpha)$ and $e(\alpha)$ to be polygons in $\chi$. Further recall that, the arrows of $Q_1$ correspond to ordered pairs of germs of polygons $(x^v, y^v)$ such that $y^v$ is the successor to $x^v$. Thus, we can also consider $\alpha$ to be an arrow between two germs of polygons. Denote by $\widehat{s}(\alpha)$ the germ of the polygon at the source of $\alpha$, and by $\widehat{e}(\alpha)$ the germ of the polygon at the target of $\alpha$. Given an arrow $\alpha \in Q_1$, we denote by $\alpha^{-1}$ the formal inverse of $\alpha$. That is, the symbolic arrow such that $s(\alpha^{-1})=e(\alpha)$ and $e(\alpha^{-1})=s(\alpha)$. We similarly define $\widehat{s}(\alpha^{-1}) = \widehat{e}(\alpha)$ and $\widehat{e}(\alpha^{-1}) = \widehat{s}(\alpha)$. By $Q_1^{-1}$, we mean the set of formal inverses of all arrows in $Q_1$.
	
	We call a word $w=\alpha_1\ldots\alpha_n$, where each symbol $\alpha_i \in Q_1 \cup Q^{-1}_1$ a \emph{string of length} $n$ if $w$ avoids the relations in $\rho$ and for each $i$, $\alpha_i \neq \alpha^{-1}_{i+1}$ and $e(\alpha_i)=s(\alpha_{i+1})$. We define $|w|=n$. We allow for strings of length zero (which can be considered as the stationary paths $\varepsilon_x$ for each $x \in Q_0$), which we call \emph{zero strings}. We say a string $w$ is a \emph{direct string} if every symbol of $w$ is in $Q_1$ and we say $w$ is an \emph{inverse string} if every symbol of $w$ is in $Q^{-1}_1$. A zero string is defined to be both direct and inverse. A \emph{band} is a cyclic string $b$ such that $b^m$ is a string, but $b$ is not a proper power of any string $w$. For any non-zero string $w=\alpha_1\ldots\alpha_n$, we define $s(w)=s(\alpha_1)$ and $e(w)=e(\alpha_n)$ and $\widehat{s}(w)=\widehat{s}(\alpha_1)$ and $\widehat{e}(w)=\widehat{e}(\alpha_n)$. If $w=\varepsilon_x$ is a zero string, then we let $s(w)=x=e(w)$, but we do not define $\widehat{s}(w)$ and $\widehat{e}(w)$, since this is not possible for zero strings. One must therefore take care when using the functions $\widehat{s}$ and $\widehat{e}$ with strings.
	
	Let $w=\alpha_1\ldots\alpha_n$ be a string, let $x_0=s(\alpha_1)$ and for each $i$, let $x_i=e(\alpha_i)$. From the string $w$, we obtain an indecomposable module $M(w) \in \Mod*A$ called a \emph{string module}. The underlying vector space of $M(w)$ is given by replacing each $x_i$ with a copy of the field $K$. We then say that the action of an arrow $\alpha \in Q_1$ is induced by the relevant identity maps if $\alpha$ or its formal inverse is in $w$, and is zero otherwise. It follows from the construction of string modules that $M(\varepsilon_x)=S(x)$.
	
	To each band $b=\beta_1\ldots\beta_m$, we obtain an infinite family of indecomposable modules $M(b, n, \phi)$ called \emph{band modules}, where $n \in \mathbb{Z}_{>0}$ and $\phi \in \mathrm{Aut}(K^n)$. We direct the reader to \cite{butlerRingel} for the full details on the construction of $M(b, n, \phi)$, however we will provide a brief summary here. The underlying vector space of $M(b, n, \phi)$ is given by replacing each vertex of $b$ with a copy of $K^n$. The action of an arrow in $\gamma \in Q_1$ on $M(b, n, \phi)$ is given by the relevant identity morphism if $\gamma = \beta_i$ or $\gamma = \beta^{-1}_i$ for some $i \neq m$. If we instead have $\gamma = \beta_m$ or $\gamma = \beta^{-1}_m$, then the action of $\gamma$ on $M(b, n, \phi)$ is $\phi$. Otherwise, $\gamma$ has a zero action on $M(b, n, \phi)$.
	
	\section{Bands in Symmetric Special Multiserial Algebras} \label{Examples}
	In this section, we will be utilising the contrapositive of a theorem of Crawley-Boevey to show an algebra is wild. We state the theorem here for convenience, where by almost all, we mean all but finitely many.
	\begin{thm}[\cite{Boevey},Theorem D] \label{tameTube}
		Let $K$ be an algebraically closed field. If $A$ is a tame $K$-algebra, then for each dimension $d$, $M \cong \tau M$ for almost all indecomposable $A$-modules of dimension $d$.
	\end{thm}
	
	We aim to construct one-parameter families of indecomposable modules that have a cyclic presentation. A well known example of such families of modules is the class of band modules. The motivation for constructing these one-parameter families of modules is that the dimension vector of the modules in the family is independent of the parameter. Thus, if the family of $A$-modules $M_\lambda$ is such a one-parameter family, where $\lambda \in K$ then the algebra $A$ is wild if $\dimbf(M_\lambda) \neq \dimbf(\tau M_\lambda)$ for (almost) all $\lambda$, since then $M_\lambda \not\cong \tau M_\lambda$ for (almost) all $\lambda$. Note that since the algebras we are interested in are symmetric, we have $\tau = \Omega^2$, where $\Omega^2 M_\lambda$ is the second syzygy of $M_\lambda$.
	
	\begin{lem} \label{PathToString}
		Let $A=KQ/I$ be a Brauer configuration algebra associated to a Brauer configuration $\chi$. Suppose there exists a (not necessarily simple) path
		\begin{equation*}
			p\colon \xymatrix@1{v_0 \ar@{-}[r]^-{x_1} & v_1 \ar@{-}[r]^-{x_2} & v_2 \ar@{-}[r] & \cdots \ar@{-}[r] & v_{n-1} \ar@{-}[r]^-{x_n} & v_n}.
		\end{equation*}
		in $\chi$ such that for each $i$, $x^{v_i}_i=x^{v_i}_{i+1}$ only if $\mathfrak{e}_{v_i}>1$. Then there exists a string of the form
		\begin{equation*}
			w= w^+_1w^-_2\ldots w^{\sigma(n-1)}_{n-1},
		\end{equation*}
		where each $w^+_i$ is a non-zero direct string such that $\widehat{s}(w^+_i)=x^{v_i}_i$ and $\widehat{e}(w^+_i)=x^{v_i}_{i+1}$, each $w^-_i$ is a non-zero inverse string such that $\widehat{s}(w^-_i)=x^{v_i}_i$ and $\widehat{e}(w^-_i)=x^{v_i}_{i+1}$, and $\sigma\colon\mathbb{N} \rightarrow \{+,-\}$ is defined by
			\begin{equation*}
				\sigma(n)=
				\begin{cases}
					+, &	\text{if } n \text{ is odd,}	\\
					-, &	\text{if } n \text{ is even.}
				\end{cases}
			\end{equation*}
	\end{lem}
	\begin{proof}
		We first aim to show that the strings $w^\pm_i$ exist for all $i<n$. Since $\val(u)=1$ for any truncated vertex $u$ and we have $x^{v_i}_i=x^{v_i}_{i+1} \Rightarrow \mathfrak{e}_{v_i}>1$, it follows that all vertices $v_i$ such that $i<n$ are non-truncated. Thus for each $i<n$, there exists a unique arrow $\alpha_i \in Q_1$ such that $\widehat{s}(\alpha_i)=x^{v_i}_i$. Let $\mathfrak{C}_{v_i,\beta_1}=\beta_1\ldots \beta_m$, where $\widehat{s}(\beta_1)=x^{v_i}_i$. In the case that $x^{v_i}_i=x^{v_i}_{i+1}$, we have $\mathfrak{e}_{v_i}>1$ and therefore there exist strings $w^+_i=\beta_1\ldots \beta_m$ and $w^-_i=\beta^{-1}_m\ldots \beta^{-1}_1$, as required. If we otherwise have $x^{v_i}_i \neq x^{v_i}_{i+1}$, then there exists an integer $k$ such that $\widehat{e}(\beta_k)=x^{v_i}_{i+1}$. Then we have $w^+_i=\beta_1\ldots \beta_k$ and $w^-_i=\beta^{-1}_n\ldots \beta^{-1}_{k+1}$, as required.
		
		The next step is to show that strings $w^+_i w^-_{i+1}$ exist for all $i<n-1$. Let $\gamma$ be the last symbol of $w^+_i$ and let $\delta$ be the first symbol of $w^-_{i+1}$. Since both $w^+_i$ and $w^-_{i+1}$ are strings, and $\gamma \in Q_1$ and $\delta \in Q^{-1}_1$, the word $w^+_i w^-_{i+1}$ avoids the relations in $I$. It is also clear that $e(\gamma)=x_{i+1}=s(\delta)$ from the constructions of $w^\pm_i$ above. Furthermore, $\widehat{s}(\delta)=\widehat{e}(\delta^{-1}) =x_{i+1}^{v_{i+1}} \neq x_{i+1}^{v_i}=\widehat{e}(\gamma)$. So $\gamma \neq \delta^{-1}$. Thus, $w^+_i w^-_{i+1}$ is a string for all $i<n-1$, and by a similar argument, so is $w^-_i w^+_{i+1}$. By iteratively concatenating strings, the result follows.
	\end{proof}
	
	For the following results, $\phi_\lambda\colon K\rightarrow K$ is a linear map defined by $\phi_\lambda(a)=\lambda a$.
	
	\begin{lem} \label{SelfFoldedBand}
		Let $A=KQ/I$ be a Brauer configuration algebra associated to a Brauer configuration $\chi$. Suppose there exists a polygon $x$ in $\chi$ that is self-folded at a vertex $v$. Then there exists a band $b$ such that $S(x) \in \tp M(b,1,\phi_\lambda)$ for all $\lambda \in K^\ast$.
	\end{lem}
	\begin{proof}
		Let $x^{v,1}$ and $x^{v,2}$ be two germs of $x$ incident to $v$. The cycle $c=(v, x^{v,1}, x^{v,2}, v)$ induces a band $b$ by Lemma~\ref{PathToString}, which follows by extending $c$ to the path
		\begin{equation*}
			p=(v, x^{v,1}, x^{v,2}, v, x^{v,1}, x^{v,2}, v, x^{v,1}, x^{v,2}, v).
		\end{equation*}
		By the construction in Lemma~\ref{PathToString}, the first symbol $\alpha$ of $b$ is an arrow and the last symbol $\beta$ of $b$ is a formal inverse such that $s(\alpha)=x=e(\beta)$, which implies $S(x) \in \tp M(b,1,\phi_\lambda)$ for all $\lambda \in K^\ast$, as required.
	\end{proof}
	
	\begin{lem} \label{TriStringLoop}
		Let $A=KQ/I$ be a Brauer configuration algebra associated to a Brauer configuration $\chi$. Suppose $\chi$ contains a polygon $x$ such that $x$ is not self-folded and $|x|>2$. Let $u$ be a vertex connected to $x$ and suppose there exists a subconfiguration $\chi'$ of $\chi$ connected to $u$ but not containing $x$. If $\chi'$ contains a simple cycle or a vertex $v$ such that $\mathfrak{e}_v>1$, then there exists a non-zero string $w=\alpha_1\ldots\alpha_n$ such that $\alpha_1,\alpha_n \in Q_1$ and $\widehat{s}(\alpha_1)=x^u=\widehat{e}(\alpha_n)$.
	\end{lem}
	\begin{proof}
		We will first examine the case where $\chi'$ contains a vertex $v$ such that $\mathfrak{e}_v>1$. If $v=u$ then let $\mathfrak{C}_{u,\beta_1}=\beta_1\ldots\beta_r$, where $s(\beta_1)=x$. The required string is then $\beta_1\ldots\beta_r$. Suppose instead that $v \neq u$ and let $u'\neq u$ be some other vertex connected to $x$. Since $\chi'$ is connected to $u$, there exists a simple path
		\begin{equation*}
			p\colon \xymatrix@1{u' \ar@{-}[r]^-{x} & u=u_0 \ar@{-}[r]^-{y_1} & u_1 \ar@{-}[r]^-{y_2} & \cdots \ar@{-}[r]^-{y_m} & u_m=v},
		\end{equation*}
		which we can extend to a (non-simple) path
		\begin{equation*}
			p'\colon \xymatrix@1{u' \ar@{-}[r]^-{x} & u_0 \ar@{-}[r]^-{y_1} & u_1 \ar@{-}[r]^-{y_2} & \cdots \ar@{-}[r]^-{y_m} & u_m \ar@{-}[r]^-{y_m} & \cdots \ar@{-}[r]^-{y_2} & u_1 \ar@{-}[r]^-{y_1} & u_0 \ar@{-}[r]^-{x} & u'},
		\end{equation*}
		of even length. The result then follows by Lemma~\ref{PathToString}.
		
		Now consider the case where $\chi'$ contains a simple cycle. Again let $u'\neq u$ be some other vertex connected to $x$. There exists a simple path
		\begin{equation*}
			p\colon \xymatrix@1{u' \ar@{-}[r]^-{x} & u=u_0 \ar@{-}[r]^-{y_1} & u_1 \ar@{-}[r]^-{y_2} & \cdots \ar@{-}[r]^-{y_m} & u_m},
		\end{equation*}
		where $u_m$ is a vertex connected to a simple cycle in $\chi'$. We assume $p$ is minimal in the sense that no vertex $u_i$ with $i<m$ is connected to a simple cycle in $\chi'$ (otherwise we may simply choose a shorter path). Let
		\begin{equation*}
			c\colon \xymatrix@1{u_m=v_0 \ar@{-}[r]^-{z_1} & v_1 \ar@{-}[r] & \cdots \ar@{-}[r] & v_{r-1} \ar@{-}[r]^-{z_r} & v_0=u_m},
		\end{equation*}
		be the simple cycle connected to $u_m$. If $c$ is of even length, then we may concatenate paths to extend $p$ to an even length path of the form
		\begin{equation*}
			p'\colon \xymatrix@1{u' \ar@{..}[r] & p \ar@{..}[r] & v_0 \ar@{..}[r] & c \ar@{..}[r] & v_0 \ar@{..}[r] & p^{-1} \ar@{..}[r] & u'}.
		\end{equation*}
		If $c$ is otherwise of odd length, then we may extend $p$ to an even length path of the form
		\begin{equation*}
			p''\colon \xymatrix@1{u' \ar@{..}[r] & p \ar@{..}[r] & v_0 \ar@{..}[r] & c \ar@{..}[r] & v_0 \ar@{..}[r] & c \ar@{..}[r] & v_0 \ar@{..}[r] & p^{-1} \ar@{..}[r] & u'}.
		\end{equation*}
		The result then follows by Lemma~\ref{PathToString}.
	\end{proof}
	
	We may now use the lemmata above to prove the following result, which will be a particularly useful tool for reducing the number of cases in proofs for later results. We provide examples of Brauer configuration algebras satisfying the proposition statement immediately following the proof.
	\begin{prop} \label{CrossBand}
		Let $A$ be a Brauer configuration algebra associated to a Brauer configuration $\chi$ and suppose there exists an $n$-gon $x$ in $\chi$ with $n>2$. Suppose either $x$ is self-folded or $x$ is locally of the form
			
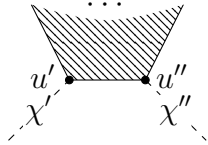
\begin{center}
				\begin{tikzpicture}
					\draw[pattern=north west lines] (0,0) -- (0.5,-1) -- (1.5,-1) -- (2,0);
					\draw[white, fill=white]  plot[smooth, tension=.7] coordinates {(0,0.01) (1,-0.2) (2,0.01)};
					\draw (1,0) node {$\cdots$};
					\draw [dashed](-0.3,-1.8) -- (0.5,-1);
					\draw [dashed](1.5,-1) -- (2.3,-1.8);
					\draw [white,fill=white] (-0.1,-1.2) rectangle (0.3,-1.7);
					\draw [white,fill=white] (1.7,-1.2) rectangle (2.1,-1.7);
					\draw (0.1,-1.4) node {$\chi'$};
					\draw (1.9,-1.4) node {$\chi''$};
					\draw [black,fill=black] (0.5,-1) ellipse (0.05 and 0.05);
					\draw [black,fill=black] (1.5,-1) ellipse (0.05 and 0.05);
					\draw (0.15,-0.95) node {$u'$};
					\draw (1.85,-0.95) node {$u''$};
				\end{tikzpicture}
			\end{center}
			where $\chi'$ and $\chi''$ are subconfigurations of $\chi$ such that both of $\chi'$ and $\chi''$ contain a simple cycle or a vertex of multiplicity strictly greater than 1. Then $A$ is wild.
	\end{prop}
	\begin{proof}
		If $x$ is self-folded, then Lemma~\ref{SelfFoldedBand} states that there exists a band $b$ such that $S(x) \in \tp M(b,1,\phi_\lambda)$ for all $\lambda \in K^\ast$. We will prove the same is true for the other cases. So suppose $x$ is not self-folded and suppose $\chi'$ and $\chi''$ are not disjoint. Then necessarily, $x$ belongs to some simple cycle in $\chi$. One can then use Lemma~\ref{PathToString} to construct a band $b$ such that $S(x) \in \tp M(b,1,\phi_\lambda)$ for all $\lambda \in K^\ast$. Now suppose $x$ is not self-folded and $\chi'$ and $\chi''$ are disjoint, then by Lemma~\ref{TriStringLoop} there exist strings $w'=\alpha_1\ldots\alpha_m$ and $w''=\beta_1\ldots\beta_n$ such that $\alpha_1,\alpha_m,\beta_1,\beta_n \in Q_1$, $\widehat{s}(\alpha_1)= x^{u'}=\widehat{e}(\alpha_m)$ and $\widehat{s}(\beta_1)= x^{u''}=\widehat{e}(\beta_n)$. Thus, we can construct a band $b=(w')^{-1}w''$. Since $\alpha^{-1}_1\beta_1$ is a substring of $b$, we conclude that $S(x) \in \tp M(b,1,\phi_\lambda)$ for all $\lambda \in K^\ast$. Thus, if $\chi$ is of any of the forms given in the proposition statement, there exists a band $b$ such that $S(x) \in \tp M(b,1,\phi_\lambda)$ for all $\lambda \in K^\ast$.
		
		Let $b$ be a band as above and let $M_\lambda=M(b,1,\phi_\lambda)$. Fix a choice of $\lambda$ and let
		\begin{equation*}
			\tp M_\lambda = \bigoplus_{i=1}^{m} S(x_i) \quad \text{ and } \quad \soc M_\lambda = \bigoplus_{i=1}^{m} S(x'_i).
		\end{equation*}
		Then we note that $M_\lambda$ has the following module structure.
		
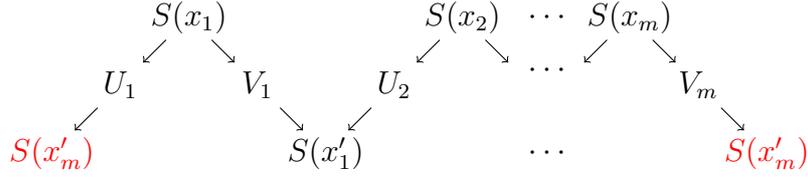
\begin{center}
			\begin{tikzpicture}
			\draw [red](0,0) node {$S(x'_m)$};
			\draw [->](0.6,0.6) -- (0.3,0.3);
			\draw (0.9,0.9) node {$U_1$};
			\draw [->](1.5,1.5) -- (1.2,1.2);
			\draw (1.8,1.8) node {$S(x_1)$};
			\draw [->](2.1,1.5) -- (2.4,1.2);
			\draw (2.7,0.9) node {$V_1$};
			\draw [->](3,0.6) -- (3.3,0.3);
			\draw (3.6,0) node {$S(x'_1)$};
			\draw [->](4.2,0.6) -- (3.9,0.3);
			\draw (4.5,0.9) node {$U_2$};
			\draw [->](5.1,1.5) -- (4.8,1.2);
			\draw (5.4,1.8) node {$S(x_2)$};
			\draw [->](5.7,1.5) -- (6,1.2);
			\draw (6.54,1.8) node {$\cdots$};
			\draw (6.54,1.1) node {$\cdots$};
			\draw (6.54,0) node {$\cdots$};
			\draw [->](7.3,1.5) -- (7,1.2);
			\draw (7.6,1.8) node {$S(x_m)$};
			\draw [->](7.9,1.5) -- (8.2,1.2);
			\draw (8.5,0.9) node {$V_m$};
			\draw [->](8.8,0.6) -- (9.1,0.3);
			\draw [red](9.4,0) node {$S(x'_m)$};
			
			\end{tikzpicture}
		\end{center}
		where the two copies of $S(x'_m)$ are identified and $U_i$ and $V_i$ are uniserial modules. Since $M_\lambda$ is a band module, none of the modules $P(x_i)$ or $P(x'_j)$ are uniserial (otherwise, $b$ would not be a cyclic string). Suppose $x_i$ and $x'_j$ are such that $P(x_i)$ and $P(x'_j)$ are biserial respectively. Then the structure of $P(x_i)$ and $P(x'_j)$ is
		
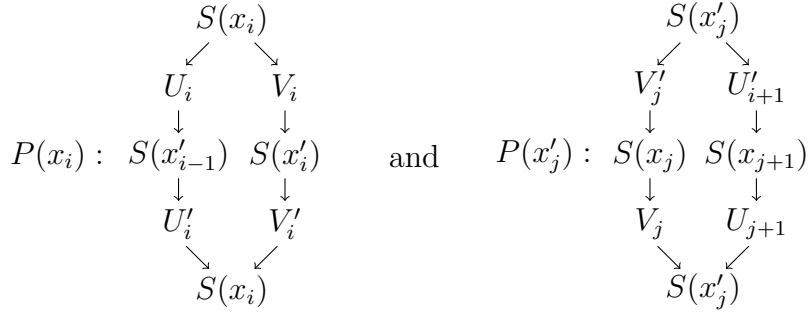
\begin{center}
			\begin{tikzpicture}
			\draw (-8.8,-1.8) node {$P(x_i):$};
			\draw (-6.5,0) node {$S(x_i)$};
			\draw [->] (-6.8,-0.3) -- (-7.1,-0.6);
			\draw [->](-6.2,-0.3) -- (-5.9,-0.6);
			\draw (-7.2,-0.9) node {$U_i$};
			\draw (-5.8,-0.9) node {$V_i$};
			\draw [->](-7.2,-1.2) -- (-7.2,-1.5);
			\draw [->](-5.8,-1.2) -- (-5.8,-1.5);
			\draw (-7.2,-1.8) node {$S(x'_{i-1})$};
			\draw (-5.8,-1.8) node {$S(x'_i)$};
			\draw [->](-7.2,-2.1) -- (-7.2,-2.4);
			\draw [->](-5.8,-2.1) -- (-5.8,-2.4);
			\draw (-7.2,-2.7) node {$U'_i$};
			\draw (-5.8,-2.7) node {$V'_i$};
			\draw [->](-7.1,-3) -- (-6.8,-3.3);
			\draw [->](-5.9,-3) -- (-6.2,-3.3);
			\draw (-6.5,-3.6) node {$S(x_i)$};
			
			\draw (-4.1,-1.8) node {and};
			
			\draw (-2.4,-1.8) node {$P(x'_j):$};
			\draw (-0.3,0) node {$S(x'_j)$};
			\draw [->] (-0.6,-0.3) -- (-0.9,-0.6);
			\draw [->](0,-0.3) -- (0.3,-0.6);
			\draw (-1,-0.9) node {$V'_j$};
			\draw (0.4,-0.9) node {$U'_{i+1}$};
			\draw [->](-1,-1.2) -- (-1,-1.5);
			\draw [->](0.4,-1.2) -- (0.4,-1.5);
			\draw (-1,-1.8) node {$S(x_j)$};
			\draw (0.4,-1.8) node {$S(x_{j+1})$};
			\draw [->](-1,-2.1) -- (-1,-2.4);
			\draw [->](0.4,-2.1) -- (0.4,-2.4);
			\draw (-1,-2.7) node {$V_j$};
			\draw (0.4,-2.7) node {$U_{j+1}$};
			\draw [->](-0.9,-3) -- (-0.6,-3.3);
			\draw [->](0.3,-3) -- (0,-3.3);
			\draw (-0.3,-3.6) node {$S(x'_j)$};
			\end{tikzpicture}
		\end{center}
		respectively, where $U'_i$ and $V'_i$ are uniserial modules.
		
		Suppose there exists an integer $r$ such that $S(x_r) \cong S(x)$. Then the structure of $P(x_r)$ is of the form
		
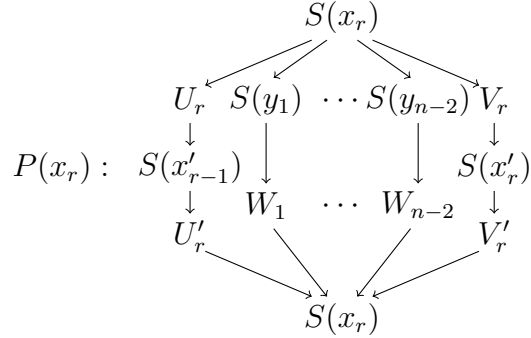
\begin{center}
			\begin{tikzpicture}
			\draw (-2.4,-1.8) node {$P(x_r):$};
			\draw (1.3,0.2) node {$S(x_r)$};
			\draw [->] (0.9,-0.1) -- (-0.5,-0.7);
			\draw [->](1.7,-0.1) -- (3.1,-0.7);
			\draw (-0.7,-0.9) node {$U_r$};
			\draw (3.3,-0.9) node {$V_r$};
			\draw [->](-0.7,-1.2) -- (-0.7,-1.5);
			\draw [->](3.3,-1.2) -- (3.3,-1.5);
			\draw (-0.7,-1.8) node {$S(x'_{r-1})$};
			\draw (3.3,-1.8) node {$S(x'_r)$};
			\draw [->](-0.7,-2.1) -- (-0.7,-2.4);
			\draw [->](3.3,-2.1) -- (3.3,-2.4);
			\draw (-0.7,-2.7) node {$U'_r$};
			\draw (3.3,-2.7) node {$V'_r$};
			\draw [->](-0.5,-2.9) -- (0.9,-3.5);
			\draw [->](3.1,-2.9) -- (1.7,-3.5);
			\draw (1.3,-3.8) node {$S(x_r)$};
			
			\draw (0.3,-0.9) node {$S(y_1)$};
			\draw (1.31,-0.9) node {$\cdots$};
			\draw (2.3,-0.9) node {$S(y_{n-2})$};
			
			\draw (0.3,-2.3) node {$W_1$};
			\draw (1.31,-2.3) node {$\cdots$};
			\draw (2.3,-2.3) node {$W_{n-2}$};
			\draw [->](0.3,-1.2) -- (0.3,-2);
			\draw [->](2.3,-1.2) -- (2.3,-2);
			\draw [->](0.4,-2.6) -- (1.1,-3.5);
			\draw [->](2.2,-2.6) -- (1.5,-3.5);
			\draw [->](1.1,-0.1) -- (0.4,-0.6);
			\draw [->](1.5,-0.1) -- (2.2,-0.6);
			\end{tikzpicture}
		\end{center}
		where $U_r$, $U'_r$, $V_r$, $V'_r$, $W_1,\ldots, W_{n-2}$ are uniserial modules and $y_1, \ldots, y_{n-2}$ are the successors to $x$ at the other vertices connected to $x$. It follows that 
		\begin{equation*}
			\tp \Omega(M_\lambda) \cong \bigoplus_{i=1}^{m} S(x'_i) \oplus \bigoplus_{i=1}^{n-2} (S(y_i))^t,
		\end{equation*}
		where $t$ is the number of direct summands in $\tp M_\lambda$ that are isomorphic to $S(x)$.
		
		Define the following non-negative integers.
		\begin{align*}
			u &= \sum_{i=1}^m \dim (U_i \varepsilon_x),			&	u' &= \sum_{i=1}^m \dim (U'_i  \varepsilon_x),		\\
			v &= \sum_{i=1}^m \dim (V_i  \varepsilon_x),			&	v' &= \sum_{i=1}^m \dim (V'_i  \varepsilon_x),		\\
			k &= \sum_{i=1}^{n-2} \dim (S(y_i) \varepsilon_x),	&	w &= \sum_{i=1}^{n-2} \dim (W_i  \varepsilon_x),	\\
			t &= \dim ((\tp M_\lambda)\varepsilon_x)	,			&	s &= \dim ((\soc M_\lambda)\varepsilon_x).
		\end{align*}
		Then we have
		\begin{align*}
			\dim (M_\lambda \varepsilon_x) &= t + u + v + s, \\
			\dim (\Omega(M_\lambda) \varepsilon_x) &= \sum_{i=1}^m \dim (P(x_i) \varepsilon_x) - \dim (M_\lambda \varepsilon_x) \\
			&= 2t + u + v + 2s + u' + v' + tk + tw - \dim (M_\lambda \varepsilon_x) \\
			&= t + s + u' + v' + tk + tw, \\
			\dim (\Omega^2(M_\lambda) \varepsilon_x) &= \sum_{i=1}^m \dim (P(x'_i) \varepsilon_x) + t \sum_{i=1}^{n-2} \dim (P(y_i) \varepsilon_x) - \dim (\Omega(M_\lambda) \varepsilon_x).
		\end{align*}
		Now
		\begin{align*}
			\sum_{i=1}^m \dim(P(x'_i) \varepsilon_x) &= 2s + u' + v' + 2t + u + v + sk + sw \text{ and} \\
			\sum_{i=1}^{n-2} \dim(P(y_i) \varepsilon_x) &\geq k + w + n-2.
		\end{align*}
		So
		\begin{align*}
			\dim (\Omega^2(M_\lambda) \varepsilon_x) &\geq t + u + v + s + sk + sw + t(n-2) \\
			&\geq \dim (M_\lambda \varepsilon_x) + sk + sw + t(n-2)
			> \dim (M_\lambda \varepsilon_x)
		\end{align*}
		if $t > 0$. Thus, if $S(x)$ is a direct summand of $\tp M_\lambda$, then $M_\lambda \not\cong \Omega^2(M_\lambda) = \tau M_\lambda$. Since $M_\lambda$ describes an infinite family of non-isomorphic indecomposable modules, we conclude in this case that the algebra $A$ must be wild by the contrapositive of Theorem~\ref{tameTube}.
	\end{proof}
	
	\begin{exam}
		Let $A_1$, $A_2$, $A_3$ and $A_4$ be Brauer configuration algebras associated to the following respective Brauer configurations.
		
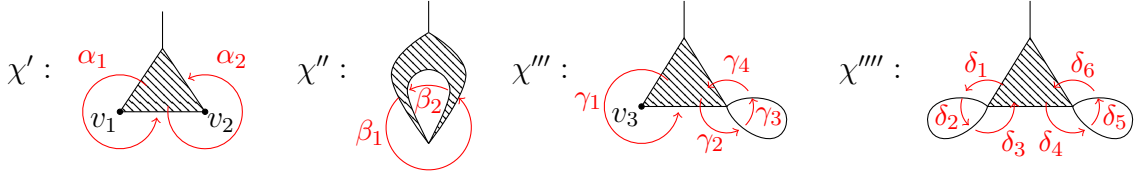
\begin{center}
			\begin{tikzpicture}[scale=0.7]
				\draw (-1.5,4) node{$\chi':$};
				\draw[pattern = north west lines] (0.2,3.2) -- (1.8,3.2) -- (1,4.4) -- (0.2,3.2);
				\draw (1,5.1) -- (1,4.4);
				\draw(-0.1,3) node{$v_1$};
				\draw (2.1,3) node{$v_2$};
			
				\draw (4,4) node{$\chi'':$};
				\draw[pattern = north west lines]  plot[smooth, tension=.7] coordinates {(6,4.7) (5.3,3.9) (6,2.6)};
				\draw[pattern = north west lines]  plot[smooth, tension=.7] coordinates {(6,4.7) (6.7,3.9) (6,2.6)};
				\draw[fill=white]  plot[smooth, tension=.7] coordinates {(6,2.6) (5.6,3.6) (6,4) (6.4,3.6) (6,2.6)};
				\draw (6,4.7) -- (6,5.3);
			
				\draw (8.1,4) node{$\chi''':$};
				\draw[pattern = north west lines] (10,3.3) -- (11.6,3.3) -- (10.8,4.6) -- (10,3.3);
				\draw (10.8,5.3) -- (10.8,4.6);
				\draw(9.7,3.1) node{$v_3$};
				\draw (11.6,3.3) .. controls (11.7,2.9) and (12.5,2.4) .. (12.7,2.9) .. controls (12.9,3.5) and (11.9,3.6) .. (11.6,3.3);
			
				\draw (14.3,4) node{$\chi'''':$};
				\draw[pattern = north west lines] (16.5,3.3) -- (18.1,3.3) -- (17.3,4.6) -- (16.5,3.3);
				\draw (17.3,5.3) -- (17.3,4.6);
				\draw (18.1,3.3) .. controls (18.2,2.9) and (19,2.4) .. (19.2,2.9) .. controls (19.4,3.5) and (18.4,3.6) .. (18.1,3.3);
				\draw (16.5,3.3) .. controls (16.4,2.9) and (15.6,2.4) .. (15.4,2.9) .. controls (15.2,3.5) and (16.2,3.6) .. (16.5,3.3);
			\draw [->, red](0.695,3.695) arc (45:351:0.7);
			\draw [->, red](1.1086,3.3095) arc (171.0006:477:0.7);
			\draw [->, red](5.4343,3.4657) arc (135:405:0.8);
			\draw [->, red](6.3632,3.6128) arc (62.9993:117:0.8);
			\draw [->, red](10.495,3.795) arc (45:351:0.7);
			\draw [->, red](11.1245,3.4545) arc (161.9999:297:0.5);
			\draw [->, red](11.9536,2.9464) arc (-45:18:0.5001);
			\draw [->, red](12.0045,3.5939) arc (36.0013:135:0.5);
			\draw [->, red](16.7939,3.7045) arc (53.9987:144:0.5);
			\draw [->, red](16.0245,3.4545) arc (161.9999:225:0.5);
			\draw [->, red](16.273,2.8545) arc (-117.0007:9:0.5);
			\draw [->, red](17.6062,3.3782) arc (171.0012:297:0.5);
			\draw [->, red](18.4536,2.9464) arc (-45:18:0.5001);
			\draw [->, red](18.5045,3.5939) arc (36.0013:135:0.5);
			
			\draw[red] (-0.3,4.2) node{\small$\alpha_1$};
			\draw[red] (2.3,4.2) node{\small$\alpha_2$};
			\draw[red] (4.9,2.8) node{\small$\beta_1$};
			\draw[red] (6,3.4001) node{\small$\beta_2$};
			\draw[red] (9,3.4) node{\small$\gamma_1$};
			\draw[red] (11.3,2.6) node{\small$\gamma_2$};
			\draw[red] (12.4,3.1) node{\small$\gamma_3$};
			\draw[red] (11.8,4.1) node{\small$\gamma_4$};
			\draw[red] (16.3,4.1) node{\small$\delta_1$};
			\draw[red] (15.7428,3.1) node{\small$\delta_2$};
			\draw[red] (17,2.6) node{\small$\delta_3$};
			\draw[red] (17.7,2.6) node{\small$\delta_4$};
			\draw[red] (18.9,3.1) node{\small$\delta_5$};
			\draw[red] (18.3,4.1) node{\small$\delta_6$};
			\draw [fill=black] (0.2,3.2) ellipse (0.05 and 0.05);
			\draw [fill=black] (1.8,3.2) ellipse (0.05 and 0.05);
			\draw [fill=black] (10,3.3) ellipse (0.05 and 0.05);
			\end{tikzpicture}
		\end{center}
		where $\mathfrak{e}_{v_i}>1$ for each $i\in\{1,2,3\}$ and all other vertices have multiplicity 1. Every Brauer configuration $\chi^{(i)}$ satisfies Proposition~\ref{CrossBand}, and thus each respective Brauer configuration algebra $A_i$ is wild. The following are examples of the bands used in the proof of Proposition~\ref{CrossBand}. For $A_1$, we have $\alpha^{-1}_1\alpha_2$. For $A_2$, we have $\beta_1\beta^{-1}_2$. For $A_3$, we have $\gamma_1^{-1}\gamma_2\gamma_3^{-1}\gamma_4$. For $A_4$, we have $\delta_3^{-1}\delta_2\delta_1^{-1}\delta_4\delta^{-1}_5\delta_6$.
	\end{exam}
	
	\section{Symmetric Special $n$-Serial Algebras ($n>3$)} \label{nSerialSection}
	We will again be using the contrapositive of 	Crawley-Boevey's Theorem~\ref{tameTube} in this section to prove particular classes of algebras are wild. The following are examples of one-parameter families of modules that are not band modules, which we require for the results of this section.
	
	\begin{exam} \label{QuadBand1}
		Let $A$ be a Brauer configuration algebra associated to a Brauer configuration $\chi$. Suppose $\chi$ contains a polygon $x$ with $|x|>3$ and let $\alpha_1, \ldots, \alpha_4 \in Q_1$ be distinct arrows of source $x$. Consider the following diagram.
		\begin{center}
			\begin{tikzpicture}
			\draw (-1.4,-0.44) node {$B:$};
			\draw (0,0) node {$x$};
			\draw (1.6,0) node {$x$};
			\draw [->](-0.2,-0.2) -- (-0.7,-0.7);
			\draw (-0.5,-0.2) node {\footnotesize$\alpha_1$};
			\draw [->](0,-0.2) -- (0,-0.7);
			\draw (-0.22,-0.5) node {\footnotesize$\alpha_2$};
			\draw [->](0.2,-0.2) -- (0.7,-0.7);
			\draw (0.56,-0.24) node {\footnotesize$\alpha_3$};
			\draw [->](1.4,-0.2) -- (0.9,-0.7);
			\draw (1.04,-0.24) node {\footnotesize$\alpha_3$};
			\draw [->](1.6,-0.2) -- (1.6,-0.7);
			\draw (1.4,-0.5) node {\footnotesize$\alpha_4$};
			\draw [->](1.8,-0.2) -- (2.3,-0.7);
			\draw (2.15,-0.3) node {\footnotesize$\alpha_1$};
			\draw [red](-0.8,-0.95) node {$y_1$};
			\draw (0,-0.95) node {$y_2$};
			\draw (0.8,-0.95) node {$y_3$};
			\draw (1.6,-0.95) node {$y_4$};
			\draw [red](2.4,-0.95) node {$y_1$};
			\end{tikzpicture}
		\end{center}
		where each $y_i$ is the target of arrow $\alpha_i$ and where the two copies of $y_1$ are identified. The structure of the module $M_\lambda$ is as follows. Define a family of modules $M_\lambda$ from $B$ as follows. The underlying vector space of $M_\lambda$ is given by replacing each vertex of $B$ with a copy of $K$. The action of an arrow $\beta \in Q_1$ on $M_\lambda$ is the zero action if $\beta \neq \alpha_i$ for any $i$. The action of $\alpha_i$ on $M_\lambda$ is given by the relevant identity morphisms if $i \neq 1$. The action of $\alpha_1$ on $M_\lambda$ is given by an identity morphism for the leftmost copy in $B$ and by $\phi_\lambda$ for the rightmost copy.
		
		We will assume $x$, $y_1$, $y_2$, $y_3$ and $y_4$ are pairwise distinct and calculate  the space $\Hom_A(M_{\lambda_1}, M_{\lambda_2})$. Let $K(x)$ and $K(y_i)$ denote the underlying $K$-vector spaces of $S(x)$ and $S(y_i)$ respectively. Then we have the following commutative squares.
		\begin{center}
			\begin{tikzpicture}
			\draw (-5.7,2.1) node{(i)};
			\draw (-4.5,3) node{$(K(x))^2$};
			\draw (-4.5,1.2) node{$(K(x))^2$};
			\draw (-2.2,3) node{$K(y_1)$};
			\draw (-2.2,1.2) node{$K(y_1)$};
			\draw [->](-3.8,3) -- (-2.8,3);
			\draw [->](-3.8,1.2) -- (-2.8,1.2);
			\draw [->](-4.5,2.7) -- (-4.5,1.5);
			\draw [->](-2.2,2.7) -- (-2.2,1.5);
			\draw (-3.3,3.3) node{\footnotesize$\left(\begin{smallmatrix} 1 & \lambda_1 \end{smallmatrix}\right)$};
			\draw (-3.3,1.5) node{\footnotesize$\left(\begin{smallmatrix} 1 & \lambda_2 \end{smallmatrix}\right)$};
			\draw (-4.2,2.1) node{\footnotesize$\varphi_x$};
			\draw (-1.83,2.1) node{\footnotesize$\varphi_{y_1}$};
			
			\draw (0.3,2.1) node{(ii)};
			\draw (1.5,3) node{$(K(x))^2$};
			\draw (1.5,1.2) node{$(K(x))^2$};
			\draw (3.8,3) node{$K(y_2)$};
			\draw (3.8,1.2) node{$K(y_2)$};
			\draw [->](2.2,3) -- (3.2,3);
			\draw [->](2.2,1.2) -- (3.2,1.2);
			\draw [->](1.5,2.7) -- (1.5,1.5);
			\draw [->](3.8,2.7) -- (3.8,1.5);
			\draw (2.7,3.3) node{\footnotesize$\left(\begin{smallmatrix} 1 & 0 \end{smallmatrix}\right)$};
			\draw (2.7,1.5) node{\footnotesize$\left(\begin{smallmatrix} 1 & 0 \end{smallmatrix}\right)$};
			\draw (1.8,2.1) node{\footnotesize$\varphi_x$};
			\draw (4.17,2.1) node{\footnotesize$\varphi_{y_2}$};
			
			\draw (-5.8,-0.9) node{(iii)};
			\draw (-4.5,0) node{$(K(x))^2$};
			\draw (-4.5,-1.8) node{$(K(x))^2$};
			\draw (-2.2,0) node{$K(y_3)$};
			\draw (-2.2,-1.8) node{$K(y_3)$};
			\draw [->](-3.8,0) -- (-2.8,0);
			\draw [->](-3.8,-1.8) -- (-2.8,-1.8);
			\draw [->](-4.5,-0.3) -- (-4.5,-1.5);
			\draw [->](-2.2,-0.3) -- (-2.2,-1.5);
			\draw (-3.3,0.3) node{\footnotesize$\left(\begin{smallmatrix} 1 & 1 \end{smallmatrix}\right)$};
			\draw (-3.3,-1.5) node{\footnotesize$\left(\begin{smallmatrix} 1 & 1 \end{smallmatrix}\right)$};
			\draw (-4.2,-0.9) node{\footnotesize$\varphi_x$};
			\draw (-1.83,-0.9) node{\footnotesize$\varphi_{y_3}$};
			
			\draw (0.26,-0.9) node{(iv)};
			\draw (1.5,0) node{$(K(x))^2$};
			\draw (1.5,-1.8) node{$(K(x))^2$};
			\draw (3.8,0) node{$K(y_4)$};
			\draw (3.8,-1.8) node{$K(y_4)$};
			\draw [->](2.2,0) -- (3.2,0);
			\draw [->](2.2,-1.8) -- (3.2,-1.8);
			\draw [->](1.5,-0.3) -- (1.5,-1.5);
			\draw [->](3.8,-0.3) -- (3.8,-1.5);
			\draw (2.7,0.3) node{\footnotesize$\left(\begin{smallmatrix} 0 & 1 \end{smallmatrix}\right)$};
			\draw (2.7,-1.5) node{\footnotesize$\left(\begin{smallmatrix} 0 & 1 \end{smallmatrix}\right)$};
			\draw (1.8,-0.9) node{\footnotesize$\varphi_x$};
			\draw (4.17,-0.9) node{\footnotesize$\varphi_{y_4}$};
			\end{tikzpicture}
		\end{center}
		Squares (ii), (iii) and (iv) imply that $\varphi_x$ is a matrix
		\begin{equation*}
			\varphi_x =
			\begin{pmatrix}
				a	&	0	\\
				0	&	a
			\end{pmatrix},
		\end{equation*}
		where $a=\varphi_{y_2}(1)=\varphi_{y_3}(1)=\varphi_{y_4}(1)$. Square (i) implies that $\varphi_{y_1}(1)=a$ and $a\lambda_1 = a\lambda_2$. Thus, $a \in K$ if $\lambda_1 = \lambda_2$ and $a=0$ if $\lambda_1 \neq \lambda_2$. From this, we conclude firstly that $\dim \End_A(M_\lambda)=1$, and so $M_\lambda$ is indecomposable for all $\lambda \in K^{\ast}$; and secondly that $\dim \Hom_A(M_{\lambda_1}, M_{\lambda_2})=0$ for all $\lambda_1 \neq \lambda_2$. Thus, $M_{\lambda_1} \not\cong M_{\lambda_2}$ for all $\lambda_1 \neq \lambda_2$ and so $M_\lambda$ describes a 1-parameter family of indecomposable $A$-modules.
	\end{exam}
	
	\begin{exam} \label{QuadBand2}
		Let $B$ and $M_\lambda$ be as in Example~\ref{QuadBand1}, except we will now assume $x$, $y_2$, $y_3$ and $y_4$ are pairwise distinct and $y_1=x$. To calculate $\Hom_A(M_{\lambda_1}, M_{\lambda_2})$ we will need to consider the following commutative squares.
		\begin{center}
						\begin{tikzpicture}
			\draw (-5.9,1.9) node{(i)};
			\draw (-4.7,-0.1) node{$(K(x))^3$};
			\draw (-4.7,-1.9) node{$(K(x))^3$};
			\draw (-2,-0.1) node{$K(y_3)$};
			\draw (-2,-1.9) node{$K(y_3)$};
			\draw [->](-4,-0.1) -- (-2.6,-0.1);
			\draw [->](-4,-1.9) -- (-2.6,-1.9);
			\draw [->](-4.7,-0.4) -- (-4.7,-1.6);
			\draw [->](-2,-0.4) -- (-2,-1.6);
			\draw (-3.3,0.2) node{\footnotesize$\left(\begin{smallmatrix} 1 & 1 & 0 \end{smallmatrix}\right)$};
			\draw (-3.3,-1.6) node{\footnotesize$\left(\begin{smallmatrix} 1 & 1 & 0 \end{smallmatrix}\right)$};
			\draw (-4.4,-1) node{\footnotesize$\varphi_x$};
			\draw (-1.63,-1) node{\footnotesize$\varphi_{y_3}$};
			
			\draw (0.5,1.9) node{(ii)};
			\draw (1.7,2.8) node{$(K(x))^3$};
			\draw (1.7,1) node{$(K(x))^3$};
			\draw (4.4,2.8) node{$K(y_2)$};
			\draw (4.4,1) node{$K(y_2)$};
			\draw [->](2.4,2.8) -- (3.8,2.8);
			\draw [->](2.4,1) -- (3.8,1);
			\draw [->](1.7,2.5) -- (1.7,1.3);
			\draw [->](4.4,2.5) -- (4.4,1.3);
			\draw (3.1,3.1) node{\footnotesize$\left(\begin{smallmatrix} 1 & 0 & 0 \end{smallmatrix}\right)$};
			\draw (3.1,1.3) node{\footnotesize$\left(\begin{smallmatrix} 1 & 0 & 0 \end{smallmatrix}\right)$};
			\draw (2,1.9) node{\footnotesize$\varphi_x$};
			\draw (4.77,1.9) node{\footnotesize$\varphi_{y_2}$};
			
			\draw (-6,-1) node{(iii)};
			\draw (-4.7,2.8) node{$(K(x))^3$};
			\draw (-4.7,1) node{$(K(x))^3$};
			\draw (-2,2.8) node{$(K(x))^3$};
			\draw (-2,1) node{$(K(x))^3$};
			\draw [->](-4,2.8) -- (-2.7,2.8);
			\draw [->](-4,1) -- (-2.7,1);
			\draw [->](-4.7,2.5) -- (-4.7,1.3);
			\draw [->](-2,2.5) -- (-2,1.3);
			\draw (-3.35,3.2) node{\footnotesize$\left(\begin{smallmatrix} 0 & 0 & 0 \\ 0 & 0 & 0 \\  1 & \lambda_1 & 0 \end{smallmatrix}\right)$};
			\draw (-3.35,1.4) node{\footnotesize$\left(\begin{smallmatrix} 0 & 0 & 0 \\ 0 & 0 & 0 \\  1 & \lambda_2 & 0 \end{smallmatrix}\right)$};
			\draw (-4.4,1.9) node{\footnotesize$\varphi_x$};
			\draw (-1.7,1.9) node{\footnotesize$\varphi_x$};
			
			\draw (0.46,-1) node{(iv)};
			\draw (1.7,-0.1) node{$(K(x))^3$};
			\draw (1.7,-1.9) node{$(K(x))^3$};
			\draw (4.4,-0.1) node{$K(y_4)$};
			\draw (4.4,-1.9) node{$K(y_4)$};
			\draw [->](2.4,-0.1) -- (3.8,-0.1);
			\draw [->](2.4,-1.9) -- (3.8,-1.9);
			\draw [->](1.7,-0.4) -- (1.7,-1.6);
			\draw [->](4.4,-0.4) -- (4.4,-1.6);
			\draw (3.1,0.2) node{\footnotesize$\left(\begin{smallmatrix} 0 & 1 & 0 \end{smallmatrix}\right)$};
			\draw (3.1,-1.6) node{\footnotesize$\left(\begin{smallmatrix} 0 & 1 & 0 \end{smallmatrix}\right)$};
			\draw (2,-1) node{\footnotesize$\varphi_x$};
			\draw (4.77,-1) node{\footnotesize$\varphi_{y_4}$};
			\end{tikzpicture}
		\end{center}
		The above squares imply that $\varphi_x$ is a matrix
		\begin{equation*}
			\varphi_x =
			\begin{pmatrix}
				c	&	0	&	0	\\
				0	&	c	&	0	\\
				a	&	b	&	c
			\end{pmatrix},
		\end{equation*}
		where $c=\varphi_{y_2}(1)=\varphi_{y_3}(1)=\varphi_{y_4}(1)$ and $c\lambda_1 = c\lambda_2$. Thus, $c \in K$ if $\lambda_1 = \lambda_2$ and $c=0$ if $\lambda_1 \neq \lambda_2$. Thus, we conclude that
		\begin{equation*}
			\End_A(M_\lambda) \cong \left\{
			\begin{pmatrix}
				c	&	0	&	0	\\
				0	&	c	&	0	\\
				a	&	b	&	c
			\end{pmatrix}
			\middle| a,b,c \in K \right\}
		\end{equation*}
		and hence, $\End_A(M_\lambda)$ contains no non-trivial idempotents. So $M_\lambda$ is indecomposable for all $\lambda \in K^{\ast}$. In addition, if $\lambda_1 \neq \lambda_2$ then we note that $M_{\lambda_1} \not\cong M_{\lambda_2}$, since then
		\begin{equation*}
			\Hom_A(M_{\lambda_1}, M_{\lambda_2}) \cong \left\{
			\begin{pmatrix}
				0	&	0	&	0	\\
				0	&	0	&	0	\\
				a	&	b	&	0
			\end{pmatrix}
			\middle| a,b \in K \right\}
		\end{equation*}
		and every $X \in \Hom_A(M_{\lambda_1}, M_{\lambda_2})$ is not invertible. Hence, $M_\lambda$ describes a 1-parameter family of indecomposable $A$-modules.
	\end{exam}
	
	\begin{thm} \label{TameQuad}
		Let $A$ be a Brauer configuration algebra. If $A$ is tame then $A$ is at most quadserial. In particular, $A$ is a tame symmetric special quadserial algebra if and only if $A$ is given by the Brauer configuration
		
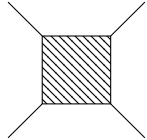
\begin{center}
			\begin{tikzpicture}[scale=0.9]
				\draw[pattern = north west lines] (0.5,-0.5) -- (0.5,0.5) -- (1.5,0.5) -- (1.5,-0.5) -- (0.5,-0.5);
				\draw (0,1) -- (0.5,0.5);
				\draw (0,-1) -- (0.5,-0.5);
			\draw (1.5,0.5) -- (2,1);
			\draw (1.5,-0.5) -- (2,-1);
			\end{tikzpicture}
		\end{center}
		in which every vertex has multiplicity one.
	\end{thm}
	\begin{proof}
		We first note that the quadserial Brauer configuration algebra presented in the theorem is a tame radical cube zero algebra from the classification in \cite{BensonRad}. So suppose instead that $\chi$ is a Brauer configuration not of the above form and that there exists an $n$-gon $x$ with $n>3$ in $\chi$. It follows from Proposition~\ref{CrossBand} that $A$ is wild if $x$ is self-folded. We shall therefore assume that $x$ is not a self-folded polygon. There are multiple cases to consider in the proof.
		
		Case 1: Suppose there exists a vertex $v$ connected to $x$ such that $\val(v)=1$. Then it follows from the definition of a Brauer configuration that $\mathfrak{e}_{v}>1$. We will show that $A$ is wild in this case. We may assume that $\mathfrak{e}_{u}=1$ for all vertices $u \neq v$ in $\chi$, and that $\chi$ is a tree, since $A$ would otherwise be wild by Proposition~\ref{CrossBand}. Now choose a 4-tuple $(v, u_2, u_3, u_4)$ of distinct vertices connected to $x$. Let $\alpha_1$ be the arrow (which is a loop in the quiver) generated by the vertex $v$ and let $\alpha_2$, $\alpha_3$ and $\alpha_4$ be the arrows of source $x$ such that $\widehat{s}(\alpha_2)=x^{u_2}$, $\widehat{s}(\alpha_3)=x^{u_3}$ and $\widehat{s}(\alpha_4)=x^{u_4}$. Let $y_2=e(\alpha_2)$, $y_3=e(\alpha_3)$ and $y_4=e(\alpha_4)$. Let $M_\lambda$ be the family of modules defined in Example~\ref{QuadBand2}.
		
		Note that there are two copies of $S(x)$ in $\tp M_\lambda$ and one copy of $S(x)$ in $\soc M_\lambda$. We further note that $P(x)$ has the following structure
		
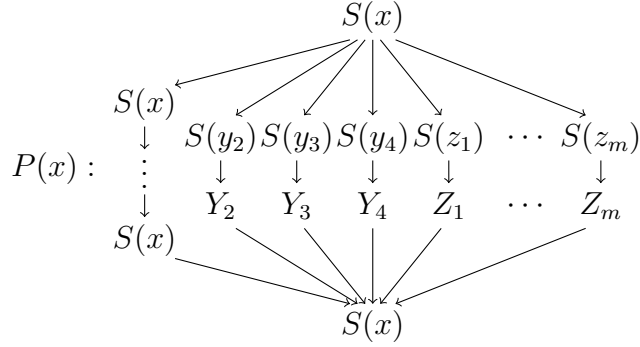
\begin{center}
			\begin{tikzpicture}
			\draw (-2.2,-1.3) node {$P(x):$};
			\draw (2,0.7) node {$S(x)$};
			
			\draw [->](1.7,0.4) -- (-0.6,-0.2);
			\draw [->](1.8,0.4) -- (0.2,-0.6);
			\draw [->](1.9,0.4) -- (1.1,-0.6);
			\draw [->](2,0.4) -- (2,-0.6);
			\draw [->](2.1,0.4) -- (2.9,-0.6);
			\draw [->](2.3,0.4) -- (4.8,-0.6);
			\draw (-1,-0.45) node {$S(x)$};
			\draw (0,-0.9) node {$S(y_2)$};
			\draw (1,-0.9) node {$S(y_3)$};
			\draw (2,-0.9) node {$S(y_4)$};
			\draw (3,-0.9) node {$S(z_1)$};
			\draw (4.04,-0.9) node {$\cdots$};
			\draw (5,-0.9) node {$S(z_m)$};
			
			\draw [->](-1,-0.75) -- (-1,-1.05);
			\draw (-1,-1.25) node {$\vdots$};
			\draw [->](-1,-1.65) -- (-1,-1.95);
			\draw (-1,-2.25) node {$S(x)$};
			
			\draw [->](0,-1.2) -- (0,-1.5);
			\draw [->](1,-1.2) -- (1,-1.5);
			\draw [->](2,-1.2) -- (2,-1.5);
			\draw (0,-1.8) node {$Y_2$};
			\draw (1,-1.8) node {$Y_3$};
			\draw (2,-1.8) node {$Y_4$};
			\draw [->](3,-1.2) -- (3,-1.5);
			\draw [->](5,-1.2) -- (5,-1.5);
			\draw (3,-1.8) node {$Z_1$};
			\draw (4.04,-1.8) node {$\cdots$};
			\draw (5,-1.8) node {$Z_m$};
			
			\draw [->](-0.6,-2.5) -- (1.7,-3.1);
			\draw [->](0.2,-2.1) -- (1.8,-3.1);
			\draw [->](1.1,-2.1) -- (1.9,-3.1);
			\draw [->](2,-2.1) -- (2,-3.1);
			\draw [->](2.9,-2.1) -- (2.1,-3.1);
			\draw [->](4.8,-2.1) -- (2.3,-3.1);
			\draw (2,-3.4) node {$S(x)$};
			\end{tikzpicture}
		\end{center}
		where the $Y_i$ and $Z_i$ are all uniserial modules. It follows from the structure of $P(x)$ and the structure of $M_\lambda$ that
		\begin{equation*}
			S(x) \oplus \bigoplus_{i=2}^4 S(y_i) \subseteq \tp \Omega(M_\lambda).
		\end{equation*}
		
		To show that Case 1 is wild, we will calculate the $y_2$ entry of the dimension vector of $\Omega^2(M_\lambda)$, although the following calculations also hold for any $y_i$. Since we have assumed that $\chi$ is a tree and $\mathfrak{e}_{u}=1$ for any vertex $u \neq v$, we have $\dim(Z_i\varepsilon_{y_2})=\dim(S(z_i)\varepsilon_{y_2})=\dim(Y_i\varepsilon_{y_2})=0$ for all $i$. So
		\begin{align*}
			\dim(\Omega(M_\lambda)\varepsilon_{y_2})	&= 2\dim(P(x)\varepsilon_{y_2}) - \dim(M_\lambda\varepsilon_{y_2}) = 2 - 1=1\\
			\dim(\Omega^2(M_\lambda)\varepsilon_{y_2})	&\geq \dim(P(x)\varepsilon_{y_2}) + \sum_{i=2}^4 \dim(P(y_i)\varepsilon_{y_2}) - \dim(\Omega(M_\lambda)\varepsilon_{y_2}) \\
			&\geq 1 + 2 - 1 = 2 > \dim(M_\lambda\varepsilon_{y_2}).
		\end{align*}
		Thus, $\tau(M_\lambda)=\Omega^2(M_\lambda) \not\cong M_\lambda$. Since $M_\lambda$ describes an infinite family of non-isomorphic indecomposable modules, we conclude in this case that the algebra $A$ must be wild by the contrapositive of Theorem~\ref{tameTube}.
		
		Case 2: Now assume that there is no vertex $v$ connected to $x$ such that $\val(v)=1$. Note that in this case, we cannot make the assumption that $\chi$ is a tree. Choose a 4-tuple $(u_1, u_2, u_3, u_4)$ of distinct vertices connected to $x$. Let $\alpha_1$, $\alpha_2$, $\alpha_3$ and $\alpha_4$ be the arrows of source $x$ such that $\widehat{s}(\alpha_i)=x^{u_i}$ for each $1 \leq i \leq 4$. Let $y_i=e(\alpha_i)$  for each $1 \leq i \leq 4$. Let $M_\lambda$ be the family of circle modules defined in Example~\ref{QuadBand1}. We have various subcases to consider.
		
		Case 2a: Suppose $|x|>4$. Then $P(x)$ is of the form
		
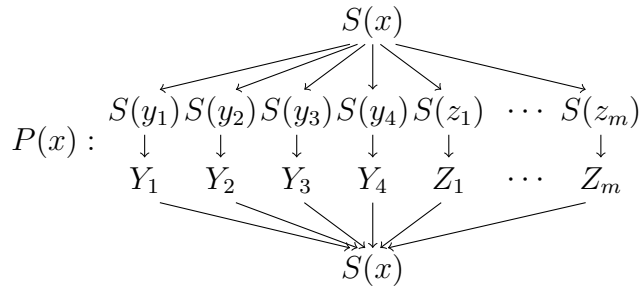
\begin{center}
			\begin{tikzpicture}
			\draw (-2.2,-1.3) node {$P(x):$};
			\draw (2,0.3) node {$S(x)$};
			\draw [->](1.7,0) -- (-0.8,-0.6);
			\draw [->](1.8,0) -- (0.2,-0.6);
			\draw [->](1.9,0) -- (1.1,-0.6);
			\draw [->](2,0) -- (2,-0.6);
			\draw [->](2.1,0) -- (2.9,-0.6);
			\draw [->](2.3,0) -- (4.8,-0.6);
			\draw (3,-0.9) node {$S(z_1)$};
			\draw (4.04,-0.9) node {$\cdots$};
			\draw (5,-0.9) node {$S(z_m)$};
			
			\draw [->](3,-1.2) -- (3,-1.5);
			\draw [->](5,-1.2) -- (5,-1.5);
			\draw (3,-1.8) node {$Z_1$};
			\draw (4.04,-1.8) node {$\cdots$};
			\draw (5,-1.8) node {$Z_m$};
			\draw (2,-3) node {$S(x)$};
			
			\draw (-1,-0.9) node {$S(y_1)$};
			\draw (0,-0.9) node {$S(y_2)$};
			\draw (1,-0.9) node {$S(y_3)$};
			\draw (2,-0.9) node {$S(y_4)$};
			
			\draw (-1,-1.8) node {$Y_1$};
			\draw (0,-1.8) node {$Y_2$};
			\draw (1,-1.8) node {$Y_3$};
			\draw (2,-1.8) node {$Y_4$};
			\draw [->](-1,-1.2) -- (-1,-1.5);
			\draw [->](0,-1.2) -- (0,-1.5);
			\draw [->](1,-1.2) -- (1,-1.5);
			\draw [->](2,-1.2) -- (2,-1.5);
			
			\draw [->](-0.8,-2.1) -- (1.7,-2.7);
			\draw [->](0.2,-2.1) -- (1.8,-2.7);
			\draw [->](1.1,-2.1) -- (1.9,-2.7);
			\draw [->](2,-2.1) -- (2,-2.7);
			\draw [->](2.9,-2.1) -- (2.1,-2.7);
			\draw [->](4.8,-2.1) -- (2.2,-2.7);
			\end{tikzpicture}
		\end{center}
		where $m>0$ and the $Y_i$ and $Z_i$ are all uniserial modules. We note in this case that there are two copies of each $S(z_i)$ in $\tp \Omega(M_\lambda)$ and that $\soc \Omega(M_\lambda) = S(x) \oplus S(x)$. Since $\soc P(z_i) = S(z_i)$ for all $i$, there must exist a copy of $S(z_i)$ in $\soc \Omega^2(M_\lambda)$. Thus, $\tau(M_\lambda)\not\cong M_\lambda$ for any $\lambda \in K^{\ast}$, and so $A$ is wild.
		
		Case 2b: Suppose $|x|=4$ and that there exists an integer $r$ such that $y_r$ is not uniserial. Then $P(y_r)$ has the following structure.
		\begin{center}
			\begin{tikzpicture}
			\draw (-2.2,-1.3) node {$P(y_r):$};
			\draw (0.5,0.1) node {$S(y_r)$};
			\draw [->](0.2,-0.2) -- (-0.5,-0.6);
			\draw [->](0.4,-0.2) -- (0.4,-1);
			\draw [->](0.8,-0.2) -- (1.5,-1);
			
			\draw (0.5,-2.8) node {$S(y_r)$};
			\draw (-0.7,-0.9) node {$Y_r$};
			\draw (0.4,-1.35) node {$V_1$};
			\draw (1.04,-1.35) node {$\cdots$};
			\draw (1.7,-1.35) node {$V_n$};
			
			\draw (-0.7,-1.8) node {$S(x)$};
			\draw [->](-0.7,-1.2) -- (-0.7,-1.5);
			
			\draw [->](-0.5,-2.1) -- (0.2,-2.5);
			\draw [->](0.4,-1.7) -- (0.4,-2.5);
			\draw [->](1.5,-1.7) -- (0.8,-2.5);
			\end{tikzpicture}
		\end{center}
		where $Y_r, V_1,\ldots,V_m$ are uniserial. We note that $\rad P(y_r) / \soc P(y_r)$ contains a direct summand which is a uniserial submodule of $P(x)$. This is precisely the module with top isomorphic to $\tp Y_r$ and socle isomorphic to $S(x)$ presented in the structure of $P(y_r)$ above. We further note that no other direct summand of $\rad P(y_r) / \soc P(y_r)$ is a submodule of $P(x)$. Since $\tp M_\lambda = S(x) \oplus S(x)$ and $S(y_r)$ is a direct summand of $\tp \Omega(M_\lambda)$, it follows that for each $i$, $\tp V_i$ is a direct summand of $\tp \Omega^2(M_\lambda)$. Thus, $\tau(M_\lambda)\not\cong M_\lambda$ for any $\lambda \in K^{\ast}$, and so $A$ is wild by the contrapositive of Theorem~\ref{tameTube}.
		
		Case 2c: Suppose that $|x|=4$, that every $y_i$ is such that $P(y_i)$ is uniserial, and that for some $r$, we have $\rad^3 P(y_r) \neq 0$. We will calculate $\tau^{-1} M_\lambda$ in this case, as it is a simpler calculation. Let $v$ be the vertex connected to $x$ and $y_r$. Let $z$ be the predecessor to $x$ at $v$ and let $a_i = \dim(Y_i\varepsilon_z)$. Then $a_r\geq 1$. If $a_i >0$ for some $i \neq r$, then this implies that there exists a polygon connecting two distinct vertices of $x$, which induces a simple cycle in $\chi$. This implies $A$ is wild by Proposition~\ref{CrossBand}, so assume that this is not the case. Similarly, we have $\dim(S(y_i)\varepsilon_z)=0$ for all $i \neq r$. Note that $\soc \Omega^{-1}(M_\lambda)=S(x) \oplus S(x)$. So
		\begin{align*}
			\dim(\Omega^{-1}(M_\lambda)\varepsilon_z) &= \sum_{i=1}^4\dim(P(y_i)\varepsilon_z) - \dim(M_\lambda\varepsilon_z) \\
			&= 2\dim(S(y_r)\varepsilon_z)+a_r - \dim(S(y_r)\varepsilon_z) \\
			&=\dim(S(y_r)\varepsilon_z)+a_r \\
			\dim(\Omega^{-2}(M_\lambda)\varepsilon_z)	&\geq 2\sum_{i=1}^4 \dim(P(x)\varepsilon_z) - \dim(\Omega^{-1}(M_\lambda)\varepsilon_z) \\
			&\geq 2\left(\dim(S(y_r)\varepsilon_z)+a_r\right)- \left(\dim(S(y_r)\varepsilon_z)+a_r\right) \\
			&\geq \dim(S(y_r)\varepsilon_z)+a_r \\
			&\geq \dim(M_\lambda\varepsilon_z)+1 >\dim(M_\lambda\varepsilon_z).
		\end{align*}
		Thus, $\tau^{-1}(M_\lambda)\not\cong M_\lambda$ for any $\lambda \in K^{\ast}$. Hence, $\tau(M_\lambda)\not\cong M_\lambda$ for any $\lambda \in K^{\ast}$, and so $A$ is wild by the contrapositive of Theorem~\ref{tameTube}.
		
		In conclusion, if $A$ is tame, then $\chi$ does not contain a polygon $x$ with $|x|>4$. If $\chi$ contains a polygon $x$ with $|x|=4$, then $x$ is not self-folded and no vertex incident to $x$ has valency one. Each vertex must therefore have an edge $y_i \neq x$ incident to it. Moreover, every edge $y_i$ incident to $x$ must be such that $P(y_i)$ is uniserial and $\rad^3 P(y_i) = 0$. Thus, $\chi$ is precisely the Brauer configuration
		
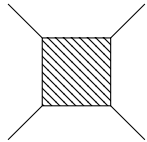
\begin{center}
			\begin{tikzpicture}[scale=0.9]
				\draw[pattern = north west lines] (0.5,-0.5) -- (0.5,0.5) -- (1.5,0.5) -- (1.5,-0.5) -- (0.5,-0.5);
				\draw (0,1) -- (0.5,0.5);
				\draw (0,-1) -- (0.5,-0.5);
			\draw (1.5,0.5) -- (2,1);
			\draw (1.5,-0.5) -- (2,-1);
			\end{tikzpicture}
		\end{center}
		in which every vertex has multiplicity one, which is known to be associated to a tame algebra.
	\end{proof}
	
	
	\section{Tame Symmetric Special Triserial Algebras} \label{QCSection}
	We will briefly recall some terminology and results from \cite{Fernandez}, \cite{Fernandez2}, \cite{trivialGentle} and \cite{AlmostGentle}.
	
	\begin{defn}[\cite{Fernandez}, \cite{Fernandez2}, \cite{trivialGentle}, \cite{AlmostGentle}] \label{AdmisCut}
		Let $A=KQ/I$ be a Brauer configuration algebra associated to a multiplicity-free Brauer configuration $\chi$. Suppose $\chi_0=\{v_1, \ldots, v_n\}$ and let $D$ be a set consisting of precisely one arrow from each cycle $\mathfrak{C}_{v_i}$. We call $D$ an \emph{admissible cut} of $Q$ and we call $KQ/\langle I \cup D \rangle$ the \emph{cut algebra associated to $D$}, where $\langle I \cup D \rangle$  is the ideal generated by $I \cup D$.
	\end{defn}
	
	\begin{thm}[\cite{AlmostGentle}, \cite{trivialGentle}] \label{CutAlgebra}
		Let $A=KQ/I$ be a Brauer configuration algebra associated to a multiplicity-free Brauer configuration $\chi$. Let $D$ be an admissible cut of $Q$. Define a quiver $Q'$ by $Q'_0=Q_0$ and $Q'_1=Q_1 \setminus D$. Then the cut algebra $KQ/\langle I \cup D \rangle$ is isomorphic to the (basic) algebra $B=KQ'/(I \cap KQ')$. Furthermore, $B$ is almost gentle and the trivial extension algebra $T(B) \cong A$. If in addition $\chi$ is a Brauer graph, then $B$ is gentle.
	\end{thm}
	
	We first aim to prove the following.
	\begin{prop} \label{MultFreeCase}
		Let $A$ be a Brauer configuration algebra associated to a configuration of the form 
		\begin{center}
			\newlength{\hatchspread}
			\newlength{\hatchthickness}
			\newlength{\hatchshift}
			\newcommand{\hatchcolor}{}
			\tikzset{hatchspread/.code={\setlength{\hatchspread}{#1}},
			         hatchthickness/.code={\setlength{\hatchthickness}{#1}},
			         hatchshift/.code={\setlength{\hatchshift}{#1}},
			         hatchcolor/.code={\renewcommand{\hatchcolor}{#1}}}
			\tikzset{hatchspread=3pt,
			         hatchthickness=0.4pt,
			         hatchshift=0pt,
			         hatchcolor=black}
			\pgfdeclarepatternformonly[\hatchspread,\hatchthickness,\hatchshift,\hatchcolor]
			   {custom north west lines}
			   {\pgfqpoint{\dimexpr-2\hatchthickness}{\dimexpr-2\hatchthickness}}
			   {\pgfqpoint{\dimexpr\hatchspread+2\hatchthickness}{\dimexpr\hatchspread+2\hatchthickness}}
			   {\pgfqpoint{\dimexpr\hatchspread}{\dimexpr\hatchspread}}
			   {
			    \pgfsetlinewidth{\hatchthickness}
			    \pgfpathmoveto{\pgfqpoint{0pt}{\dimexpr\hatchspread+\hatchshift}}
			    \pgfpathlineto{\pgfqpoint{\dimexpr\hatchspread+0.15pt+\hatchshift}{-0.15pt}}
			    \ifdim \hatchshift > 0pt
			      \pgfpathmoveto{\pgfqpoint{0pt}{\hatchshift}}
			      \pgfpathlineto{\pgfqpoint{\dimexpr0.15pt+\hatchshift}{-0.15pt}}
			    \fi
			    \pgfsetstrokecolor{\hatchcolor}
			    \pgfusepath{stroke}
			   }
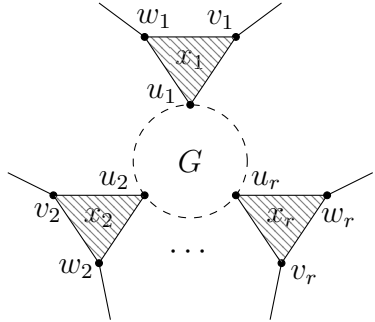
			\begin{tikzpicture}[scale=1.5, ]
			\draw [dashed] (0,0) ellipse (0.5 and 0.5);
			\draw (0,0) node {$G$};
			\draw [fill=black] (0,0.5) ellipse (0.03 and 0.03);
			\draw (-0.25,0.6) node {$u_1$};
			\draw [fill=black] (-0.4,-0.3) ellipse (0.03 and 0.03);
			\draw (-0.667,-0.1665) node {$u_2$};
			\draw [fill=black] (0.4,-0.3) ellipse (0.03 and 0.03);
			\draw (0.6665,-0.1665) node {$u_r$};
			
			\draw [fill=black] (0.4,1.1) ellipse (0.03 and 0.03);
			\draw (0.2665,1.25) node {$v_1$};
			\draw [fill=black] (-0.4,1.1) ellipse (0.03 and 0.03);
			\draw (-0.3165,1.2634) node {$w_1$};
			\draw [fill=black] (-1.2,-0.3) ellipse (0.03 and 0.03);
			\draw (-1.25,-0.45) node {$v_2$};
			\draw [fill=black] (-0.8,-0.9) ellipse (0.03 and 0.03);
			\draw (-1,-0.95) node {$w_2$};
			\draw [fill=black] (0.8,-0.9) ellipse (0.03 and 0.03);
			\draw (1,-1) node {$v_r$};
			\draw [fill=black] (1.2,-0.3) ellipse (0.03 and 0.03);
			\draw (1.3,-0.5) node {$w_r$};
			\draw (0,0.9) node {$x_1$};
			\draw (-0.8,-0.5) node {$x_2$};
			\draw (0.8,-0.5) node {$x_r$};
			
			\draw [pattern=custom north west lines, hatchcolor=gray](0,0.5) -- (-0.4,1.1) -- (0.4,1.1) -- (0,0.5);
			\draw (-0.8,1.4) -- (-0.4,1.1);
			\draw (0.8,1.4) -- (0.4,1.1);
			\draw [pattern=custom north west lines, hatchcolor=gray](-1.2,-0.3) -- (-0.4,-0.3) -- (-0.8,-0.9) -- (-1.2,-0.3);
			\draw (-1.6,-0.1) -- (-1.2,-0.3);
			\draw (-0.7,-1.4) -- (-0.8,-0.9);
			\draw [pattern=custom north west lines, hatchcolor=gray](0.4,-0.3) -- (1.2,-0.3) -- (0.8,-0.9) -- (0.4,-0.3);
			\draw (1.6,-0.1) -- (1.2,-0.3);
			\draw (0.7,-1.4) -- (0.8,-0.9);
						\draw (0,-0.8) node {$\cdots$};
			\end{tikzpicture}
		\end{center}
		where $G$ is a Brauer graph connected the (not necessarily distinct) vertices $u_1,\ldots, u_r$ and $\mathfrak{e}_v=1$ for all $v \in \chi_0$. Then $A$ is tame.
	\end{prop}
	
	For each $i$, let $\mathfrak{C}_{u_i,\alpha_{i,1}}=\alpha_{i,1}\ldots\alpha_{i,n_i}$, $\mathfrak{C}_{v_i,\beta_i}=\beta_i\gamma_i$ and $\mathfrak{C}_{w_i,\delta_i}=\delta_i\zeta_i$, where $s(\alpha_{i,1})=s(\beta_i)=s(\delta_i)=x_i$. We choose an admissible cut $D$ of $Q$ such that $\beta_i, \delta_i \in D$ for all $i$. Let $Q'$ be the quiver such that $Q'_0=Q_0$ and $Q'_1=Q_1 \setminus D$. Then by Theorem~\ref{CutAlgebra}, the algebra $B=KQ' / (I \cap KQ')$ is almost gentle. For each $i$, there are three possible arrangements of the arrows in $Q'$ at $x_i$, which are as follows.
	
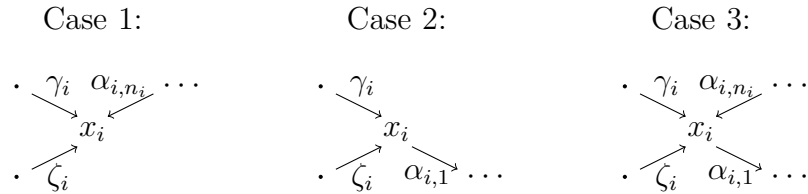
\begin{center}
		\begin{tikzpicture}
		\draw (-4,1.5) node {Case 1:};
		\draw (-4,0) node {$x_i$};
		\draw [->](-3.2,0.5) -- (-3.8,0.2);
		\draw (-3.65,0.6) node {$\alpha_{i,n_i}$};
		\draw [->](-4.8,0.5) -- (-4.2,0.2);
		\draw (-4.45,0.6) node {$\gamma_i$};
		\draw [->](-4.8,-0.5) -- (-4.2,-0.2);
		\draw (-4.45,-0.6) node {$\zeta_i$};
		\draw (-2.8,0.6) node {$\cdots$};
		\draw [fill=black] (-5,0.6) ellipse (0.02 and 0.02);
		\draw [fill=black] (-5,-0.6) ellipse (0.02 and 0.02);
		
		\draw (0,1.5) node {Case 2:};
		\draw (0,0) node {$x_i$};
		\draw [->](0.2,-0.2) -- (0.8,-0.5);
		\draw (0.4,-0.6) node {$\alpha_{i,1}$};
		\draw [->](-0.8,0.5) -- (-0.2,0.2);
		\draw (-0.45,0.6) node {$\gamma_i$};
		\draw [->](-0.8,-0.5) -- (-0.2,-0.2);
		\draw (-0.45,-0.6) node {$\zeta_i$};
		\draw (1.2,-0.6) node {$\cdots$};
		\draw [fill=black] (-1,0.6) ellipse (0.02 and 0.02);
		\draw [fill=black] (-1,-0.6) ellipse (0.02 and 0.02);
		
		\draw (4,1.5) node {Case 3:};
		\draw (4,0) node {$x_i$};
		\draw [->](4.8,0.5) -- (4.2,0.2);
		\draw (4.35,0.6) node {$\alpha_{i,n_i}$};
		\draw [->](4.2,-0.2) -- (4.8,-0.5);
		\draw (4.4,-0.6) node {$\alpha_{i,1}$};
		\draw [->](3.2,0.5) -- (3.8,0.2);
		\draw (3.55,0.6) node {$\gamma_i$};
		\draw [->](3.2,-0.5) -- (3.8,-0.2);
		\draw (3.55,-0.6) node {$\zeta_i$};
		\draw (5.2,0.6) node {$\cdots$};
		\draw (5.2,-0.6) node {$\cdots$};
		\draw [fill=black] (3,0.6) ellipse (0.02 and 0.02);
		\draw [fill=black] (3,-0.6) ellipse (0.02 and 0.02);
		\end{tikzpicture}
	\end{center}
	In Case 1, there are no relations in $I \cap KQ'$ involving either $\gamma_i$ or $\zeta_i$. In Cases 2 and 3, we have zero relations $\gamma_i \alpha_{i,1}$ and $\zeta_i \alpha_{i,1}$. Also note the following remark.
	
	\begin{rem} \label{GentleCut}
		Let $Q''$ be a quiver defined by $Q''_0 = Q'_0 \setminus \{s(\zeta_i) : 1 \leq i \leq r\}$ and $Q''_1= Q'_1 \setminus \{\zeta_i : 1 \leq i \leq r \}$. Then $KQ'' / (I \cap KQ'')$ is gentle. This follows from Theorem~\ref{CutAlgebra} and the fact that the vertices $u_1,\ldots,u_r \in \chi_0$ are connected by a Brauer graph.
	\end{rem}
	
	\begin{lem} \label{SGIso}
		Let $B$ be the cut algebra defined above and let $Q''$ be the quiver defined in Remark~\ref{GentleCut}. Let $C= KQ'''/I'''$, where $Q'''_0 = Q''_0$, $Q'''_1 = Q''_1 \cup \{\eta_1,\ldots,\eta_r\}$ with $s(\eta_i)=s(\gamma_i)=e(\eta_i)$ for all $1 \leq i \leq r$, and $I''' = \langle (I \cap KQ'') \cup \{ \eta^2_i -\eta : 1 \leq i \leq r\}\rangle$. Then $B \cong C$.
	\end{lem}
	
	\begin{proof}
		For the purposes of clarity, we will write each element $\xi \in C$ as $\widehat{\xi}$. For each $i$, there are three possible arrangements of the arrows in $Q'''$ at $x_i$ and $s(\widehat{\gamma}_i)$, which are as follows.
		
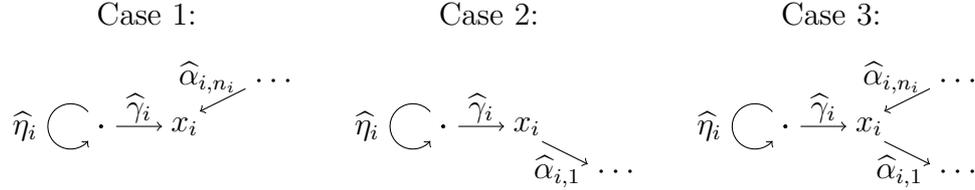
\begin{center}
			\begin{tikzpicture}
			\draw (-5,1.5) node {Case 1:};
			\draw (-4.5,0) node {$x_i$};
			\draw [->](-3.7,0.5) -- (-4.3,0.2);
			\draw (-4.2,0.65) node {$\widehat{\alpha}_{i,n_i}$};
			\draw [->](-5.4,0) -- (-4.8,0);
			\draw (-5.1,0.25) node {$\widehat{\gamma}_i$};
			\draw [->](-5.7702,0.1928) arc (39.9963:320:0.3);
			\draw (-6.6,0) node {$\widehat{\eta}_i$};
			\draw (-3.3,0.6) node {$\cdots$};
			\draw [fill=black] (-5.6,0) ellipse (0.02 and 0.02);
			
			\draw (-0.5,1.5) node {Case 2:};
			\draw (0,0) node {$x_i$};
			\draw [->](0.2,-0.2) -- (0.8,-0.5);
			\draw (0.4,-0.6) node {$\widehat{\alpha}_{i,1}$};
			\draw [->](-0.9,0) -- (-0.3,0);
			\draw (-0.6,0.25) node {$\widehat{\gamma}_i$};
			\draw [->](-1.2702,0.1928) arc (39.9963:320:0.3);
			\draw (-2.1,0) node {$\widehat{\eta}_i$};
			\draw (1.2,-0.6) node {$\cdots$};
			\draw [fill=black] (-1.1,0) ellipse (0.02 and 0.02);
			
			\draw (4,1.5) node {Case 3:};
			\draw (4.5,0) node {$x_i$};
			\draw [->](5.3,0.5) -- (4.7,0.2);
			\draw (4.8,0.65) node {$\widehat{\alpha}_{i,n_i}$};
			\draw [->](4.7,-0.2) -- (5.3,-0.5);
			\draw (4.9,-0.6) node {$\widehat{\alpha}_{i,1}$};
			\draw [->](3.6,0) -- (4.2,0);
			\draw (3.9,0.25) node {$\widehat{\gamma}_i$};
			\draw [->](3.2298,0.1928) arc (39.9963:320:0.3);
			\draw (2.4,0) node {$\widehat{\eta}_i$};
			\draw (5.7,0.6) node {$\cdots$};
			\draw (5.7,-0.6) node {$\cdots$};
			\draw [fill=black] (3.4,0) ellipse (0.02 and 0.02);
			\end{tikzpicture}
		\end{center}
		In Cases 2 and 3, there exists a relation $\widehat{\gamma}_i\widehat{\alpha}_{i,1}$. In all three cases, we have a relation $\widehat{\eta}^2_i-\widehat{\eta}_i$. Define a linear map $f\colon C \rightarrow B$ as follows.
		\begin{equation*}
			f(\widehat{\xi})=
			\begin{cases}
				\varepsilon_{s(\gamma_i)} + \varepsilon_{s(\zeta_i)}	&	\text{if } 
				\widehat{\xi} = \widehat{\varepsilon}_{s(\widehat{\gamma}_i)} \text{ for some } 1 \leq i \leq r,	\\
				\varepsilon_{s(\zeta_i)}	&	\text{if } \widehat{\xi}=\widehat{\eta}_i \text{ for some } 1 \leq i \leq r,	\\
				\gamma_i+\zeta_i		&	\text{if } \widehat{\xi}=\widehat{\gamma}_i \text{ for some } 1 \leq i \leq r,	\\
				\zeta_i 					&	\text{if } \widehat{\xi}=\widehat{\eta}_i\widehat{\gamma}_i\text{ for some } 1 \leq i \leq r,	\\
				\xi						&	\text{otherwise}.
			\end{cases}
		\end{equation*}

		Note the following for each $i$.
		\begin{center}
			\begin{tabular}{c c c c c c c}
				$f(\widehat{\gamma}_i \widehat{\varepsilon}_{x_i})$				&$=$&
				$f(\widehat{\gamma}_i)$												&$=$&
				$\gamma_i+\zeta_i$														&$=$&
				$f(\widehat{\gamma}_i)f(\widehat{\varepsilon}_{x_i})$				\\
				$f(\widehat{\varepsilon}_{s(\widehat{\gamma}_i)}\widehat{\gamma}_i)$	&$=$&
				$f(\widehat{\gamma}_i)$												&$=$&
				$\gamma_i+\zeta_i$ 													&$=$&
				$f(\widehat{\varepsilon}_{s(\widehat{\gamma}_i)})f(\widehat{\gamma}_i)$	\\
				$f(\widehat{\eta}_i \widehat{\varepsilon}_{s(\widehat{\gamma}_i)})$		&$=$&
				$f(\widehat{\eta}_i)$												&$=$&
				$\varepsilon_{s(\zeta_i)}$											&$=$&
				$f(\widehat{\eta}_i)f(\widehat{\varepsilon}_{s(\widehat{\gamma}_i)})$		\\
				$f(\widehat{\varepsilon}_{s(\widehat{\gamma}_i)}\widehat{\eta}_i)$			&$=$&
				$f(\widehat{\eta}_i)$												&$=$&
				$\varepsilon_{s(\zeta_i)}$											&$=$&
				$f(\widehat{\varepsilon}_{s(\widehat{\gamma}_i)})f(\widehat{\eta}_i)$		\\
				$f(\widehat{\eta}_i\widehat{\gamma}_i)$							&$=$&
																					&&
				$\zeta_i$ 															&$=$&
				$f(\widehat{\eta}_i)f(\widehat{\gamma}_i)$						\\
				$f(\widehat{\gamma}_i\widehat{\alpha}_{i,1})$					&$=$&
																					&&
				0 																	&$=$&
				$f(\widehat{\gamma}_i)f(\widehat{\alpha}_{i,1})$
			\end{tabular}
		\end{center}
		Trivially, we also have $f(\widehat{\xi}_1\widehat{\xi}_2)=\xi_1 \xi_2=f(\widehat{\xi}_1)f(\widehat{\xi}_2)$ for all
		\begin{equation*}
			\widehat{\xi}_1,\widehat{\xi}_2 \in C \setminus \{\widehat{\varepsilon}_{s(\widehat{\gamma}_i)},\widehat{\eta}_i,\widehat{\gamma}_i, \widehat{\eta}_i\widehat{\gamma}_i  : 1 \leq i \leq r\}.
		\end{equation*}
		Thus, $f$ is well-defined as an algebra homomorphism.
		
		To show $f$ is an isomorphism, note that for each vertex $y \in Q'''_0$, the image of the vector subspace $\widehat{\varepsilon}_{y} C$ of $C$ under $f$ is
		\begin{equation*}
			f(\widehat{\varepsilon}_y C)=
			\begin{cases}
				\varepsilon_y B	&	\text{if } y \in Q'''_0 \setminus \{s(\widehat{\gamma}_i) : 1 \leq i \leq r\},	\\
				V_i					&	\text{if } y \in \{s(\widehat{\gamma}_i) : 1 \leq i \leq r\},
			\end{cases}
		\end{equation*}
		where each $V_i=\langle \varepsilon_{s(\gamma_i)}+\varepsilon_{s(\zeta_i)}, \gamma_i + \zeta_i, \varepsilon_{s(\zeta_i)}, \zeta_i \rangle$ is considered as a vector subspace of $B$. By the definition of $B$ and $C$, we have $\dim_K \widehat{\varepsilon}_y C = \dim_K \varepsilon_y B$ for all $y \in Q'''_0 \setminus \{s(\widehat{\gamma}_i) : 1 \leq i \leq r\}$. Since $B$ and $C$ are finite dimensional, this implies that each subspace $\widehat{\varepsilon}_y C$ of $C$ with $y \in Q'''_0 \setminus \{s(\widehat{\gamma}_i) : 1 \leq i \leq r\}$ is bijectively mapped onto the corresponding subspace $\varepsilon_y B$ of $B$ under $f$. However, also note that $V_i\cong \varepsilon_{s(\gamma_i)}B \oplus \varepsilon_{s(\zeta_i)}B$ for all $i$, and $\dim_K\varepsilon_{s(\gamma_i)}B=\dim_K\varepsilon_{s(\zeta_i)}B=2$. So $\dim_K \widehat{\varepsilon}_{s(\widehat{\gamma}_i)} C=\dim_K V_i$ for all $i$. Hence, each subspace $\widehat{\varepsilon}_{s(\widehat{\gamma}_i)} C$ of $C$ is bijectively mapped onto the subspace $\varepsilon_{s(\gamma_i)}B \oplus \varepsilon_{s(\zeta_i)}B$ of $B$. Now
		\begin{equation*}
			B=\bigoplus_{y \in Q'_0} \varepsilon_y B \qquad \text{and} \qquad C=\sum_{\widehat{y} \in Q'''_0} \widehat{\varepsilon}_y C.
		\end{equation*}
		Since $f$ is bijective on each subspace $\widehat{\varepsilon}_y C$ and $f(\widehat{\varepsilon}_y C)\cap f(\widehat{\varepsilon}_z C)=\emptyset$ for all $y \neq z$, $C$ is actually a direct sum of subspaces $\widehat{\varepsilon}_y C$. Moreover, $f$ is bijective on the whole of $C$, and hence is an isomorphism.
	\end{proof}
	
	We may now prove Proposition~\ref{MultFreeCase}
	\begin{proof}
		By Remark~\ref{GentleCut}, the quiver and relations of the algebra $C$ given in Lemma~\ref{SGIso} satisfy the axioms of a gentle algebra (with the exception of the idempotent relations on the loops $\eta_i$). In fact, $C$ is a skewed gentle algebra in the sense of \cite{Pena}. The trivial extension algebra of a skewed gentle algebra is tame by \cite[Remark 4.9]{Pena}. But since $B$ is a cut algebra obtained from our original algebra $A$, we have $T(C) \cong T(B) \cong A$ by Lemma~\ref{SGIso} and Theorem~\ref{CutAlgebra}. Thus, $A$ is isomorphic to the trivial extension of a skewed gentle algebra and is hence tame.
	\end{proof}
	
	We now wish to generalise to the case where the Brauer configuration is not multiplicity-free.
	\begin{thm} \label{TameQC}
		Let $A=KQ/I$ be a Brauer configuration algebra associated to a Brauer configuration $\chi$ of the form
		
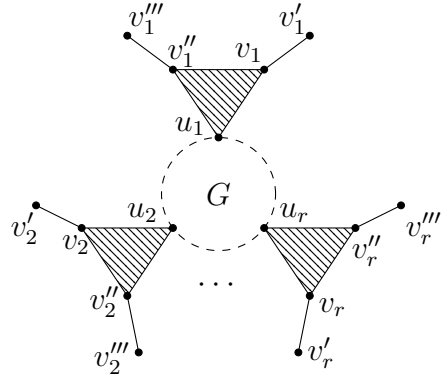
\begin{center}
			\begin{tikzpicture}[scale=1.5, ]
			\draw [dashed] (0,0) ellipse (0.5 and 0.5);
			\draw (0,0) node {$G$};
			\draw [fill=black] (0,0.5) ellipse (0.03 and 0.03);
			\draw (-0.25,0.6) node {$u_1$};
			\draw [fill=black] (-0.4,-0.3) ellipse (0.03 and 0.03);
			\draw (-0.667,-0.1665) node {$u_2$};
			\draw [fill=black] (0.4,-0.3) ellipse (0.03 and 0.03);
			\draw (0.6665,-0.1665) node {$u_r$};
			
			\draw [fill=black] (0.4,1.1) ellipse (0.03 and 0.03);
			\draw (0.2665,1.25) node {$v_1$};
			\draw [fill=black] (0.8,1.4) ellipse (0.03 and 0.03);
			\draw (0.65,1.55) node {$v'_1$};
			\draw [fill=black] (-0.4,1.1) ellipse (0.03 and 0.03);
			\draw (-0.3165,1.2634) node {$v''_1$};
			\draw [fill=black] (-0.8,1.4) ellipse (0.03 and 0.03);
			\draw (-0.6299,1.55) node (o) {$v'''_1$};
			\draw [fill=black] (-1.2,-0.3) ellipse (0.03 and 0.03);
			\draw (-1.25,-0.45) node {$v_2$};
			\draw [fill=black] (-1.6,-0.1) ellipse (0.03 and 0.03);
			\draw (-1.7,-0.3) node {$v'_2$};
			\draw [fill=black] (-0.8,-0.9) ellipse (0.03 and 0.03);
			\draw (-1,-0.95) node {$v''_2$};
			\draw [fill=black] (-0.7,-1.4) ellipse (0.03 and 0.03);
			\draw (-0.9335,-1.45) node {$v'''_2$};
			\draw [fill=black] (0.8,-0.9) ellipse (0.03 and 0.03);
			\draw (1,-1) node {$v_r$};
			\draw [fill=black] (0.7,-1.4) ellipse (0.03 and 0.03);
			\draw (0.9,-1.4) node {$v'_r$};
			\draw [fill=black] (1.2,-0.3) ellipse (0.03 and 0.03);
			\draw (1.3,-0.5) node {$v''_r$};
			\draw [fill=black] (1.6,-0.1) ellipse (0.03 and 0.03);
			\draw (1.8,-0.3) node {$v'''_r$};
			
			\draw [pattern=north west lines](0,0.5) -- (-0.4,1.1) -- (0.4,1.1) -- (0,0.5);
			\draw (-0.8,1.4) -- (-0.4,1.1);
			\draw (0.8,1.4) -- (0.4,1.1);
			\draw [pattern=north west lines](-1.2,-0.3) -- (-0.4,-0.3) -- (-0.8,-0.9) -- (-1.2,-0.3);
			\draw (-1.6,-0.1) -- (-1.2,-0.3);
			\draw (-0.7,-1.4) -- (-0.8,-0.9);
			\draw [pattern=north west lines](0.4,-0.3) -- (1.2,-0.3) -- (0.8,-0.9) -- (0.4,-0.3);
			\draw (1.6,-0.1) -- (1.2,-0.3);
			\draw (0.7,-1.4) -- (0.8,-0.9);
			\draw (0,-0.8) node {$\cdots$};
			\end{tikzpicture}
		\end{center}
		where $G$ is a Brauer graph connecting the (not necessarily distinct) vertices $u_1,\ldots,u_r$ and $\mathfrak{e}_{v_i}=\mathfrak{e}_{v'_i}=\mathfrak{e}_{v''_i}=\mathfrak{e}_{v'''_i}=1$ for all $i$. Then $A$ is tame.
	\end{thm}
	
	From the Brauer configuration $\chi$ in Theorem~\ref{TameQC}, we will construct a new (multiplicity-free) Brauer configuration $\chi'$ in the following way. Let
	\begin{equation*}
		V_{\chi}= \{ t \in \chi_0 : \mathfrak{e}_t>1\}= \{t_1, \ldots, t_n\}
	\end{equation*}
	Let $Y_{\chi}=\{y^{t_1}_1,\ldots,y^{t_n}_n\}$ be a set of germs of polygons in $\chi$ such that each germ is incident to a vertex in $V_\chi$ and no two germs in $Y_\chi$ are incident to the same vertex.  For each $i$, let $z^{t_i}_i$ be the successor to $y^{t_i}_i$. Define a configuration $\chi'=(\chi'_0,\chi'_1,\kappa')$ by
	\begin{align*}
		\chi'_0 &= \chi_0 \cup \{t'_1 \ldots t'_n\} \\
		\chi'_1 &= \chi_1 \cup \{p_1,q_1, \ldots, p_n,q_n\}
	\end{align*}
	where each $p_i = \{p^{t_i}_i, p^{t'_i}_i\}$ and each $q_i=\{q^{t_i}_i, q^{t'_i}_i\}$. The definition of $\kappa'$ is obtained implicitly from the notation. To give $\chi'$ the structure of a Brauer configuration, define the multiplicity function to be $\mathfrak{e}'\equiv 1$. Define the cyclic ordering at each vertex in $\chi'$ to be the same as the vertices in $\chi$, except for those in $V_\chi$.  The cyclic ordering in $\chi'$ at each $t_i \in V_{\chi}$ is defined to be the same as in $\chi$, except for the segment
	\begin{equation*}
		\ldots, y^{t_i}_i, p^{t_i}_i, q^{t_i}_i, z^{t_i}_i, \ldots
	\end{equation*}
	in $\chi'$. Finally, each $q^{t'_i}_i$ is defined to be the successor to $p^{t'_i}_i$ in $\chi'$.
	
	Informally, $\chi'$ is obtained from $\chi$ by adding a simple cycle of length 2 to each vertex in $\chi$ of multiplicity strictly greater than 1, and then by setting the multiplicity of each such vertex to 1. Since $\chi'$ is multiplicity-free, the associated Brauer configuration algebra $A'=KQ'/I'$ is tame by Proposition~\ref{MultFreeCase}. The next step is to construct a function $F:\Ob(\uMod*A) \rightarrow \Ob(\Mod*A')$, where $\uMod*A$ is the projectively stable module category of $A$.
	
	For each $i$, let $\alpha_i$ be the arrow in $Q$ such that $\widehat{s}(\alpha)=y^{t_i}_i$. Note that $Q_0 \subset Q'_0$ and $Q_1 \setminus \{\alpha_1,\ldots, \alpha_n\} \subset Q'_1$. For each $i$, let $\beta_i$, $\gamma_i$, $\delta_i$, $\zeta_i$ and $\eta_i$ be the arrows in $Q'$ such that
	\begin{align*}
		\widehat{s}(\beta_i)&=y^{t_i}_i,		&	\widehat{s}(\zeta_i)&=p^{t'_i}_i, \\
		\widehat{s}(\gamma_i)&=p^{t_i}_i,	&	\widehat{s}(\eta_i)&=q^{t'_i}_i, \\
		\widehat{s}(\delta_i)&=q^{t_i}_i.
	\end{align*}
	Let $M=(M_x, \varphi_\xi)_{x \in Q_0, \xi \in Q_1}$ be a $K$-representation $Q$. For each $i$, write $\varphi_{\alpha_i}=\theta_i\psi_i$, where $\psi_i\colon M_{y_i} \rightarrow \im \varphi_{\alpha_i}$ is a surjective map and $\theta_i\colon \im \varphi_{\alpha_i} \rightarrow M_{z_i}$ is an inclusion map. Define a $K$-representation $FM=((FM)_{x'}, \varphi'_{\xi'})_{x' \in Q'_0, \xi' \in Q'_1}$ as follows.
	\begin{align*}
		(FM)_{x'}&=
		\begin{cases}
			M_{x'} 					&	\text{if } x' \in Q_0	\\
			\im \varphi_{\alpha_i}	&	\text{if } x'\in\{p_i,q_i : 1 \leq i \leq |V_\chi|\}.
		\end{cases} \\
		\varphi'_{\xi'}&=
		\begin{cases}
			\varphi_{\xi'} 	&	\text{if } \xi' \in Q_1 \setminus \{\alpha_i\in Q_1 : 1 \leq i \leq |V_\chi|\}	\\
			\psi_i		&	\text{if } \xi' = \beta_i \text{ for some } i \\
			\theta_i		&	\text{if } \xi' = \delta_i \text{ for some } i \\
			\id			&	\text{if } \xi' = \eta_i \text{ for some } i \\
			0			&	\text{if } \xi' = \{\gamma_i,\zeta_i: 1 \leq i \leq |V_\chi|\}.
		\end{cases}
	\end{align*}
	
	The reason for defining the function $F$ on $\uMod*A$ is that $FP$ does not respect the relations of $A'$ for some projective-injective $P \in \Mod* A$. Specifically, if $x$ is a polygon in $\chi$ such that $x$ is incident to some vertex $t_i \in V_\chi$, then $F(P(x))$ is not a module in $\Mod* A'$. For any non-projective $A$-module $M$, note that $M$ is an $A'$-module, since the action of each arrow $\gamma_i$ on $M$ is zero, and thus $FM$ respects the relations in $A'$. This is made clearer in the following example. 
	
	\begin{exam} \label{IndFuncExam}
		Consider the Brauer configuration algebras $A$ and $A'$ associated to the following respective Brauer configurations.
		
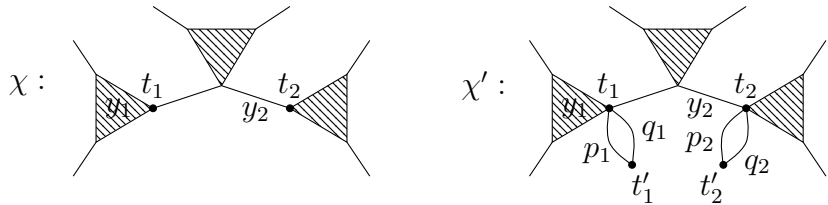
\begin{center}
			\begin{tikzpicture}[scale=1.5, baseline=(o.base)]
			\draw (-6,0.5) node {$\chi:$};
			\draw [fill=black] (-4.9,0.3) ellipse (0.03 and 0.03);
			\draw (-4.9,0.5) node {$t_1$};
			\draw [fill=black] (-3.7,0.3) ellipse (0.03 and 0.03);
			\draw (-3.7,0.5) node {$t_2$};
			
			\draw (-4.9,0.3) -- (-4.3,0.5) --(-3.7,0.3);
			\draw (-4,0.2665) node {$y_2$};
			
			\draw [pattern=north west lines](-4.3,0.5) -- (-4.6,1) -- (-4,1) -- (-4.3,0.5);
			\draw (-4.9,1.2) -- (-4.6,1);
			\draw (-3.7,1.2) -- (-4,1);
			\draw [pattern=north west lines](-5.4,0.6) -- (-4.9,0.3) -- (-5.4,0) -- (-5.4,0.6);
			\draw (-5.6,0.9) -- (-5.4,0.6);
			\draw (-5.6,-0.3) -- (-5.4,0);
			\draw (-5.2,0.3) node {$y_1$};
			\draw [pattern=north west lines](-3.7,0.3) -- (-3.2,0.6) -- (-3.2,0) -- (-3.7,0.3);
			\draw (-3,0.9) -- (-3.2,0.6);
			\draw (-3,-0.3) -- (-3.2,0);

			\draw (-2,0.5) node {$\chi':$};
			\draw [fill=black] (-0.9,0.3) ellipse (0.03 and 0.03);
			\draw (-0.9,0.5) node {$t_1$};
			\draw [fill=black] (-0.7,-0.2) ellipse (0.03 and 0.03);
			\draw (-0.6,-0.4) node {$t'_1$};
			\draw [fill=black] (0.3,0.3) ellipse (0.03 and 0.03);
			\draw (0.3,0.5) node {$t_2$};
			\draw [fill=black] (0.1,-0.2) ellipse (0.03 and 0.03);
			\draw (0,-0.4) node {$t'_2$};
			
			\draw (-0.9,0.3) -- (-0.3,0.5) --(0.3,0.3);
			\draw (-0.1,0.3) node {$y_2$};
			
			\draw [pattern=north west lines](-0.3,0.5) -- (-0.6,1) -- (0,1) -- (-0.3,0.5);
			\draw (-0.9,1.2) -- (-0.6,1);
			\draw (0.3,1.2) -- (0,1);
			\draw [pattern=north west lines](-1.4,0.6) -- (-0.9,0.3) -- (-1.4,0) -- (-1.4,0.6);
			\draw (-1.6,0.9) -- (-1.4,0.6);
			\draw (-1.6,-0.3) -- (-1.4,0);
			\draw (-1.2,0.3) node {$y_1$};
			\draw [pattern=north west lines](0.3,0.3) -- (0.8,0.6) -- (0.8,0) -- (0.3,0.3);
			\draw (1,0.9) -- (0.8,0.6);
			\draw (1,-0.3) -- (0.8,0);
			
			\draw  plot[smooth, tension=.7] coordinates {(-0.9,0.3) (-0.9,0) (-0.7,-0.2)};
			\draw (-1,-0.1) node {$p_1$};
			\draw  plot[smooth, tension=.7] coordinates {(-0.9,0.3) (-0.7,0.1) (-0.7,-0.2)};
			\draw (-0.5,0.1) node {$q_1$};
			
			\draw  plot[smooth, tension=.7] coordinates {(0.3,0.3) (0.1,0.1) (0.1,-0.2)};
			\draw (-0.1,0) node {$p_2$};
			\draw  plot[smooth, tension=.7] coordinates {(0.3,0.3) (0.3,0) (0.1,-0.2)};
			\draw (0.4,-0.2) node {$q_2$};

			\end{tikzpicture}
		\end{center}
		where $\mathfrak{e}_{t_1},\mathfrak{e}_{t_2}>1$ in $\chi$ and $\mathfrak{e}'\equiv 1$ in $\chi'$. Note that $\chi'$ is obtained from $\chi$ using the process outlined above (after Theorem~\ref{TameQC}) with the sets $V_\chi=\{t_1,t_2\}$ and $Y_\chi=\{y^{t_1}_1, y^{t_2}_2\}$. Consider the following family of $A$-modules given by $K$-representations
		\begin{center}
			\begin{tikzpicture}[scale=1.5, ]
			\draw (-1.1,0.2) node {$M_\lambda:$};
			\draw [->](-0.1,0.2) -- (-0.3,0.4);
			\draw [->](-0.4,0.3) -- (-0.2,0.1);
			\draw (-0.5,0.5) node {$0$};
			\draw [<-](-0.1,-0.2) -- (-0.3,-0.4);
			\draw [<-](-0.4,-0.3) -- (-0.2,-0.1);
			\draw (-0.5,-0.5) node {$0$};
			\draw (0,0) node {$K$};
			\draw [->](0.2,-0.05) -- (0.8,-0.05);
			\draw (0.5,-0.25) node {$\left(\begin{smallmatrix} 0 \\ \lambda \end{smallmatrix}\right)$};
			\draw [->](0.8,0.05) -- (0.2,0.05);
			\draw (0.5,0.25) node {$\left(\begin{smallmatrix} 1 & 0 \end{smallmatrix}\right)$};
			\draw (1,0) node {$K^2$};
			\draw [->](1.2,0) -- (1.8,0);
			\draw (1.5,-0.25) node {$\left(\begin{smallmatrix} 1 & 0 \\ 0 & 0 \end{smallmatrix}\right)$};
			\draw (2,0) node {$K^2$};
			\draw [->](1.85,0.15) -- (1.65,0.45);
			\draw (2.5,-0.25) node {$\left(\begin{smallmatrix} 0 & 1 \end{smallmatrix}\right)$};
			\draw (1.5,0.6) node {$K$};
			\draw [->](1.35,0.45) -- (1.15,0.15);
			\draw (1,0.4) node {$\left(\begin{smallmatrix} 0 \\ 1 \end{smallmatrix}\right)$};
			\draw [->](1.4,0.8) -- (1.2,1);
			\draw [->](1.1,0.9) -- (1.3,0.7);
			\draw (1,1.1) node {$0$};
			\draw [->](1.8,1) -- (1.6,0.8);
			\draw [->](1.7,0.7) -- (1.9,0.9);
			\draw (2,1.1) node {$0$};
			\draw [->](2.2,-0.05) -- (2.8,-0.05);
			\draw (2.1,0.4) node {$\left(\begin{smallmatrix} 0 & 1 \end{smallmatrix}\right)$};
			\draw [->](2.8,0.05) -- (2.2,0.05);
			\draw (2.6,0.25) node {$\left(\begin{smallmatrix} 1 \\ 0 \end{smallmatrix}\right)$};
			\draw (3,0) node {$K$};
			\draw [->](3.2,0.1) -- (3.4,0.3);
			\draw [->](3.3,0.4) -- (3.1,0.2);
			\draw (3.5,0.5) node {$0$};
			\draw [->](3.1,-0.2) -- (3.3,-0.4);
			\draw [->](3.4,-0.3) -- (3.2,-0.1);
			\draw (3.5,-0.5) node {$0$};
			\end{tikzpicture}
		\end{center}
		of $Q$, with $\lambda \in K$. (This is actually a family of band modules.) Note that the composition of the linear map $\left(\begin{smallmatrix} 1 & 0 \end{smallmatrix}\right)$ with $\varphi_{\alpha_1}=\left(\begin{smallmatrix} 0 \\ \lambda \end{smallmatrix}\right)$ is non-zero. This is permitted, since $\mathfrak{e}_{t_1}>1$. Let $F:\uMod*A \rightarrow \Mod*A'$ be the function defined above. Then we obtain a family of $A'$-modules given by the following $K$-representations
		\begin{center}
			\begin{tikzpicture}[scale=1.5, ]
			\draw (-1.1,0) node {$FM_\lambda:$};
			\draw [->](-0.1,0.2) -- (-0.3,0.4);
			\draw [->](-0.4,0.3) -- (-0.2,0.1);
			\draw (-0.5,0.5) node {$0$};
			\draw [<-](-0.1,-0.2) -- (-0.3,-0.4);
			\draw [<-](-0.4,-0.3) -- (-0.2,-0.1);
			\draw (-0.5,-0.5) node {$0$};
			\draw (0,0) node {$K$};
			\draw [->](0.8,0) -- (0.2,0);
			\draw (0.5,0.25) node {$\left(\begin{smallmatrix} 1 & 0 \end{smallmatrix}\right)$};
			\draw [->](0.1,-0.2) -- (0.2,-0.7);
			\draw (0.05,-0.5) node {$\begin{smallmatrix} \lambda \end{smallmatrix}$};
			\draw (0.2,-0.9) node {$K$};
			\draw [->] plot[smooth, tension=.7] coordinates {(0.4,-0.8) (0.6,-0.7) (0.8,-0.8)};
			\draw (0.6,-0.6) node {$\begin{smallmatrix} 0 \end{smallmatrix}$};
			\draw [->](0.8,-0.9) -- (0.4,-0.9);
			\draw (0.6,-0.82) node {$\begin{smallmatrix} 1 \end{smallmatrix}$};
			\draw [->] plot[smooth, tension=.7] coordinates {(0.4,-1) (0.6,-1.1) (0.8,-1)};
			\draw (0.6,-1) node {$\begin{smallmatrix} 0 \end{smallmatrix}$};
			\draw (1,-0.9) node {$K$};
			\draw [->](1,-0.7) -- (1,-0.2);
			\draw (0.8,-0.4) node {$\left(\begin{smallmatrix} 0 \\ 1 \end{smallmatrix}\right)$};
			\draw (1,0) node {$K^2$};
			\draw [->](1.2,0) -- (1.8,0);
			\draw (1.5,-0.25) node {$\left(\begin{smallmatrix} 1 & 0 \\ 0 & 0 \end{smallmatrix}\right)$};
			\draw (2,0) node {$K^2$};
			\draw [->](1.85,0.15) -- (1.65,0.45);
			\draw (1.7,-0.6) node {$\left(\begin{smallmatrix} 0 & 1 \end{smallmatrix}\right)$};
			\draw (1.5,0.6) node {$K$};
			\draw [->](1.35,0.45) -- (1.15,0.15);
			\draw (1,0.4) node {$\left(\begin{smallmatrix} 0 \\ 1 \end{smallmatrix}\right)$};
			\draw [->](1.4,0.8) -- (1.2,1);
			\draw [->](1.1,0.9) -- (1.3,0.7);
			\draw (1,1.1) node {$0$};
			\draw [->](1.8,1) -- (1.6,0.8);
			\draw [->](1.7,0.7) -- (1.9,0.9);
			\draw (2,1.1) node {$0$};
			\draw [->](2.8,0.05) -- (2.2,0.05);
			\draw (2.6,0.25) node {$\left(\begin{smallmatrix} 1 \\ 0 \end{smallmatrix}\right)$};
			\draw [->](2,-0.2) -- (2,-0.7);
			\draw (2.1,0.4) node {$\left(\begin{smallmatrix} 0 & 1 \end{smallmatrix}\right)$};
			\draw (2,-0.9) node {$K$};
			\draw [->] plot[smooth, tension=.7] coordinates {(2.2,-0.8) (2.4,-0.7) (2.6,-0.8)};
			\draw (2.4,-0.6) node {$\begin{smallmatrix} 0 \end{smallmatrix}$};
			\draw [->](2.6,-0.9) -- (2.2,-0.9);
			\draw (2.4,-0.82) node {$\begin{smallmatrix} 1 \end{smallmatrix}$};
			\draw [->] plot[smooth, tension=.7] coordinates {(2.2,-1) (2.4,-1.1) (2.6,-1)};
			\draw (2.4,-1) node {$\begin{smallmatrix} 0 \end{smallmatrix}$};
			\draw (2.8,-0.9) node {$K$};
			\draw [->](2.8,-0.7) -- (2.9,-0.2);
			\draw (2.95,-0.5) node {$\begin{smallmatrix} 1 \end{smallmatrix}$};
			\draw (3,0) node {$K$};
			\draw [->](3.2,0.1) -- (3.4,0.3);
			\draw [->](3.3,0.4) -- (3.1,0.2);
			\draw (3.5,0.5) node {$0$};
			\draw [->](3.1,-0.2) -- (3.3,-0.4);
			\draw [->](3.4,-0.3) -- (3.2,-0.1);
			\draw (3.5,-0.5) node {$0$};
			
			\end{tikzpicture}
		\end{center}
		In this setting, we have $\varphi_{\beta_1}=(\lambda)$, $\varphi_{\gamma_1}=0$ and $\varphi_{\delta_1}=\left(\begin{smallmatrix} 0 \\ 1 \end{smallmatrix}\right)$. Since $\mathfrak{e}'_{t_1}=1$, we require that the composition of $\left(\begin{smallmatrix} 1 & 0 \end{smallmatrix}\right)$ with $\varphi_{\delta_1}\varphi_{\gamma_1}\varphi_{\beta_1}$ be zero, which is indeed the case. A similar observation can be made with the vertex $t_2$.
	\end{exam}

	\begin{lem} \label{IndFunctor}
		For any indecomposable $M \in \uMod*A$, the module $FM \in \Mod*A'$ is indecomposable.
	\end{lem}
	\begin{proof}
		Let $M=(M_x, \varphi_\xi)_{x \in Q_0, \xi \in Q_1}$ be a $K$-representation of $Q$. We will prove the contrapositive statement. That is, we will show that if $FM$ is decomposable then $M$ is decomposable. So suppose $FM=((FM)_{x'}, \varphi'_{\xi'})_{x' \in Q'_0, \xi' \in Q'_1}$ is decomposable. Then there exists a non-trivial idempotent $f'=(f'_{x'})_{x' \in Q'_0} \in \End_{A'}(FM)$.
		
		For each $i$, there exist commutative squares
		\begin{equation} \tag{$\ast$} \label{ImageSquares}
			\xymatrix{
				M_{y_i} \ar[d]_{f'_{y_i}} \ar[r]^-{\psi_i}
				&	\im \varphi_{\alpha_i} \ar[d]_{f'_{p_i}}
				&	\ar[l]_-{\id} \im \varphi_{\alpha_i} \ar[r]^-{\theta_i} \ar[d]_{f'_{q_i}}
				&	M_{z_i} \ar[d]^{f'_{z_i}} \\
				M_{y_i} \ar[r]^-{\psi_i}
				&	\im \varphi_{\alpha_i}
				&	\ar[l]_-{\id} \im \varphi_{\alpha_i} \ar[r]^-{\theta_i}
				&	M_{z_i}
			}
		\end{equation}
		arising from the arrows $\beta_i$, $\eta_i$ and $\delta_i$ defined in the construction $F$. The middle square implies $f'_{p_i}=f'_{q_i}$. The leftmost and rightmost squares imply that $\psi_i f'_{y_i}= f'_{p_i} \psi_i$ and $\theta_i f'_{q_i}= f'_{z_i} \theta_i$. But then 
		\begin{equation*}
			\theta_i\psi_i f'_{y_i}= \theta_i f'_{p_i} \psi_i=\theta_i f'_{q_i} \psi_i = f'_{z_i} \theta_i \psi_i.
		\end{equation*}
		Recall that for each $i$, we have $\varphi_{\alpha_i}=\theta_i\psi_i$. So for each $i$, there exists a commutativity relation $\varphi_{\alpha_i} f'_{y_i}= f'_{z_i} \varphi_{\alpha_i}$. Since $\varphi'_\xi=\varphi_\xi$ for all arrows $\xi \in Q_1 \setminus \{\alpha_1,\ldots,\alpha_n\}$ by the definition of $FM$, we also have commutativity relations $\varphi_\xi f'_{s(\xi)} = f'_{e(\xi)} \varphi_\xi$ for any arrow $\xi \in Q_1 \setminus \{\alpha_1,\ldots,\alpha_n\}$. This implies the existence of a map $f=(f_x)_{x \in Q_0} \in \End_A (M)$ defined by $f_x=f'_x$ for each $x \in Q_0$. But $f'$ is idempotent, so $f$ must also be a idempotent.
		
		It remains to show that $f$ is non-trivial. Suppose for a contradiction that $f$ is trivial. Then either $f_x=0$ for all $x \in Q_0$ or $f_x$ is an identity map for all $x \in Q_0$. Suppose $f_{z_i}=0$. Then by the commutativity relations arising from the rightmost square in (\ref{ImageSquares}), we have $\theta_i f'_{q_i}= 0$. Since $\theta_i$ is injective, this implies $f'_{q_i}= 0$. But $f'_{q_i}=f'_{p_i}$, which implies $f'_{x'}=0$ for all $x' \in Q'_0$. This contradicts our original assumption that $f'$ is non-trivial. The same contradiction argument essentially works for the case where we assume that $f_x$ is an identity map for all $x \in Q_0$. Thus, $f$ must be a non-trivial idempotent in $\End_A (M)$ and hence, $M$ is decomposable.
	\end{proof}
	
	\begin{lem} \label{InjFunctor}
		For any $M,N \in \uMod*A$, we have $FM \cong FN$ if and only if $M \cong N$.
	\end{lem}
	\begin{proof}
		Let
		\begin{align*}
			M&=(M_x, \varphi_\xi)_{x \in Q_0, \xi \in Q_1}, &
			N&=(N_x, \phi_\xi)_{x \in Q_0, \xi \in Q_1}, \\
			FM&=((FM)_{x'}, \varphi'_{\xi'})_{x' \in Q'_0, \xi' \in Q'_1}, &
			FN&=((FN)_{x'}, \phi'_{\xi'})_{x' \in Q'_0, \xi' \in Q'_1}.
		\end{align*}
		For each $i$, write $\phi_{\alpha_i}=\theta'_i\psi'_i$, where $\psi'_i\colon N_{y_i} \rightarrow \im \phi_{\alpha_i}$ is a surjective map and $\theta'_i\colon \im \phi_{\alpha_i} \rightarrow N_{z_i}$ is an inclusion map. The maps $\psi_i$ and $\theta_i$ are as defined in the construction of $F$.
		
		$(\Rightarrow:)$ Let $f'=(f'_{x'})_{x' \in Q'_0} \in \Hom_{A'}(FM,FN)$ be an isomorphism. Then for each $i$, we have commutative squares
		\begin{equation*}
			\xymatrix{
				M_{y_i} \ar[d]_{f'_{y_i}} \ar[r]^-{\psi_i}
				&	\im \varphi_{\alpha_i} \ar[d]_{f'_{p_i}}
				&	\ar[l]_-{\id} \im \varphi_{\alpha_i} \ar[r]^-{\theta_i} \ar[d]_{f'_{q_i}}
				&	M_{z_i} \ar[d]^{f'_{z_i}} \\
				N_{y_i} \ar[r]^-{\psi'_i}
				&	\im \phi_{\alpha_i}
				&	\ar[l]_-{\id} \im \phi_{\alpha_i} \ar[r]^-{\theta'_i}
				&	N_{z_i}
			}
		\end{equation*}
		arising from the arrows $\beta_i$, $\eta_i$ and $\delta_i$ defined in the construction $F$, where each map $f'_{x'}$ is a bijection. By composing maps, we obtain a commutative square 
		\begin{equation*}
			\xymatrix{
				M_{y_i} \ar[d]_{f'_{y_i}} \ar[r]^-{\varphi_{\alpha_i}}
				&	M_{z_i} \ar[d]^{f'_{z_i}} \\
				N_{y_i} \ar[r]^-{\phi_{\alpha_i}}	
				&	N_{z_i}.
			}
		\end{equation*}
		with $f'_{y_i}$ and $f'_{z_i}$ bijections. Since $(FM)_x=M_x$ for any $x \in Q_0$ and $\varphi'_\xi=\varphi_\xi$ for any $\xi \in Q_1 \setminus \{\alpha_i\in Q_1 : 1 \leq i \leq n\}$, we conclude that there exists an isomorphism $f=(f_x)_{x \in Q_0} \in \Hom_{A}(M,N)$ defined by $f_x=f'_x$ for all $x \in Q_0$.
		
		$(\Leftarrow:)$ Let $f=(f_x)_{x \in Q_0} \in \Hom_{A}(M,N)$ be an isomorphism. Consider the following commutative diagram.
		\begin{equation*}
			\xymatrix@R=1ex{
				M_{y_i}
					\ar@<-0.5ex>[ddd]_-{f_{y_i}}
					\ar[rr]^-{\varphi_{\alpha_i}}
					\ar[dr]_-{\psi_i}	
				&	&	M_{z_i}
					\ar@<-0.5ex>[ddd]_-{f_{z_i}}
				\\
				&	\im \varphi_{\alpha_i}	\ar[ur]_-{\theta_i}	&	\\
				&	\im \phi_{\alpha_i}	\ar[dr]^-{\theta'_i}	&	\\
				N_{y_i}
					\ar@<-0.5ex>[uuu]_-{f^{-1}_{y_i}}
					\ar[rr]^-{\phi_{\alpha_i}}
					\ar[ur]^-{\psi'_i}	
				&	&	N_{z_i}
					\ar@<-0.5ex>[uuu]_-{f^{-1}_{z_i}}
				\\	
			}
		\end{equation*}
		Note that $\varphi_{\alpha_i}=f^{-1}_{z_i}\theta'_i\psi'_i f_{y_i}$ and that the map $f^{-1}_{z_i}\theta'_i$ is a monomorphism. By the universal property of images, there exists a unique morphism $g_i:\im \varphi_{\alpha_i}\rightarrow \im\phi_{\alpha_i}$ such that $\theta_i=f^{-1}_{z_i}\theta'_i g_i$. A similar argument shows that there exists a unique morphism $h_i:\im\phi_{\alpha_i} \rightarrow \im \varphi_{\alpha_i}$ such that $\theta'_i=f_{z_i}\theta_i h_i$. From this, we obtain
		\begin{align*}
			\theta_i&=f^{-1}_{z_i}(f_{z_i}\theta_i h_i) g_i=\theta_i h_i g_i \text{, and} \\
			\theta'_i&=f_{z_i}(f^{-1}_{z_i}\theta'_i g_i) h_i = \theta'_i g_i h_i
		\end{align*}
		Since $\theta_i$ and $\theta'_i$ are monomorphisms, we obtain $h_i g_i = \id$ and $g_i h_i= \id$. Thus, $g_i$ is an isomorphism. Moreover, $f_{z_i} \theta_i=\theta'_i g_i$, and thus we also have
		\begin{equation*}
		 	\theta'_i g_i \psi_i = f_{z_i} \theta_i \psi_i=f_{z_i} \varphi_{\alpha_i} = \phi_{\alpha_i} f_{y_i}=\theta'_i \psi'_i f_{y_i}.
		 \end{equation*}
		Since $\theta'_i$ is a monomorphism, $g_i \psi_i = \psi'_i f_{y_i}$. Hence for each $i$, the sequence of squares 
		\begin{equation*}
			\xymatrix{
				M_{y_i} \ar[d]_{f_{y_i}} \ar[r]^-{\psi_i}
				&	\im \varphi_{\alpha_i} \ar[d]_{g_i}
				&	\ar[l]_-{\id} \im \varphi_{\alpha_i} \ar[r]^-{\theta_i} \ar[d]_{g_i}
				&	M_{z_i} \ar[d]^{f_{z_i}} \\
				N_{y_i} \ar[r]^-{\psi'_i}
				&	\im \phi_{\alpha_i}
				&	\ar[l]_-{\id} \im \phi_{\alpha_i} \ar[r]^-{\theta'_i}
				&	N_{z_i}
			}
		\end{equation*}
		is commutative with each map $f_x$ and each map $g_i$ a bijection. This determines an isomorphism in $\Hom_{A'}(FM,FN)$ in the natural way.
	\end{proof}
	
	\begin{rem} \label{KASimples}
		Let $\lambda \in K$. Recall that $S_\lambda$ is a simple (right) $K[a]$-module if and only if $\dim_K S_\lambda=1$ and the action of $a$ on $S_\lambda$ is defined by $s a = \lambda s$ for all $s \in S_\lambda$.
	\end{rem}
	
	\begin{lem} \label{FamFunctor}
		Let $M\in \uMod*A$. Suppose there exists a $K[a]$-$A'$-bimodule $N'$ that is finitely generated and free as a left $K[a]$-module such that $FM \cong S \otimes_{K[a]} N'$ for some simple right $K[a]$-module $S$. Then there exists a $K[a]$-$A$-bimodule $N$ that is finitely generated and free as a left $K[a]$-module such that $M \cong S \otimes_{K[a]} N$. Moreover, if there exists some other module $M'\in \uMod*A$ such that $FM' \cong S' \otimes_{K[a]} N'$ where $S'\not\cong S$ is a simple right $K[a]$-module, then $M' \cong S' \otimes_{K[a]} N$.
	\end{lem}
	\begin{proof}
		For the purposes of readability, we will denote each subspace $N'\varepsilon_{x'} \subset N'$ by $N'_{x'}$. For each arrow $\xi' \in Q'_1$, define a linear map $\widehat{\xi}'\colon N'\rightarrow N'$ by $\widehat{\xi}'(c)=c\xi'$. For some vector subspace $V \subset N'$, we also denote the canonical projection map from $N'$ onto $V$ by $\pi_V$.
		
		Let $f'=(f_{x'})_{x' \in Q'_0}:FM \rightarrow S \otimes_{K[a]} N'$ be a bijection. Note that since $S \otimes_{K[a]} N' \cong FM$ for some $M \in \uMod*A$, the map $\pi_{N'_{p_i}}\widehat{\eta}_i|_{N'_{q_i}}$ is a bijection. This follows (for each $i$) from the commutativity of the square
		\begin{equation*}
			\xymatrix{
				(FM)_{p_i} \ar@{<-}[r]^-{\id} \ar[d]_{f'_{p_i}} & (FM)_{q_i} \ar[d]^{f'_{q_i}} \\
				(S \otimes_{K[a]} N')_{p_i} \ar@{<-}[r]^-{\phi'_{\eta_i}} & (S\otimes_{K[a]} N')_{q_i},
			}
		\end{equation*}
		where $\phi'_{\eta_i}: (S \otimes_{K[a]} N')_{q_i} \rightarrow (S \otimes_{K[a]} N')_{p_i}$ is the linear map defined by $\phi'_{\eta_i}(b \otimes c)= b \otimes \pi_{N'_{p_i}}\widehat{\eta}_i|_{N'_{q_i}}(c)$ and $f'_{p_i}$ and $f'_{q_i}$ are bijections. Define $\sigma: N'_{p_i} \rightarrow N'_{q_i}$ to be the inverse linear map to $\pi_{N'_{p_i}}\widehat{\eta}_i|_{N'_{q_i}}$.
		
		We now define a $K[a]$-$A$ bimodule $N$ associated to $N'$. The underlying vector space of $N$ is then given by
		\begin{equation*}
			N= \bigoplus_{x \in Q_0} N'_x,
		\end{equation*}
		which is a vector subspace of $N'$. We give $N$ the structure of a right $A$-module in the following way. We define the action of an arrow $\xi \in Q_1 \setminus \{\alpha_1, \ldots, \alpha_{|V_\chi|}\}$ on $N$ to be given by $c\xi=\widehat{\xi}|_N(c)$ for each $c \in N$. The action of each arrow $\alpha_i$ on $N$ to is given by $c\alpha_i=\pi_N\widehat{\delta}_i|_{N_{q_i}}\sigma\pi_{N_{p_i}}\widehat{\beta}_i|_N(c)$. That $N$ is also a finitely generated and free left $K[A]$-module follows from the fact that $N'$ is finitely generated and free.
		
		To show that $M \cong S \otimes_{K[a]} N$, we note that for each arrow $\xi' \in Q'_1$, there exists a commutative square 
		\begin{equation} \tag{$\dagger$} \label{TensorComSquare}
			\xymatrix{
				(FM)_{s(\xi')} \ar[r]^-{\varphi'_{\xi'}} \ar[d]_{f'_{s(\xi')}} & (FM)_{e(\xi')} \ar[d]^{f'_{e(\xi')}} \\
				(S \otimes_{K[a]} N')_{s(\xi')} \ar[r]^-{\phi'_{\xi'}} & (S\otimes_{K[a]} N')_{e(\xi')},
			}
		\end{equation}
		where the maps $f'_{s(\xi')}$ and $f'_{e(\xi')}$ are bijective and $\phi'_{\xi'}(b \otimes c)= b \otimes \pi_{N'_{e(\xi')}}\widehat{\xi}'|_{N'_{s(\xi')}}(c)$. If $\xi \in Q_1 \setminus \{\alpha_1,\ldots,\alpha_{|V_\chi|}\}$ then the square (\ref{TensorComSquare}) is the same as the commutative square
		\begin{equation*}
			\xymatrix{
				M_{s(\xi)} \ar[r]^-{\varphi_{\xi}} \ar[d]_{f'_{s(\xi)}} & M_{e(\xi)} \ar[d]^{f'_{e(\xi)}} \\
				(S \otimes_{K[a]} N)_{s(\xi)} \ar[r]^-{\phi_{\xi}} & (S\otimes_{K[a]} N)_{e(\xi)}.
			}
		\end{equation*}
		where $\phi_{\xi}(b \otimes c)= b \otimes \pi_{N_{e(\xi)}}\widehat{\xi}|_{N_{s(\xi)}}(c)$. It is easy to see that for each $i$, the square
		\begin{equation*}
			\xymatrix{
				M_{y_i} \ar[r]^-{\varphi_{\alpha_i}} \ar[d]_{f'_{y_i}} & M_{z_i} \ar[d]^{f'_{z_i}} \\
				(S \otimes_{K[a]} N)_{y_i} \ar[r]^-{\phi_{\alpha_i}} & (S\otimes_{K[a]} N)_{z_i}.
			}
		\end{equation*}
		is commutative, since we have each $\alpha_i=\theta_i\psi_i$ and a sequence of commutative squares
		\begin{equation*}
			\xymatrix{
				M_{y_i} \ar[d]_{f'_{y_i}} \ar[r]^-{\psi_i}
				&	\im \varphi_{\alpha_i} \ar[d]_{f'_{p_i}}
				&	\ar[l]_-{\id} \im \varphi_{\alpha_i} \ar[r]^-{\theta_i} \ar[d]_{f'_{q_i}}
				&	M_{z_i} \ar[d]^{f'_{z_i}} \\
				(S \otimes_{K[a]} N')_{y_i} \ar[r]^-{\phi'_{\beta_i}}
				&	(S \otimes_{K[a]} N')_{p_i}
				&	\ar[l]_-{\phi'_{\eta_i}} (S \otimes_{K[a]} N')_{q_i} \ar[r]^-{\phi'_{\delta_i}}
				&	(S \otimes_{K[a]} N')_{z_i}
			}
		\end{equation*}
		which implies a commutativity relation 
		\begin{equation*}
			\phi'_{\delta_i}\phi'^{-1}_{\eta_i}\phi'_{\beta_i} f'_{y_i}= \phi_{\alpha_i}f'_{y_i}= f'_{z_i} \varphi_{\alpha_i}= f'_{z_i} \theta_i \psi_i.
		\end{equation*}
		Thus, the map $(f'_x)_{x \in Q_0}:M \rightarrow S \otimes_{K[a]} N$ is an isomorphism.
		
		The statement that $M \cong S \otimes_{K[a]} N$ and $M' \cong S' \otimes_{K[a]} N$ whenever $FM \cong S \otimes_{K[a]} N'$ and $FM' \cong S' \otimes_{K[a]} N'$ for any simple $S' \not\cong S$ follows from the fact that the construction of $N$ depends only on the bimodule $N'$ and not on the choice of simple in the tensor product.
	\end{proof}
	
	\begin{exam}
		Consider the Brauer configuration algebra and function $F$ from Example~\ref{IndFuncExam}. In particular, consider the family of $A'$-modules $FM_\lambda$. It is easy to see that each $FM_\lambda \cong S_\lambda \otimes_{K[a]} N'$, where $S_\lambda$ is as in Remark~\ref{KASimples} and $N'$ is the $K[a]$-$A'$ bimodule defined by the following $K$-representation.
		\begin{center}
			\begin{tikzpicture}[scale=1.5, ]
			\draw (-2.1,0) node {$N':$};
			\draw [->](-1.1,0.2) -- (-1.3,0.4);
			\draw [->](-1.4,0.3) -- (-1.2,0.1);
			\draw (-1.5,0.5) node {$0$};
			\draw [<-](-1.1,-0.2) -- (-1.3,-0.4);
			\draw [<-](-1.4,-0.3) -- (-1.2,-0.1);
			\draw (-1.5,-0.5) node {$0$};
			\draw (-0.8,0) node {$K[a]$};
			\draw [->](0.2,0) -- (-0.4,0);
			\draw (-0.1,0.25) node {$\left(\begin{smallmatrix} 1 & 0 \end{smallmatrix}\right)$};
			\draw [->](-0.6,-0.2) -- (-0.4,-0.7);
			\draw (-0.6,-0.5) node {$\begin{smallmatrix} a \end{smallmatrix}$};
			\draw (-0.3,-0.9) node {$K[a]$};
			\draw [->] plot[smooth, tension=.7] coordinates {(0,-0.8) (0.2,-0.7) (0.4,-0.8)};
			\draw (0.2,-0.6) node {$\begin{smallmatrix} 0 \end{smallmatrix}$};
			\draw [->](0.4,-0.9) -- (0,-0.9);
			\draw (0.2,-0.82) node {$\begin{smallmatrix} 1 \end{smallmatrix}$};
			\draw [->] plot[smooth, tension=.7] coordinates {(0,-1) (0.2,-1.1) (0.4,-1)};
			\draw (0.2,-1) node {$\begin{smallmatrix} 0 \end{smallmatrix}$};
			\draw (0.7,-0.9) node {$K[a]$};
			\draw [->](0.7,-0.7) -- (0.7,-0.2);
			\draw (0.5,-0.4) node {$\left(\begin{smallmatrix} 0 \\ 1 \end{smallmatrix}\right)$};
			\draw (0.7,0) node {$(K[a])^2$};
			\draw [->](1.2,0) -- (1.8,0);
			\draw (1.5,-0.25) node {$\left(\begin{smallmatrix} 1 & 0 \\ 0 & 0 \end{smallmatrix}\right)$};
			\draw (2.3,0) node {$(K[a])^2$};
			\draw [->](1.85,0.15) -- (1.65,0.45);
			\draw (2,-0.6) node {$\left(\begin{smallmatrix} 0 & 1 \end{smallmatrix}\right)$};
			\draw (1.5,0.65) node {$K[a]$};
			\draw [->](1.35,0.45) -- (1.15,0.15);
			\draw (1,0.4) node {$\left(\begin{smallmatrix} 0 \\ 1 \end{smallmatrix}\right)$};
			\draw [->](1.4,0.95) -- (1.2,1.15);
			\draw [->](1.1,1.05) -- (1.3,0.85);
			\draw (1,1.25) node {$0$};
			\draw [->](1.8,1.15) -- (1.6,0.95);
			\draw [->](1.7,0.85) -- (1.9,1.05);
			\draw (2,1.25) node {$0$};
			\draw [->](3.35,0) -- (2.75,0);
			\draw (3.05,0.25) node {$\left(\begin{smallmatrix} 1 \\ 0 \end{smallmatrix}\right)$};
			\draw [->](2.3,-0.2) -- (2.3,-0.7);
			\draw (2.1,0.4) node {$\left(\begin{smallmatrix} 0 & 1 \end{smallmatrix}\right)$};
			\draw (2.3,-0.9) node {$K[a]$};
			\draw [->] plot[smooth, tension=.7] coordinates {(2.6,-0.8) (2.8,-0.7) (3,-0.8)};
			\draw (2.8,-0.6) node {$\begin{smallmatrix} 0 \end{smallmatrix}$};
			\draw [->](3,-0.9) -- (2.6,-0.9);
			\draw (2.8,-0.82) node {$\begin{smallmatrix} 1 \end{smallmatrix}$};
			\draw [->] plot[smooth, tension=.7] coordinates {(2.6,-1) (2.8,-1.1) (3,-1)};
			\draw (2.8,-1) node {$\begin{smallmatrix} 0 \end{smallmatrix}$};
			\draw (3.3,-0.9) node {$K[a]$};
			\draw [->](3.3,-0.7) -- (3.55,-0.2);
			\draw (3.55,-0.5) node {$\begin{smallmatrix} 1 \end{smallmatrix}$};
			\draw (3.65,0) node {$K[a]$};
			\draw [->](3.95,0.2) -- (4.15,0.4);
			\draw [->](4.05,0.5) -- (3.85,0.3);
			\draw (4.25,0.6) node {$0$};
			\draw [->](3.85,-0.3) -- (4.05,-0.5);
			\draw [->](4.15,-0.4) -- (3.95,-0.2);
			\draw (4.25,-0.6) node {$0$};
			
			\end{tikzpicture}
		\end{center}
		Note that we have used a slight abuse of notation to define $N'$, since the vector spaces at each vertex in the $K$-representation above are in fact left $K[a]$-modules. The $K[a]$-$A$-bimodule $N$ corresponding to $N'$ from the construction in the proof of Lemma~\ref{FamFunctor} is then the following.
		\begin{center}
			\begin{tikzpicture}[scale=1.5, ]
			\draw (-2.1,0) node {$N:$};
			\draw [->](-1.1,0.2) -- (-1.3,0.4);
			\draw [->](-1.4,0.3) -- (-1.2,0.1);
			\draw (-1.5,0.5) node {$0$};
			\draw [<-](-1.1,-0.2) -- (-1.3,-0.4);
			\draw [<-](-1.4,-0.3) -- (-1.2,-0.1);
			\draw (-1.5,-0.5) node {$0$};
			\draw (-0.8,0) node {$K[a]$};
			\draw [->](0.2,0.05) -- (-0.4,0.05);
			\draw (-0.1,0.25) node {$\left(\begin{smallmatrix} 1 & 0 \end{smallmatrix}\right)$};
			\draw [->](-0.4,-0.05) -- (0.2,-0.05);
			\draw (-0.1,-0.25) node {$\left(\begin{smallmatrix} 0 \\ a \end{smallmatrix}\right)$};
			\draw (0.7,0) node {$(K[a])^2$};
			\draw [->](1.2,0) -- (1.8,0);
			\draw (1.5,-0.25) node {$\left(\begin{smallmatrix} 1 & 0 \\ 0 & 0 \end{smallmatrix}\right)$};
			\draw (2.3,0) node {$(K[a])^2$};
			\draw [->](1.85,0.15) -- (1.65,0.45);
			\draw (1.5,0.65) node {$K[a]$};
			\draw [->](1.35,0.45) -- (1.15,0.15);
			\draw (1,0.4) node {$\left(\begin{smallmatrix} 0 \\ 1 \end{smallmatrix}\right)$};
			\draw [->](1.4,0.95) -- (1.2,1.15);
			\draw [->](1.1,1.05) -- (1.3,0.85);
			\draw (1,1.25) node {$0$};
			\draw [->](1.8,1.15) -- (1.6,0.95);
			\draw [->](1.7,0.85) -- (1.9,1.05);
			\draw (2,1.25) node {$0$};
			\draw [->](3.35,0.05) -- (2.75,0.05);
			\draw (3.05,0.25) node {$\left(\begin{smallmatrix} 1 \\ 0 \end{smallmatrix}\right)$};
			\draw (2.1,0.4) node {$\left(\begin{smallmatrix} 0 & 1 \end{smallmatrix}\right)$};
			
			\draw [->](2.75,-0.05) -- (3.35,-0.05);
			\draw (3.05,-0.25) node {$\left(\begin{smallmatrix} 0 & 1 \end{smallmatrix}\right)$};
			\draw (3.65,0) node {$K[a]$};
			\draw [->](3.95,0.2) -- (4.15,0.4);
			\draw [->](4.05,0.5) -- (3.85,0.3);
			\draw (4.25,0.6) node {$0$};
			\draw [->](3.85,-0.3) -- (4.05,-0.5);
			\draw [->](4.15,-0.4) -- (3.95,-0.2);
			\draw (4.25,-0.6) node {$0$};
			
			\end{tikzpicture}
		\end{center}
		It then follows that each $M_\lambda \cong S_\lambda \otimes_{K[a]} N$.
	\end{exam}
	
	We now prove Theorem~\ref{TameQC}. Throughout the proof, we denote by $\uind_d A$ the full subcategory of $\uMod*A$ whose objects are indecomposable modules $M$ with $\dim_K M=d$. We denote by $\ind_{\leq d'} A'$ the full subcategory of $\Mod*A'$ whose objects are indecomposable modules $M'$ with $\dim_K M' \leq d$.
	\begin{proof}
		Let $V_\chi$ and $Y_\chi$ be the sets, $A'$ be the Brauer configuration algebra associated to the Brauer configuration $\chi'$, and let $F$ be the function outlined in the construction above.
		
		Let $M$ be a non-projective $A$-module of dimension $d$. Note that for each $1\leq i \leq |V_\chi|$,
		\begin{equation*}
			\dim_K ((FM)\varepsilon_{p_i})=\dim_K ((FM)\varepsilon_{q_i}) \leq \mathfrak{e}_{t_i},
		\end{equation*}
		where $t_i \in V_\chi$. Since $\dim_K ((FM)\varepsilon_{x}) = \dim_K (M\varepsilon_{x})$ for all $x \in Q_0$, we conclude that $\dim_K FM \leq \Delta(d)$, where $\Delta:\mathbb{Z}_{> 0} \rightarrow \mathbb{Z}_{> 0}$ is a function defined by
		\begin{equation*}
			\Delta(n)= n + 2 \sum_{i=1}^{|V_\chi|} \mathfrak{e}_{t_i}
		\end{equation*}

		Recall that since $A'$ is tame, there exists a finite number of $K[a]$-$A'$-bimodules $N'_1, \ldots, N'_{r_d}$ that are finitely generated and free as a left $K[a]$-modules such that almost all modules $M' \in \ind_{\leq \Delta(d)} A'$ are isomorphic to a module of the form $S \otimes_{K[a]} N'_k$, where $S$ is some simple right $K[a]$-module and $k \in \{1,\ldots, r_d\}$. But $FM$ is indecomposable for any indecomposable $M \in \uMod*A$ (Lemma~\ref{IndFunctor}). Thus, any indecomposable $A$-module $M$ of dimension $d$ belongs to precisely one of the following sets.
		\begin{equation*}
			B_{d}^{(i)}=\{M \in \uind_d A : FM \cong S \otimes_{K[a]} N'_i \text{ for some simple } S \in \fin K[a]\}
		\end{equation*}
		for some $i \in \{1,\ldots, r_d\}$, or to the set
		\begin{equation*}
			O_d=\{M \in \uind_d A : M \not\in B_d^{(i)} \text{ for any } i\},
		\end{equation*}
		or to the set
		\begin{equation*}
			P_d=\{M \in \Mod* A : \dim_K M = d \text{ and } M \text{ indecomposable projective}\}.
		\end{equation*}
		
		First note that the sets $O_d$ and $P_d$ are finite. The claim that $O_d$ is finite follows from the fact that since $A'$ is tame, there are only finitely many modules $M' \in \ind_{\leq \Delta(d)} A'$ such that $M' \not\cong S \otimes_{K[a]} N'_i$ for any simple $S \in \fin K[a]$ and any $i$. But since $F$ is injective on isomorphism classes (by Lemma~\ref{InjFunctor}), this implies there exist only finitely many $M\in \uind_d A$ such that $FM \not\cong S \otimes_{K[a]} N'_i$ for any simple $S \in \fin K[a]$ and any $i$. The set $P_d$ is finite (or empty) because $A$ is finite dimensional.
		
		Lemma~\ref{FamFunctor} implies that any $M_\lambda \in B_{d}^{(k)}$ is isomorphic to a module of the form $S_\lambda \otimes_{K[a]} N_k$ for some simple $K[a]$-module $S_\lambda$, where $N_k$ is a $K[a]$-$A$-bimodule (corresponding to $N'_k$) which is finitely generated and free as a left $K[a]$-module.
		
		Define a set
		\begin{equation*}
			J_d =\{j : B_{d}^{(j)} \neq \emptyset\}.
		\end{equation*}
		Then it follows that there exists a finite set $\{N_j : j \in J_d\}$ consisting of $K[a]$-$A$-bimodules that are finitely generated and free as a left $K[a]$-modules such that almost all indecomposable modules $M\in \Mod*A$ of dimension $d$ are isomorphic to a module of the form $S \otimes_{K[a]} N_j$, where $S$ is some simple right $K[a]$-module and $j \in J_d$. Thus, $A$ is tame.
	\end{proof}
	
	\section{Symmetric Special Multiserial Algebras of Type $\mathbb{E}$} \label{TriserialResults}
	Throughout this section, we will assume that $A$ is a symmetric special triserial algebra. That is, $A$ is a Brauer configuration algebra associated to a configuration $\chi$ such that for any polygon $x$ in $\chi$, we have $|x| \leq 3$. We aim to prove the following Theorem.
	\begin{thm} \label{ExceptionalAlgebras}
		Let $A=KQ/I$ be a Brauer configuration algebra associated to a Brauer configuration $\chi$. Suppose $\chi$ is of the form
		
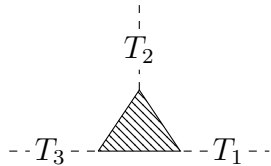
\begin{center}
			\begin{tikzpicture}[baseline = (o.base),scale=0.9]
				\draw[pattern = north west lines] (0.2,-1) -- (1.4,-1) -- (0.8,-0.1) -- (0.2,-1);
				\draw[dashed] (1.8,-1) -- (1.4,-1);
				\draw[dashed] (2.4,-1) -- (2.7,-1);
				\draw (2.1,-1) node{$T_1$};
				\draw[dashed] (0.8,-0.1) -- (0.8,0.2);
				\draw[dashed] (0.8,0.8) -- (0.8,1.2) node (o) {};
				\draw (0.8,0.5) node{$T_2$};
				\draw[dashed] (-1.1,-1) -- (-0.8,-1);
				\draw[dashed] (-0.2,-1) -- (0.2,-1);
				\draw (-0.5,-1) node{$T_3$};
			\end{tikzpicture}
		\end{center}
		where $T_1$, $T_2$ and $T_3$ are distinct multiplicity-free Brauer trees containing $m_1$, $m_2$ and $m_3$ edges respectively. Suppose further that at least two of $m_1$, $m_2$ and $m_3$ are strictly greater than 1. Then $A$ is tame if and only if the values of the (unordered) triple $(m_1,m_2,m_3)$ conform to a column of the following table.
		\begin{center}
			\begin{tabular}{c | c c c c c c}
				$m_1$	&	$1$	&	$1$	&	$1$	&	$1$	&	$1$	&	$2$	\\	\hline
				$m_2$	&	$2$	&	$2$	&	$2$	&	$2$	&	$3$	&	$2$	\\	\hline
				$m_3$	&	$2$	&	$3$	&	$4$	&	$5$	&	$3$	&	$2$
			\end{tabular}
		\end{center}
	\end{thm}
	
	The columns of the above table correspond to the Dynkin and Euclidean diagrams of type $\mathbb{E}_p$ and $\widetilde{\mathbb{E}}_p$ ($p \in\{6,7,8\}$). The first step of the proof is to show that any algebra of the above form is derived equivalent to a Brauer configuration algebra $\widetilde{A}$ associated to a Brauer configuration of the form 
	
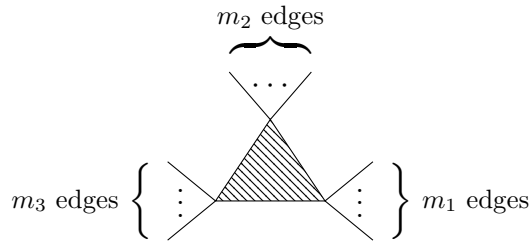
\begin{center}
		\begin{tikzpicture}[scale=0.9]
			\draw[pattern = north west lines] (0,-1) -- (1.6,-1) -- (0.8,0.2) -- (0,-1);
			
			\draw (2.3,-1.5778) -- (1.6,-1) -- (2.3,-0.4222);
			\draw (1.4,0.9) -- (0.8,0.2) -- (0.2,0.9);
			\draw (-0.7,-0.4222) -- (0,-1) -- (-0.7,-1.5778);
			
			\draw (2.1111,-0.8889) node {$\vdots$};
			\draw (0.8333,0.7) node {$\cdots$};
			\draw (-0.5111,-0.8889) node {$\vdots$};
			
			\draw (3.8,-1) node {\footnotesize$m_1$ edges};
			\draw (0.8,1.7) node {\footnotesize$m_2$ edges};
			\draw (-2.2,-1) node {\footnotesize$m_3$ edges};
			
			\draw (2.6,-1) node {$\left.\begin{smallmatrix} \\ \\ \\ \\ \\ \end{smallmatrix} \right\}$};
			\draw (0.8,1) node {$\overbrace{\begin{matrix} & & & \end{matrix}}$};
			\draw (-1,-1) node {$\left\{\begin{smallmatrix} \\ \\ \\ \\ \\ \end{smallmatrix} \right.$};
		\end{tikzpicture}
	\end{center}
	This is essentially Rickard's Brauer Star Theorem \cite[Theorem 4.2]{Rickard} applied to Brauer configuration algebras. By the results of \cite{AlmostGentle}, we then know that $\widetilde{A}$ is the trivial extension of a hereditary algebra $KQ'$, where $Q'$ is a quiver of the form
	
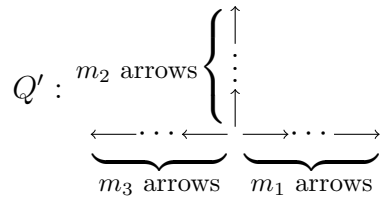
\begin{center}
		\begin{tikzpicture}
		\draw(-2.6,0.6) node {$Q':$};
		
		\draw [->](0.1,0) -- (0.7,0);
		\draw (1,0) node {$\cdots$};
		\draw [->](1.3,0) -- (1.9,0);
		
		\draw [->](0,0.1) -- (0,0.6);
		\draw (0,1) node {$\vdots$};
		\draw [->](0,1.2) -- (0,1.7);
		
		\draw [<-](-1.9,0) -- (-1.3,0);
		\draw (-1,0) node {$\cdots$};
		\draw [<-](-0.7,0) -- (-0.1,0);
		
		\draw (1,-0.1) node {$\underbrace{\begin{matrix} & & & & & \end{matrix}}$};
		\draw (-0.2,0.9) node {$\left\{ \begin{smallmatrix} \\ \\ \\ \\ \\ \\ \\ \end{smallmatrix} \right.$};
		\draw (-1,-0.1) node {$\underbrace{\begin{matrix} & & & & & \end{matrix}}$};
		
		\draw (-1.3,0.8) node {\footnotesize$m_2$ arrows};
		\draw (1,-0.7) node {\footnotesize$m_1$ arrows};
		\draw (-1,-0.7) node {\footnotesize$m_3$ arrows};
		\end{tikzpicture}
	\end{center}
	Under the assumption that at least two of $m_1$, $m_2$ and $m_3$ are strictly greater than 1, $\widetilde{A}=T(KQ')$ is tame if and only if the triple $(m_1,m_2,m_3)$ conforms to a column in the table of Theorem~\ref{ExceptionalAlgebras}, in which case, $Q'$ is an orientation of $\mathbb{E}_p$ or $\widetilde{\mathbb{E}}_p$ ($p \in\{6,7,8\}$). In fact, the algebra $\widetilde{A}$ is either of finite representation type (\cite[Theorem 1.4]{FiniteTrivialExtension}) or is representation-infinite domestic (\cite{DomesticTrivialExtension}, \cite{SkowDomestic}).
	
	Many details of the proof for \cite[Theorem 4.2]{Rickard} carry over to the multiserial case. However, for the benefit of the reader, we will run through the full details of the proof here.
	
	\subsection{Initial assumptions}
	We will make the following assumption in the construction of the tilting complex that follows.
	\begin{assumption} \label{ComplexAssum}
		Let $A=KQ/I$ be any Brauer configuration algebra associated to a Brauer configuration $\chi$. Assume $\chi$ is a multiplicity-free tree that contains precisely one $3$-gon and no $n$-gons with $n>3$.
	\end{assumption}
	 We will use this assumption to prove the following.
	
	\begin{prop} \label{DerivedReduction}
		Let $A$ be a Brauer configuration algebra associated to a configuration $\chi$ satisfying Assumption~\ref{ComplexAssum}. Let $x$ be the unique 3-gon of $\chi$ and suppose $\chi$ contains a subtree $\chi'$ of the form
		\begin{center}
			\begin{tikzpicture}
			\draw [anchor=east](-1.3,0) node {$\chi':$};
			\draw (0.7,0.25) node {$y_1$};
			\draw [anchor=west](-0.4,0.5) node {$y_2$};
			\draw (-0.6,0.1) node {$\vdots$};
			\draw [anchor=west](-0.5,-0.6) node {$y_{\mathrm{val}(u)}$};
			\draw (0,0) node {$u$};
			\draw (1.43,0) node {$u'$};
			\draw (0.2,0) -- (1.2,0);
			\draw (-0.8,0.7) -- (-0.2,0.1);
			\draw (-0.8,-0.7) -- (-0.2,-0.1);
			\end{tikzpicture}
		\end{center}
		such that $y_2, \ldots, y_{\val(u)}$ are truncated and $y_1 \neq x$. Let $y_1,y'_2,\ldots, y'_{\val(u')}$ be the successor sequence of $y_1$ at $u'$. Then $A$ is derived equivalent to a Brauer configuration algebra associated to a Brauer configuration $\widetilde{\chi}$ such that
		\begin{equation*}
			\widetilde{\chi} = (\chi \setminus \{y_2, \ldots, y_{\val(u)}\}) \cup \{\widetilde{y}_2, \ldots, \widetilde{y}_{\val(u)}\},
		\end{equation*}
		where $\widetilde{y}_2, \ldots, \widetilde{y}_{\val(u)}$ are truncated edges connected to $u'$ in $\widetilde{\chi}$, and the successor sequence of $y_1$ at $u'$ in $\widetilde{\chi}$ is
		\begin{equation*}
			y_1,\widetilde{y}_2, \ldots, \widetilde{y}_{\val(u)}, y'_2, \ldots, y'_{\val(u')}.
		\end{equation*}
	\end{prop}
	
	\subsection{The maps between indecomposable projective modules}
	We will begin by investigating the morphisms between the indecomposable projective modules in $A$. We have the following remark from Rickard, which is a trivial consequence of the multiserial nature of the indecomposable projective modules.
	\begin{rem}[\cite{Rickard},Remark 4.1] \label{ProjHom}
		Let $x$ and $y$ be distinct polygons in a Brauer configuration $\chi$ satisfying Assumption~\ref{ComplexAssum}.
		\begin{enumerate}[label=(\alph*)]
			\item If $x$ and $y$ have no common vertex, then $\dim_K \Hom_A(P(x),P(y))=0$. Otherwise, $\dim_K \Hom_A(P(x),P(y))=1$.
			\item $\dim_K\End_A(P(x))=2$ for any polygon $x$ in $\chi$.
		\end{enumerate}
	\end{rem}
	
	It will later be convenient to know the maps between indecomposable projective modules associated to consecutive polygons in the cyclic ordering at any vertex in $\chi$ in detail.
	
	\begin{rem} \label{CanonSurject}
		If $y$ is the direct predecessor to $x$ at some vertex in $\chi$, then $S(x) \subseteq \tp (\rad P(y))$. Thus, the canonical surjection of $P(x)$ into the maximal uniserial submodule $V \subset P(y)$ with $\tp V = S(x)$ is a basis element of $\Hom_A(P(x),P(y))$.
	\end{rem}
	
	\begin{lem} \label{FactorMap}
		Let $A$ be a Brauer configuration algebra associated to a Brauer configuration $\chi$ satisfying Assumption~\ref{ComplexAssum}. Let $x_1$ be a polygon connected to a non-truncated vertex $v$ in $\chi$. Let $x_1, \ldots, x_{\val(v)}$ be the successor sequence of $x_1$ at $v$. Denote by $f_j$ the basis element of $\Hom(P(x_{j+1}), P(x_j))$ given in Remark~\ref{CanonSurject}.
		\begin{enumerate}[label=(\alph*)]
			\item $\{f_1\ldots f_{r-1}\}$ is a basis for $\Hom_A(P(x_r), P(x_1))$ when $r \neq 1$.
			\item $\{\id_{P(x_1)}, f_1\ldots f_{\val(v)}\}$ is a basis for $\End_A(P(x_1))$.
		\end{enumerate}
	\end{lem}
	\begin{proof}
		(a) First note that $f_j$ is the canonical surjection of $P(x_{j+1})$ into the indecomposable uniserial submodule $V_j \subseteq \rad P(x_j)$ such that $\tp V_j=S(x_{j+1})$. It follows that $f_{j-1} f_j$ is equivalent to the canonical surjection of $P(x_{j+1})$ into the indecomposable uniserial submodule $V_{j-1} \subseteq \rad^2 P(x_{j-1})$ such that $\tp V_{j-1}=S(x_{j+1})$. This follows since $S(x_j) \subseteq \rad P(x_{j-1}) / \rad^2 P(x_{j-1})$, $S(x_{j+1}) \subseteq \rad^2 P(x_{j-1}) / \rad^3 P(x_{j-1})$, and $\tp P(x_j) \subseteq \Coker f_j$. Using this argument iteratively, we can see that the map $f_{j-n}\ldots f_{j}$ is equivalent to the canonical surjection of $P(x_{j+1})$ into the indecomposable uniserial submodule $V_{j-n} \subseteq \rad^{n+1} P(x_{j-n})$ such that $\tp V_{j-n} = S(x_{j+1})$. By setting $j=r-1$ and $n=r-2$, the result follows.
		
		(b) With Assumption~\ref{ComplexAssum}, $\End_A(P(x_1))$ has a basis $\{\id_{P(x_1)}, t\}$, where $t$ is the canonical surjection of $P(x_1)$ into $\soc P(x_1)$. It then follows from the structure of $P(x_1)$ and by the arguments of (a) that $t=f_1\ldots f_{\val(v)}$.
	\end{proof}
	
	\subsection{Construction of the tilting complex}
	Let $x$ be the unique 3-gon of $\chi$ under Assumption~\ref{ComplexAssum}. Suppose $\chi$ contains a subtree $\chi'$ of the form
	\begin{center}
		\begin{tikzpicture}
			\draw [anchor=east](-1.3,0) node {$\chi':$};
			\draw (0.7,0.25) node {$y_1$};
			\draw [anchor=west](-0.4,0.5) node {$y_2$};
			\draw (-0.6,0.1) node {$\vdots$};
			\draw [anchor=west](-0.5,-0.6) node {$y_{\mathrm{val}(u)}$};
			\draw (0,0) node {$u$};
			\draw (1.43,0) node {$u'$};
			\draw (0.2,0) -- (1.2,0);
			\draw (-0.8,0.7) -- (-0.2,0.1);
			\draw (-0.8,-0.7) -- (-0.2,-0.1);
			\end{tikzpicture}
	\end{center}
	such that $y_2, \ldots, y_{\val(u)}$ are truncated and $y_1 \neq x$.
	
	For the polygon $y_1$ in $\chi'$, define a stalk complex
	\begin{equation*}
		T(y_1): \xymatrix@1{0 \ar[r] & P(y_1) \ar[r] & 0},
	\end{equation*}
	where $P(y_1)$ is in degree zero. For every other polygon $y_i$ in $\chi'$, define a complex
	\begin{equation*}
		T(y_i): \xymatrix@1{0 \ar[r] & P(y_1) \ar[r]^{f_i} & P(y_i) \ar[r] & 0},
	\end{equation*}
	where the $P(y_1)$ term is in degree zero. Note that by Remark~\ref{ProjHom}, such a complex is unique up to isomorphism in $K^b(\proj A)$. For every other polygon $z$ in $\chi$ (that is, for any polygon $z$ not in $\chi'$), we define a stalk complex
	\begin{equation*}
		T(z): \xymatrix@1{0 \ar[r] & P(z) \ar[r] & 0},
	\end{equation*}
	where $P(z)$ is in degree zero. Then define $T=\bigoplus_{x \in Q_0} T(x)$. Note that for any $i \neq 1$, we have $\varepsilon_{y_i} A \varepsilon_z A =0$ whenever $z \neq y_j$ for any $j$. Moreover, $P(y_1)=P(\varepsilon_{y_i} A \varepsilon_{y_1} A)$ for all $i \neq 1$, since each $P(y_i)$ is uniserial. Thus, $T$ is an Okuyama-Rickard tilting complex.
	
	\subsection{The maps between the direct summands of $T$} \label{ComplexMaps}
	For the purposes of readability, let $\mathcal{C}=K^b(\proj A)$. We aim to calculate $\End_\mathcal{C}(T)$. For maps between stalk complexes, this can simply be viewed as a map between indecomposable projective modules. We will investigate the morphisms in $\Hom_\mathcal{C}(T(y_i), T(y_j))$.
	
	Suppose $j < i$. Then by Lemma~\ref{FactorMap}(a), any map $f_j \in \Hom_A(P(y_1),P(y_j))$ can be written as a map $f_j=h f_i$, where $f_i \in \Hom_A(P(y_1),P(y_i))$. Thus, given a morphism
	\begin{equation*}
		\xymatrix{
			0 \ar[r]	&	P(y_1) \ar[r]^{f_i} \ar[d]^{g_0}	&	P(y_i) \ar[r] \ar[d]^{g_1}	&	0	\\
			0 \ar[r]	&	P(y_1)	\ar[r]^{f_j}					&	P(y_j) \ar[r]					&	0
		},
	\end{equation*}
	we can see that $\dim_K\Hom_\mathcal{C}(T(y_i), T(y_j)) \leq 2$. Namely, 
	\begin{equation*}
		\Hom_\mathcal{C}(T(y_i), T(y_j))=\spn \{(\id_{P(y_1)}, h), (t_1, 0)\},
	\end{equation*}
	where $t_1\in \End_A(P(y_1))$ maps an element from $\tp P(y_1)$ to $\soc P(y_1)$. But $(t_1, 0) \simeq 0$, since any map in $\End_A(P(y_1))$ that factors through $f_i$ is a scalar multiple of $t_1$ and any map in $\Hom_{A}(P(y_i), P(y_j))$ that factors through $f_j$ is zero (by the assumption that $\mathfrak{e}_u=1$). So $\dim_K\Hom_\mathcal{C}(T(y_i), T(y_j)) = 1$.
	
	Now suppose $j > i$. Then for any map $g_1\in \Hom_A(P(y_i),P(y_j))$, the composition $g_1 f_i$ factors through the map $t_1$. Such a factorisation is non-trivial, so $g_1 f_i=0$. So
	\begin{equation*}
		\Hom_\mathcal{C}(T(y_i), T(y_j))=\spn \{(t_1,0), (0, g_1)\},
	\end{equation*}
	In addition, there exists a map $h \in \Hom_A(P(y_i),P(y_1))$ such that $h f_i = t_1$ and $g_1=f_j h$. So $(t_1, 0) \simeq (0, -g_1)$. Hence, $\dim_K\Hom_\mathcal{C}(T(y_i), T(y_j))=1$.
	
	Now suppose $i=j$. Then
	\begin{equation*}
		\End_\mathcal{C}(T(y_i))=\spn \{(\id_{P(y_1)}, \id_{P(y_i)}), (t_1, 0), (0, t_i)\},
	\end{equation*}
	where $t_i\in \End_A(P(y_i))$ maps an element from $\tp P(y_i)$ to $\soc P(y_i)$. In fact, $(t_1, 0) \simeq (0,-t_i)$, since there exists a map $h \in \Hom_A(P(y_i), P(y_1))$ such that $t_1 = h f_i$ and $t_i= f_i h$. Moreover, $(t_1, 0) \not\simeq (\id_{P(y_1)}, \id_{P(y_i)})$, since for any morphism $h:P(y_i) \rightarrow P(y_1)$, we have $f_i h\neq \lambda\id_{P(y_i)}$ for any $\lambda \neq 0$. Thus, $\dim_K\End_\mathcal{C}(T(y_i)) =2$.
	
	\begin{lem} \label{DerivedCycle}
		Let $y_1, y'_2, \ldots, y'_{\val(u')}$ be the successor sequence of $y_1$ at $u'$. For all $1 \leq i < \val(u)$, let $\alpha_i: T(y_{i+1})\rightarrow T(y_i)$ denote the morphism such that the degree zero map is the identity. Let $\alpha_{\val(u)}:T(y'_2) \rightarrow T(y_{\val(u)})$ denote the morphism such that the degree zero map is the basis element of $\Hom_A(P(y'_2), P(y_1))$ given in Remark~\ref{CanonSurject}. Finally, for all $2 \leq i \leq \val(u')$, let $\alpha_{\val(u)+i-1}:T(y'_{i+1}) \rightarrow T(y'_i)$ denote the morphism such that the degree zero map is the basis element of $\Hom_A(P(y'_{i+1}), P(y'_i))$ given in Remark~\ref{CanonSurject}, where $y'_{\val(u')+1}:=y_1$. For any $1 \leq i,j < \val(u) + \val(u')$, consider the vector space $\Hom_\mathcal{C}(T(z_i), T(z_j))$, where $z_k = y_k$ if $k \leq \val(u)$, $z_k=y'_{k - \val(u)+1}$ if $\val(u)<k<\val(u) + \val(u')$ and $z_{\val(u)+\val(u')}=z_1=y_1$.
		\begin{enumerate} [label=(\alph*)]
			\item For all $j<i$, $\{\alpha_{j}\alpha_{j+1}\ldots \alpha_{i-1}\}$ is a basis for $\Hom_\mathcal{C}(T(z_i), T(z_j))$.
			\item For all $j>i$,
			\begin{equation*}
				\{\alpha_{j}\alpha_{j+1}\ldots \alpha_{\val(u)+\val(u')-1}\alpha_1\ldots\alpha_{i-1}\}
			\end{equation*}
			is a basis for $\Hom_\mathcal{C}(T(z_i), T(z_j))$.
			\item A basis for $\End_\mathcal{C}(T(z_j))$ is
			\begin{equation*}
				\{\id_{T(z_j)}, \alpha_{j}\alpha_{j+1}\ldots \alpha_{\val(u)+\val(u')-1}\alpha_1\ldots\alpha_{j-1}\}.
			\end{equation*}
		\end{enumerate}
	\end{lem}
	\begin{proof}
		(a) For $\val(u)<j<i\leq\val(u)+\val(u')$, this follows from Lemma~\ref{FactorMap}(a), since $T(z_i)$ and $T(z_j)$ are stalk complexes. A similar argument for holds for $\val(u)<i<\val(u)+\val(u')$ and $j=\val(u)$ when considering the maps between degree zero terms.
		
		For $1 < j<i \leq \val(u)$, it follows from the reasoning at the start of this subsection that the degree zero map is the identity and the degree $-1$ map is in the space $\Hom_A(P(i), P(j))$. Thus, the result again follows from Lemma~\ref{FactorMap}(a) when considering the degree $-1$ maps. For $1 <i \leq \val(u)$ and $j=1$. The degree zero map is either the identity map or the map $t_1:\tp P(y_1) \rightarrow \soc P(y_1)$. But by Lemma~\ref{FactorMap}(b), $t_1$ factors through a map in $\Hom_A(P(y_1),P(y_i))$. So any morphism in $\Hom_{\mathcal{C}}(T(y_i), T(y_1))$ with degree zero map $t_1$ is homotopic to zero. If the degree zero map is instead the identity then any morphism in $\Hom_{\mathcal{C}}(T(y_i), T(y_1))$ is equal to the composition of some morphism in $\Hom_{\mathcal{C}}(T(y_i), T(y_2))$ with the morphism $\alpha_1$. By considering Lemma~\ref{FactorMap}(a) on the degree $-1$ terms, the result follows.
		
		(b) The arguments used in the proof for (a) form a cycle of maps. The proof for (b) is hence similar.
		
		(c) If $f \in \End_{\mathcal{C}}(T(z_j))$ is a non-identity map, then the degree zero map must be a map from the top to the socle of the projective module in degree zero. By Lemma~\ref{FactorMap}(b), this is equivalent to a cycle of maps between the indecomposable projective modules corresponding to the polygons around the vertex $u'$. The result then follows from the proof of (a) and (b).
	\end{proof}

	\begin{lem} \label{DerivedArrows}
		Define the following set of germs of polygons in $\chi$.
		\begin{equation*}
			Z= \{ z^v | v \text{ non-truncated}, v \neq u' \text{ and } z \neq y_i \text{ for all } i\}
		\end{equation*}
		Let $z^v \in Z$ and let $z'$ be the successor to $z$ at $v$. Denote by $\beta_{v, z}$ the morphism in $\Hom_{\mathcal{C}}(T(z'), T(z))$ whose degree zero map is the basis element of $\Hom_A(P(z'),P(z))$ given in Remark~\ref{CanonSurject}. Then
		\begin{equation*}
			\left< \alpha_1, \ldots, \alpha_{\val(u)+\val(u')-1}, \beta_{v, z} \right>_{ z^v \in Z} = \End_{K^b(\proj A)}(T),
		\end{equation*}
		where $\alpha_1, \ldots, \alpha_{\val(u)+\val(u')-1}$ are as in Lemma~\ref{DerivedCycle}.
	\end{lem}
	\begin{proof}
		This is a trivial consequence of both Lemma~\ref{FactorMap}(b) and Lemma~\ref{DerivedCycle}.
	\end{proof}

	\subsection{The relations of the endomorphism algebra of $T$}
	We will now explicitly calculate the algebra $\End_{K^b(\proj A)}(T)=K\widetilde{Q}/\widetilde{I}$. By Lemma~\ref{DerivedArrows}, the arrows of $\widetilde{Q}$ are given by the maps $\alpha_1, \ldots, \alpha_{\val(u)+\val(u')-1}$ and $(\beta_{v, z})_{ z^v \in Z}$. It remains to calculate the relations that generate $\widetilde{I}$.
	
	\begin{lem} \label{StandardRelations}
		Suppose a polygon $z$ of $\chi$ is connected to non-truncated vertices $v,v' \not\in \{u,u'\}$. Let $z=z_1,\ldots, z_{\val(v)}$ and $z=z'_1,\ldots, z'_{\val(v')}$ be the successor sequences of $z$ at $v$ and $v'$ respectively. Then
		\begin{enumerate} [label=(\alph*)]
			\item $\beta_{v, z_1}\ldots\beta_{v, z_{\val(v)}} = \beta_{v', z'_1}\ldots\beta_{v', z'_{\val(v')}} \neq 0$.
			\item $\beta_{v, z_{\val(v)}}\beta_{v', z'_1} = 0$ and $\beta_{v', z'_{\val(v')}}\beta_{v, z_1} = 0$.
		\end{enumerate}
		If $z$ is instead connected to only one non-truncated vertex $v$, then
		\begin{enumerate} [label=(\alph*)]
			\setcounter{enumi}{2}
			\item $\beta_{v, z_1}\ldots\beta_{v, z_{\val(v)}}\beta_{v, z_1}=0$.
		\end{enumerate}
	\end{lem}
	\begin{proof}
		(a) Since all morphisms are maps between stalk complexes, this is a trivial consequence of Lemma~\ref{FactorMap}(a). Namely, $\beta_{v, z_1}\ldots\beta_{v, z_{\val(v)}}$ is the morphism such that the degree zero map corresponds to the basis element of $\End_A(P(z_1))$ that maps $\tp P(z_1)$ to $\soc P(z_1)$. The same is true for $(\beta_{v', z'_1}\ldots\beta_{v', z'_{\val(v')}})^{\mathfrak{e}_{v'}}$.
		
		(b) This follows from Remark~\ref{ProjHom}(a), since this corresponds a map between the indecomposable projective modules of two polygons that have no common vertex.
		
		(c) Similar to (a), $(\beta_{v, z_2}\ldots\beta_{v, z_{\val(v)}}\beta_{v, z_1})^{\mathfrak{e}_v}$ maps $\tp P(z_2)$ to $\soc P(z_2)$. But $\soc P(z_2)$ is in the kernel of the degree zero map of $\beta_{v, z_1}$. Thus, the result follows.
	\end{proof}
	
	\begin{lem} \label{NonStandardRelations}
		Suppose a polygon $z_1$ of $\chi$ is connected to a non-truncated vertex $v' \in \{u, u'\}$. Suppose $z_1$ is connected to another non-truncated vertex $v \not\in \{u,u',v'\}$ and let $z_1,\ldots, z_{\val(v)}$ be the successor sequence of $z_1$ at $v$. Suppose $T(z_1)$ is the domain of a map $\alpha_r$ from Lemma~\ref{DerivedCycle} and let
		\begin{equation*}
			C_r= \alpha_{r}\alpha_{r+1}\ldots \alpha_{\val(u)+\val(u')-1}\alpha_1\ldots\alpha_{r-1}.
		\end{equation*}
		Then
		\begin{enumerate} [label=(\alph*)]
			\item $(\beta_{v, z_1}\ldots\beta_{v, z_{\val(v)}})^{\mathfrak{e}_{v}} = C_r \neq 0$.
			\item $\alpha_{r-1}\beta_{v, z_1} = 0$ and $\beta_{v, z_{\val(v)}}\alpha_r = 0$.
		\end{enumerate}
		Suppose instead that there is no non-truncated vertex $v \not\in \{u,u',v'\}$ connected to $z_1$. Then
		\begin{enumerate} [label=(\alph*)]
			\setcounter{enumi}{2}
			\item $C_r\alpha_r=0$.
		\end{enumerate}
	\end{lem}
	\begin{proof}
		The degree zero map of $C_r$ corresponds to a non-identity map in $\End_A(P(z_1))$. Thus, the degree zero map of $C_r$ maps $\tp P(z_1)$ to $\soc P(z_1)$. It follows from the calculations in Subsection~\ref{ComplexMaps} that if $T(z_1)$ is not a stalk complex, then $C_r$ is equivalent in $K^b(\proj A)$ to a morphism in which the map in degree $-1$ is zero. The proofs to (a), (b) and (c) are then similar to the corresponding proofs in Lemma~\ref{StandardRelations}.
	\end{proof}
	
	Every arrow of $\widetilde{Q}$ belongs to either a cycle of the form $\beta_{v, z_1}\ldots\beta_{v, z_{\val(v)}}$ for some non-truncated vertex $v$ in $\chi$ or belongs to the cycle $C_r$. The only possible paths in $\widetilde{Q}$ that are non-zero are subpaths of the paths in Lemma~\ref{StandardRelations}(a) and Lemma~\ref{NonStandardRelations}(a). One can see that there are no further relations in $\widetilde{A}$, since otherwise these maps would be zero. Furthermore, these are precisely the relations of the (opposite) Brauer configuration algebra associated the Brauer configuration $\widetilde{\chi}$ in Proposition~\ref{DerivedReduction}, thus proving the proposition by Theorem~\ref{RickardTheorem} (since the algebra is symmetric).
	
	\subsection{The proof of Theorem~\ref{ExceptionalAlgebras}}
	\begin{prop} \label{DerivedEquivalence}
		Let $A$ be a Brauer configuration algebra associated to a Brauer configuration $\chi$. Suppose $\chi$ is a multiplicity-free tree with precisely one 3-gon, which is locally of the form
		
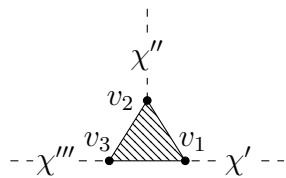
\begin{center}
			\begin{tikzpicture}
				\draw[pattern = north west lines] (0.3,-0.6) -- (1.3,-0.6) -- (0.8,0.2) -- (0.3,-0.6);
				\draw[dashed] (1.7,-0.6) -- (1.3,-0.6);
				\draw[dashed] (2.3,-0.6) -- (2.6,-0.6);
				\draw (2,-0.6) node{$\chi'$};
				\draw[dashed] (0.8,0.2) -- (0.8,0.5);
				\draw[dashed] (0.8,1.1) -- (0.8,1.5);
				\draw (0.8,0.8) node{$\chi''$};
				\draw[dashed] (-1,-0.6) -- (-0.7,-0.6);
				\draw[dashed] (-0.1,-0.6) -- (0.3,-0.6);
				\draw (-0.4,-0.6) node{$\chi'''$};
				\draw [fill=black] (1.3,-0.6) ellipse (0.05 and 0.05);
				\draw [fill=black] (0.8,0.2) ellipse (0.05 and 0.05);
				\draw [fill=black] (0.3,-0.6) ellipse (0.05 and 0.05);
				\draw (1.4,-0.35) node{$v_1$};
				\draw (0.45,0.2) node{$v_2$};
				\draw (0.15,-0.35) node{$v_3$};
			\end{tikzpicture}
		\end{center}
		where $\chi'$, $\chi''$ and $\chi'''$ are subconfigurations of $\chi$. Suppose further that $\chi'$, $\chi''$ and $\chi'''$ contain $m_1$, $m_2$ and $m_3$ edges respectively. Then $A$ is derived equivalent to a Brauer configuration algebra associated to the following Brauer configuration.
		
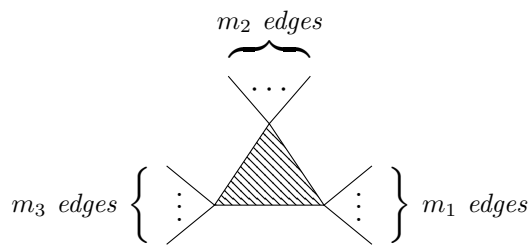
\begin{center}
			\begin{tikzpicture}[scale=0.9]
				\draw[pattern = north west lines] (0,-1) -- (1.6,-1) -- (0.8,0.2) -- (0,-1);
				
				\draw (2.3,-1.5778) -- (1.6,-1) -- (2.3,-0.4222);
				\draw (1.4,0.9) -- (0.8,0.2) -- (0.2,0.9);
				\draw (-0.7,-0.4222) -- (0,-1) -- (-0.7,-1.5778);
				
				\draw (2.1111,-0.8889) node {$\vdots$};
				\draw (0.8333,0.7) node {$\cdots$};
				\draw (-0.5111,-0.8889) node {$\vdots$};
				
				\draw (3.8,-1) node {\footnotesize$m_1$ edges};
				\draw (0.8,1.7) node {\footnotesize$m_2$ edges};
				\draw (-2.2,-1) node {\footnotesize$m_3$ edges};
				
				\draw (2.6,-1) node {$\left.\begin{smallmatrix} \\ \\ \\ \\ \\ \end{smallmatrix} \right\}$};
				\draw (0.8,1) node {$\overbrace{\begin{matrix} & & & \end{matrix}}$};
				\draw (-1,-1) node {$\left\{\begin{smallmatrix} \\ \\ \\ \\ \\ \end{smallmatrix} \right.$};
			\end{tikzpicture}
		\end{center}
	\end{prop}
	\begin{proof}
		Using Proposition~\ref{DerivedReduction} iteratively on the truncated edges attached to non-truncated 2-gons, the result follows.
	\end{proof}
	
	By the reasoning at the start of this section, the above proposition proves Theorem~\ref{ExceptionalAlgebras}.
	
	\section{Brauer Configurations with 3-gons, Cycles and Multiplicities} \label{CycleMultSubsection}
	The aim of this section is to prove the following proposition.
	\begin{prop} \label{WildCycleMult}
		Let $A$ be a Brauer configuration algebra associated to a Brauer configuration $\chi$. Suppose $\chi$ contains a 3-gon $x$. Suppose further that $\chi$ contains a cycle or a vertex $v$ such that $\mathfrak{e}_v>1$. If $\chi$ is not of the form
		
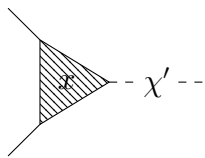
\begin{center}
			\begin{tikzpicture}[scale=0.7]
				\draw[pattern = north west lines] (0.2,-0.8) -- (1.5,0) -- (0.2,0.8) -- (0.2,-0.8);
				\draw (-0.4,1.4) -- (0.2,0.8);
				\draw (-0.4,-1.4) -- (0.2,-0.8);
				\draw[dashed] (1.5,0) -- (2,0);
				\draw[dashed] (2.8,0) -- (3.3,0);
				\draw (2.4,0) node{$\chi'$};
				\draw (0.7,0) node{$x$};
			\end{tikzpicture}
		\end{center}
		where $\chi'$ is a subconfiguration of $\chi$ and all vertices except perhaps those in $\chi'$ have multiplicity one, then $A$ is wild.
	\end{prop}
	
	To prove this, we will explicitly construct a representation embedding $H: \fin K \langle a_1,a_2 \rangle \rightarrow \Mod* A$.  We also make the following assumption.
	\begin{assumption}	\label{WildAssum}
	Let $A=KQ/I$ be a Brauer configuration algebra associated to a Brauer configuration algebra $\chi$. Suppose $\chi$ contains a 3-gon $x$ that is not self-folded and is locally of the form
	
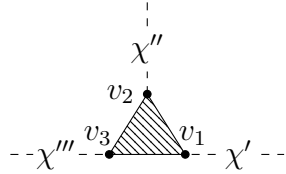
\begin{center}
		\begin{tikzpicture}
			\draw[pattern = north west lines] (0.3,-0.6) -- (1.3,-0.6) -- (0.8,0.2) -- (0.3,-0.6);
			\draw[dashed] (1.7,-0.6) -- (1.3,-0.6);
			\draw[dashed] (2.3,-0.6) -- (2.6,-0.6);
			\draw (2,-0.6) node{$\chi'$};
			\draw[dashed] (0.8,0.2) -- (0.8,0.5);
			\draw[dashed] (0.8,1.1) -- (0.8,1.5);
			\draw (0.8,0.8) node{$\chi''$};
			\draw[dashed] (-1,-0.6) -- (-0.7,-0.6);
			\draw[dashed] (-0.1,-0.6) -- (0.3,-0.6);
			\draw (-0.4,-0.6) node{$\chi'''$};
			\draw [fill=black] (1.3,-0.6) ellipse (0.05 and 0.05);
			\draw [fill=black] (0.8,0.2) ellipse (0.05 and 0.05);
			\draw [fill=black] (0.3,-0.6) ellipse (0.05 and 0.05);
			\draw (1.4,-0.35) node{$v_1$};
			\draw (0.45,0.2) node{$v_2$};
			\draw (0.15,-0.35) node{$v_3$};
		\end{tikzpicture}
	\end{center}
	where $\chi'$, $\chi''$ and $\chi'''$ are disjoint subconfigurations of $\chi$. Suppose that $\chi'$ contains a cycle or a vertex $v$ such that $\mathfrak{e}_v>1$. Assume that both $\chi''$ and $\chi'''$ are multiplicity-free trees and that $\chi'''$ contains at least two distinct polygons. In addition, let $w_\chi=\alpha_1\ldots\alpha_n$ be a non-zero string such that $\widehat{s}(\alpha_1)=x^{v_1}=\widehat{e}(\alpha_n)$ and $\alpha_1,\alpha_n \in Q_1$ (which exists by Lemma~\ref{TriStringLoop}).
	\end{assumption}
	
	\begin{rem}	\label{StringNotInTree}
		Let $u$ be any vertex in $\chi \setminus \chi'$ of Assumption~\ref{WildAssum}. Then no symbol $\alpha_i$ of the string $w_\chi$ from Assumption~\ref{WildAssum} is such that $\alpha_i$ or $\alpha^{-1}_i$ is an arrow in $\mathfrak{C}_u$.
	\end{rem}
	
	We will now begin the construction of the representation embedding. Let $w_\chi=\alpha_1\ldots\alpha_n$ be the string from Assumption~\ref{WildAssum} and consider the following wild acyclic graph.
	
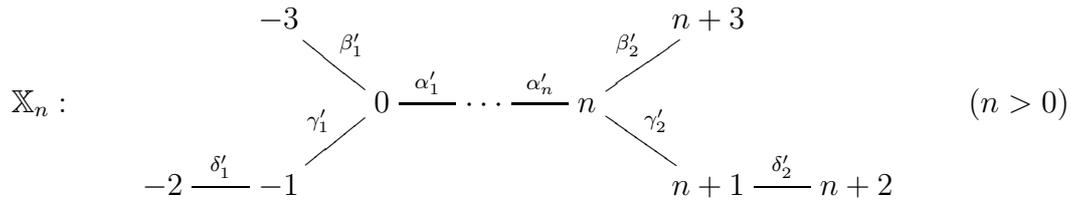
\begin{equation*}
		\xymatrix@R=1.2em@C=1.8em{
							&									&	-3 \ar@{-}[dr]^-{\beta'_1}		&					 &							&									&	n+3 \\	
			\mathbb{X}_n:	&									&										&	0	\ar@{-}[r]^-{\alpha'_1} &	\cdots \ar@{-}[r]^-{\alpha'_n}	&	n \ar@{-}[ur]^-{\beta'_2} \ar@{-}[dr]^-{\gamma'_2} &&& (n>0) \\
							&	-2 \ar@{-}[r]^-{\delta'_1}	&	-1 \ar@{-}[ur]^-{\gamma'_1}	&					 &							&									&	n+1 \ar@{-}[r]^-{\delta'_2}	&	n+2
		}
	\end{equation*}
	From the graph $\mathbb{X}_n$, we define a quiver $Q'$. Let $Q'$ be a quiver with vertex set identical to the vertex set of $\mathbb{X}_n$. We define the arrow set of $Q'$ as follows. We say that there exists an arrow $\alpha'_i: i \rightarrow i+1$ in $Q'$ whenever $\alpha_i \in Q_1$ and an arrow $\alpha'_i: i \leftarrow i+1$ in $Q'$ whenever $\alpha_i \in Q^{-1}_1$. In addition, we have arrows $\beta'_1:0 \rightarrow -3$, $\gamma'_1:0 \rightarrow -1$, $\beta'_2:n \leftarrow n+3$ and $\gamma'_2:n \leftarrow n+1$ in $Q'$. If $\val(v_3)>2$ then we say that there exist arrows $\delta'_1:-1 \rightarrow -2$ and $\delta'_2: n+1 \leftarrow n+2$ in $Q'$. Otherwise if $\val(v_3)=2$, then we say that there exist arrows $\delta'_1:-1 \leftarrow -2$ and $\delta'_2: n+1 \rightarrow n+2$ in $Q'$.
	
	It is easy to see that $KQ'$ is a finite-dimensional wild hereditary algebra (in fact, strictly wild). We will explicitly describe a fully faithful representation embedding $F:\fin K\langle a_1,a_2 \rangle \rightarrow \Mod* KQ'$. Recall that $K\langle a_1,a_2 \rangle$ is an infinite-dimensional path algebra associated to a quiver with a single vertex and two loops. Let $M=(M_0, \lambda, \mu)$ be a $K$-representation of some finite-dimensional $K\langle a_1,a_2 \rangle$-module $M$. That is, $M_0$ is a finite-dimensional vector space and $\lambda,\mu \in \End_K(M_0)$ are $K$-linear maps. Define a $K$-representation $FM=((FM)_i, \varphi_{\zeta'})_{i \in Q'_0, \zeta' \in Q'_1}$ of $Q'$, which is either of the form
	\begin{equation*}
		\xymatrix@R=2em@C=2.35em{
			&	M_0^2
				\ar@{<-}[dr]^-{\left( \begin{smallmatrix} 1 & \lambda & 0 & 0 \\ 0 & 0 & \mu & 1 \end{smallmatrix} \right)}
			&	 & &	&	&	&	M_0^2 \\
			&	&	M_0^4
				\ar[r]^-{\left( \begin{smallmatrix} 1 & 0 & 1 & 0 \\ 0 & 1 & 0 & 0 \\ 0 & 0 & 0 & 1 \end{smallmatrix} \right)}
			& M_0^3
				\ar@{-}[r]^-{\left( \begin{smallmatrix} 1 & 0 & 0 \\ 0 & 1 & 0 \\ 0 & 0 & 1 \end{smallmatrix} \right)}
			& \cdots
				\ar@{-}[r]^-{\left( \begin{smallmatrix} 1 & 0 & 0 \\ 0 & 1 & 0 \\ 0 & 0 & 1 \end{smallmatrix} \right)}
			& M_0^3
				\ar[r]^-{\left( \begin{smallmatrix} 1 & 0 & 0 \\ 0 & 1 & 0 \\ 1 & 0 & 0 \\ 0 & 0 & 1 \end{smallmatrix} \right)}
			&	M_0^4
				\ar@{<-}[ur]^-{\left( \begin{smallmatrix} 1 & 0 \\ 1 & 0 \\ 0 & 1 \\ 0 & 1 \end{smallmatrix} \right)}
				\ar@{<-}[dr]^-{\left( \begin{smallmatrix} 1 & 0 \\ 0 & 0 \\ 0 & 1 \\ 0 & 0 \end{smallmatrix} \right)} \\
			M_0
				\ar@{<-}[r]^-{\left( \begin{smallmatrix} 0 & 1 \end{smallmatrix} \right)}
			&	M_0^2
				\ar@{<-}[ur]^-{\left( \begin{smallmatrix} 1 & 0 & 0 & 0 \\ 0 & 0 & 1 & 0 \end{smallmatrix} \right)}
			&	 &	&	& & &	M_0^2
				\ar@{<-}[r]^-{\left( \begin{smallmatrix} 0 \\ 1 \end{smallmatrix} \right)}	&
			M_0
		}
	\end{equation*}
	if $\val(v_3)>2$, or of the form
	\begin{equation*}
		\xymatrix@R=2em@C=2.35em{
			&	M_0^2
				\ar@{<-}[dr]^-{\left( \begin{smallmatrix} \lambda & 1 & 0 & 0 \\ 0 & 0 & 1 & \mu \end{smallmatrix} \right)}
			&	 & &	&	&	&	M_0^2 \\
			&	&	M_0^4
				\ar[r]^-{\left( \begin{smallmatrix} 1 & 0 & 1 & 0 \\ 0 & 1 & 0 & 0 \\ 0 & 0 & 0 & 1 \end{smallmatrix} \right)}
			& M_0^3
				\ar@{-}[r]^-{\left( \begin{smallmatrix} 1 & 0 & 0 \\ 0 & 1 & 0 \\ 0 & 0 & 1 \end{smallmatrix} \right)}
			& \cdots
				\ar@{-}[r]^-{\left( \begin{smallmatrix} 1 & 0 & 0 \\ 0 & 1 & 0 \\ 0 & 0 & 1 \end{smallmatrix} \right)}
			& M_0^3
				\ar[r]^-{\left( \begin{smallmatrix} 1 & 0 & 0 \\ 0 & 1 & 0 \\ 1 & 0 & 0 \\ 0 & 0 & 1 \end{smallmatrix} \right)}
			&	M_0^4
				\ar@{<-}[ur]^-{\left( \begin{smallmatrix} 1 & 0 \\ 1 & 0 \\ 0 & 1 \\ 0 & 1 \end{smallmatrix} \right)}
				\ar@{<-}[dr]^-{\left( \begin{smallmatrix} 1 & 0 \\ 0 & 0 \\ 0 & 1 \\ 0 & 0 \end{smallmatrix} \right)} \\
			M_0
				\ar[r]^-{\left( \begin{smallmatrix} 0 \\ 1 \end{smallmatrix} \right)}
			&	M_0^2
				\ar@{<-}[ur]^-{\left( \begin{smallmatrix} 1 & 0 & 0 & 0 \\ 0 & 0 & 1 & 0 \end{smallmatrix} \right)}
			&	 &	&	& & &	M_0^2
				\ar[r]^-{\left( \begin{smallmatrix} 0 & 1 \end{smallmatrix} \right)}	&
			M_0
		}
	\end{equation*}
	if $\val(v_3)=2$. If $n=1$ then we simply have a linear map
	\begin{equation*}
		\varphi_{\alpha'_1}=
		\begin{pmatrix}
			1 & 0 & 0 \\ 0 & 1 & 0 \\ 1 & 0 & 0 \\ 0 & 0 & 1
		\end{pmatrix}
		\begin{pmatrix}
			1 & 0 & 1 & 0 \\ 0 & 1 & 0 & 0 \\ 0 & 0 & 0 & 1
		\end{pmatrix} =
		\begin{pmatrix}
			1 & 0 & 1 & 0 \\ 0 & 1 & 0 & 0 \\ 1 & 0 & 1 & 0 \\ 0 & 0 & 0 & 1
		\end{pmatrix}.
	\end{equation*}
	where $\varphi_{\alpha'_1}$ is the linear map corresponding to the arrow $\alpha'_1 \in Q'_1$.
	
	Let $N$ be a $K \langle a_1,a_2 \rangle$-module with corresponding $K$-representation $(N_0, \lambda', \mu')$. Then $f \in \Hom_{K \langle a_1,a_2 \rangle}(M,N)$ can be viewed as a $K$-linear map $f_0:M_0 \rightarrow N_0$ such that $\lambda' f_0 = f_0 \lambda$ and $\mu' f_0 = f_0 \mu$. Define $Ff \in \Hom_{KQ'}(FM,FN)$ to be the morphism $Ff=(Ff)_{i \in Q'_0}$, where each $(Ff)_i$ is a block diagonal matrix with diagonal entires $f_0$. It is easy to verify that this definition satisfies the necessary commutativity relations for $Ff$ to be a genuine morphism of $KQ'$-modules.
	
	\begin{lem} \label{FunctorFStrict}
		The functor $F:\fin K\langle a_1,a_2 \rangle \rightarrow \Mod* KQ'$ is $K$-linear, exact and fully faithful. Hence, $F$ is a strict representation embedding.
	\end{lem}
	\begin{proof}
		That $F$ is $K$-linear follows trivially from the definition. It is also easy to see from the definition that given any morphisms $f,g \in \Hom_{K \langle a_1,a_2 \rangle}(M,N)$, $Ff = Fg$ if and only if $f=g$. So $F$ is faithful. To show that $F$ is exact, it is sufficient to show that $\Ker Ff = F(\Ker f)$ and $\im Ff = F(\im f)$ for any $f \in \Hom_{K \langle a_1,a_2 \rangle}(M,N)$, since given any pair of morphisms $f,g \in \Hom_{K \langle a_1,a_2 \rangle}(M,N)$ such that $\Ker f = \im g$, we have $\Ker Ff = \im Fg$.
		
		We will first calculate $\Ker Ff$. Let $\theta:\Ker f \rightarrow M$ be an inclusion morphism such that $f\theta=0$. Note that $\theta$ is unique up to the universal property. Since each $(Ff)_i$ is a block diagonal matrix, we have $(\Ker Ff)_i = \Ker ((Ff)_i)=(F(\Ker f))_i$ for all $i$. Moreover, $(F\theta)_i$ is an inclusion morphism such that $(Ff)_i(F\theta)_i=0$ for all $i$. Thus, for any arrow $\zeta':i \rightarrow j$ in $Q'$, we have commutative squares of the following form.
		\begin{center}
			\begin{tikzpicture}
				\draw (0,2) node {$\mathrm{Ker} f_0$};
				\draw (0,0) node {$M_0$};
				\draw (0,-2) node {$N_0$};
			\draw [<-](0.52,2.2) arc (140.0037:-140:0.3);
			\draw [<-](-0.52,1.8) arc (-39.9963:-320:0.3);
			\draw [<-](0.4,0.2) arc (140.0037:-140:0.3);
			\draw [<-](-0.4,-0.2) arc (-39.9963:-320:0.3);
			\draw [<-](0.4,-1.8) arc (140.0037:-140:0.3);
			\draw [<-](-0.4,-2.2) arc (-39.9963:-320:0.3);
			\draw [->](0,1.7) -- (0,0.4);
			\draw [->](0,-0.3) -- (0,-1.6);
			\draw (1.3,2) node {$\widetilde{\lambda}$};
			\draw (-1.3,2) node {$\widetilde{\mu}$};
			\draw (1.2,0) node {$\lambda$};
			\draw (-1.2,0) node {$\mu$};
			\draw (1.2,-2) node {$\lambda'$};
			\draw (-1.2,-2) node {$\mu'$};
			\draw (-0.4,1.1) node {$\theta_0$};
			\draw (-0.4,-0.9) node {$f_0$};
			
			\draw (7,2) node {$(\mathrm{Ker} f_0)^{n_1}$};
			\draw (9.5,2) node {$(\mathrm{Ker} f_0)^{n_2}$};
			\draw (7,0) node {$M_0^{n_1}$};
			\draw (9.5,0) node {$M_0^{n_2}$};
			\draw (7,-2) node {$N_0^{n_1}$};
			\draw (9.5,-2) node {$N_0^{n_2}$};
			\draw [->](7,1.7) -- (7,0.4);
			\draw [->](9.5,1.7) -- (9.5,0.4);
			\draw [->](7,-0.3) -- (7,-1.6);
			\draw [->](9.5,-0.3) -- (9.5,-1.6);
			
			\draw [->](5.8,2) -- (6.1,2);
			\draw [->](7.8,2) -- (8.6,2);
			\draw [->](10.3,2) -- (10.6,2);
			\draw [->](5.8,0) -- (6.5,0);
			\draw [->](7.4,0) -- (9,0);
			\draw [->](9.9,0) -- (10.6,0);
			\draw [->](5.8,-2) -- (6.5,-2);
			\draw [->](7.4,-2) -- (9,-2);
			\draw [->](9.9,-2) -- (10.6,-2);
			
			\draw (5.4,2) node {$\cdots$};
			\draw (5.4,0) node {$\cdots$};
			\draw (5.4,-2) node {$\cdots$};
			\draw (11.1,2) node {$\cdots$};
			\draw (11.1,0) node {$\cdots$};
			\draw (11.1,-2) node {$\cdots$};
			\draw (6,1) node {$\left(\begin{smallmatrix} \theta_0 & & 0 \\ & \ddots & \\ 0& & \theta_0 \end{smallmatrix}\right)$};
			\draw (10.5,1) node {$\left(\begin{smallmatrix} \theta_0 & & 0 \\ & \ddots & \\ 0& & \theta_0 \end{smallmatrix}\right)$};
			\draw (6,-1) node {$\left(\begin{smallmatrix} f_0 & & 0 \\ & \ddots & \\ 0& & f_0 \end{smallmatrix}\right)$};
			\draw (10.5,-1) node {$\left(\begin{smallmatrix} f_0 & & 0 \\ & \ddots & \\ 0& & f_0 \end{smallmatrix}\right)$};
			\draw (8.2,2.3) node {$\widetilde{\varphi}_{\zeta'}$};
			\draw (8.2,0.2) node {$\varphi_{\zeta'}$};
			\draw (8.2,-1.7) node {$\varphi'_{\zeta'}$};
			
			\draw [dashed,->](2.5,0) -- (4,0);
			\draw (3.2,0.3) node {$F$};
			\end{tikzpicture}
		\end{center}
		If the block matrix entry $(\varphi_{\zeta'})_{kl}$ of $\varphi_{\zeta'}$ is an identity map, then it follows that $(\widetilde{\varphi}_{\zeta'})_{kl}$ is also an identity map. Similarly, if $(\varphi_{\zeta'})_{kl}=0$ then $(\widetilde{\varphi}_{\zeta'})_{kl}=0$, since $\theta_0$ is an inclusion morphism. If $(\varphi_{\zeta'})_{kl}$ is either the map $\lambda$ or $\mu$, then we simply note that $\lambda \theta_0 = \theta_0 \widetilde{\lambda}$ and $\mu \theta_0 = \theta_0 \widetilde{\mu}$. So $(\widetilde{\varphi}_{\zeta'})_{kl}$ is the map $\widetilde{\lambda}$ or $\widetilde{\mu}$ respectively. This is precisely the $K$-representation $F(\Ker f)$, as required. The proof for showing $\im Ff = F(\im f)$ is similar -- we simply look at the canonical surjection $\xi$ of $M$ into the image of $f$ and show that $F\xi$ is a surjection into the $\im Ff$. Thus, the functor $F$ is exact.
		
		It remains to show that $F$ is full. Let $M=(M_0, \lambda, \mu)$ and $N=(N_0, \lambda', \mu')$ and let $F_{M,N}:\Hom_{K\langle a_1, a_2 \rangle}(M,N) \rightarrow \Hom_{KQ'}(FM,FN)$ be the function defined by $F_{M,N}(f)=Ff$. We will calculate $\Hom_{KQ'}(FM, FN)$ and show that $\im F_{M,N} = \Hom_{KQ'}(FM, FN)$. We will only give the proof for the case where $\val(v_3)>2$, as the proof for the other case is similar. Let $\Phi=(\Phi_i)_{i \in Q'_0} \in \Hom_{KQ'}(FM, FN)$. We have the following commutative squares.
		\begin{equation*}
			\begin{matrix}
				\\ \\ \text{(i):}
			\end{matrix} \;
			\xymatrix@R=2em@C=2.5em{
				M_0^2
					\ar[r]^-{\left( \begin{smallmatrix} 0 & 1 \end{smallmatrix} \right)}
					\ar[d]_{\Phi_{-1}}
				& M_0
					\ar[d]^{\Phi_{-2}}\\ 
				N_0^2
					\ar[r]^-{\left( \begin{smallmatrix} 0 & 1 \end{smallmatrix} \right)}
				& N_0
			} \qquad
			\begin{matrix}
				\\ \\ \text{(ii):}
			\end{matrix} \;
			\xymatrix@R=2em@C=2.5em{
				M_0^4
					\ar[r]^-{\left( \begin{smallmatrix} 1 & 0 & 0 & 0 \\ 0 & 0 & 1 & 0 \end{smallmatrix} \right)}
					\ar[d]_{\Phi_0}
				& M_0^2 
					\ar[d]^{\Phi_{-1}}\\ 
				N_0^4
					\ar[r]^-{\left( \begin{smallmatrix} 1 & 0 & 0 & 0 \\ 0 & 0 & 1 & 0 \end{smallmatrix} \right)}
				& N_0^2
			} \qquad
			\begin{matrix}
				\\ \\ \text{(iii):}
			\end{matrix} \;
			\xymatrix@R=2em@C=2.5em{
				M_0^4
					\ar[r]^-{\left( \begin{smallmatrix} 1 & 0 & 1 & 0 \\ 0 & 1 & 0 & 0 \\ 1 & 0 & 1 & 0 \\ 0 & 0 & 0 & 1 \end{smallmatrix} \right)}
					\ar[d]_{\Phi_0}
				& M_0^4 
					\ar[d]^{\Phi_n}\\ 
				N_0^4
					\ar[r]^-{\left( \begin{smallmatrix} 1 & 0 & 1 & 0 \\ 0 & 1 & 0 & 0 \\ 1 & 0 & 1 & 0 \\ 0 & 0 & 0 & 1 \end{smallmatrix} \right)}
				& N_0^4
			}
		\end{equation*}
		\begin{equation*}
		\begin{matrix}
				\\ \\ \text{(iv):}
			\end{matrix} \;
			\xymatrix@R=2em@C=2.5em{
				M_0^2
					\ar[r]^-{\left( \begin{smallmatrix} 1 & 0 \\ 1 & 0 \\ 0 & 1 \\ 0 & 1 \end{smallmatrix} \right)}
					\ar[d]_{\Phi_{n+3}}
				& M_0^4 
					\ar[d]^{\Phi_{n}}\\ 
				N_0^2
					\ar[r]^-{\left( \begin{smallmatrix} 1 & 0 \\ 1 & 0 \\ 0 & 1 \\ 0 & 1 \end{smallmatrix} \right)}
				& N_0^4
			} \qquad
			\begin{matrix}
				\\ \\ \text{(v):}
			\end{matrix} \;
			\xymatrix@R=2em@C=2.5em{
				M_0^2
					\ar[r]^-{\left( \begin{smallmatrix} 1 & 0 \\ 0 & 0 \\ 0 & 1 \\ 0 & 0 \end{smallmatrix} \right)}
					\ar[d]_{\Phi_{n+1}}
				& M_0^4 
					\ar[d]^{\Phi_{n}}\\ 
				N_0^2
					\ar[r]^-{\left( \begin{smallmatrix} 1 & 0 \\ 0 & 0 \\ 0 & 1 \\ 0 & 0 \end{smallmatrix} \right)}
				& N_0^4
			} \qquad
			\begin{matrix}
				\\ \\ \text{(vi):}
			\end{matrix} \;
			\xymatrix@R=2em@C=2.5em{
				M_0
					\ar[r]^-{\left( \begin{smallmatrix} 0  \\ 1 \end{smallmatrix} \right)}
					\ar[d]_{\Phi_{n+2}}
				& M_0^2
					\ar[d]^{\Phi_{n+1}}\\ 
				N_0
					\ar[r]^-{\left( \begin{smallmatrix} 0\\ 1 \end{smallmatrix} \right)}
				& N_0^2
			}
		\end{equation*}
		Note that if the length $n$ of the string $w_\chi$ used to construct $Q'$ is such that $n>1$, then square (iii) is obtained by composing the linear maps between the vertices 0 and $n$ of $Q'$. Square (i) implies that
		\begin{equation*}
			\Phi_{-1}=
			\begin{pmatrix}
				f & g \\
				0 & h
			\end{pmatrix},
		\end{equation*}
		where $f,g,h \in \Hom_K(M_0,N_0)$ and $h=\Phi_{-2}$. Square (ii) then implies that
		\begin{align*}
			(\Phi_0)_{11}&=f,& (\Phi_0)_{12}&=0,& (\Phi_0)_{13}&=g,& (\Phi_0)_{14}&=0, \\
			(\Phi_0)_{31}&=0,& (\Phi_0)_{32}&=0,& (\Phi_0)_{33}&=h,& (\Phi_0)_{34}&=0.
		\end{align*}
		Similarly, square (vi) implies that
		\begin{equation*}
			\Phi_{n+1}=
			\begin{pmatrix}
				f' & 0 \\
				g' & h'
			\end{pmatrix},
		\end{equation*}
		where $f',g',h' \in \Hom_K(M_0,N_0)$ and $h'=\Phi_{n+2}$. Square (v) then implies that
		\begin{align*}
			(\Phi_n)_{11}&=f',& (\Phi_n)_{13}&=0, \\
			(\Phi_n)_{21}&=0,& (\Phi_n)_{23}&=0, \\
			(\Phi_n)_{31}&=g',& (\Phi_n)_{33}&=h', \\
			(\Phi_n)_{41}&=0,& (\Phi_n)_{43}&=0.
		\end{align*}
		By square (iii), we have
		\begin{equation*}
			\begin{pmatrix}
				f					&	0					&	g +h				&	0					\\
				(\Phi_0)_{21}		&	(\Phi_0)_{22}		&	(\Phi_0)_{23}		&	(\Phi_0)_{24}		\\
				f					&	0					&	g+h				&	0					\\
				(\Phi_0)_{41}		&	(\Phi_0)_{42}		&	(\Phi_0)_{43}		&	(\Phi_0)_{44}
			\end{pmatrix}
			=
			\begin{pmatrix}
				f'					&	(\Phi_n)_{12}		&	f'					&	(\Phi_n)_{14}		\\
				0					&	(\Phi_n)_{22}		&	0					&	(\Phi_n)_{24}		\\
				g' + h'				&	(\Phi_n)_{32}		&	g'+h'				&	(\Phi_n)_{34}		\\
				0					&	(\Phi_n)_{42}		&	0					&	(\Phi_n)_{44}
			\end{pmatrix}.
		\end{equation*}
		So $f=f'$. In addition, square (iv) implies
		\begin{align*}
			(\Phi_n)_{22}&=f,	&	(\Phi_n)_{24}&=0, \\
			(\Phi_n)_{42}&=g',	&	(\Phi_n)_{44}&=h'.
		\end{align*}
		and 
		\begin{equation*}
			\Phi_{n+3}=
			\begin{pmatrix}
				f & 0 \\
				g' & h'
			\end{pmatrix}.
		\end{equation*}
		So
		\begin{equation*}
			\Phi_0 =
			\begin{pmatrix}
				f	&	0	&	g	&	0	\\
				0	&	f	&	0	&	0	\\
				0	&	0	&	h	&	0	\\
				0	&	g'	&	0	&	h'
			\end{pmatrix}
			\quad \text{and} \quad
			\Phi_n=
			\begin{pmatrix}
				f	&	0	&	0	&	0	\\
				0	&	f	&	0	&	0	\\
				g'	&	0	&	h'	&	0	\\
				0	&	g'	&	0	&	h'
			\end{pmatrix}
		\end{equation*}
		with the relation $f=g+h=g'+h'$. Finally, we consider the commutative square
		\begin{equation*}
			\xymatrix@R=2em@C=2.9em{
				M_0^4
					\ar[r]^-{\left( \begin{smallmatrix} 1 & \lambda & 0 & 0 \\ 0 & 0 & \mu & 1 \end{smallmatrix} \right)}
					\ar[d]_{\Phi_0}
				& M_0^2 
					\ar[d]^{\Phi_{-3}}\\ 
				N_0^4
					\ar[r]^-{\left( \begin{smallmatrix} 1 & \lambda' & 0 & 0 \\ 0 & 0 & \mu' & 1 \end{smallmatrix} \right)}
				& N_0^2
			}
		\end{equation*}
		which gives us $g=g'=0$ and therefore $f=h=h'$. The above square also gives us
		\begin{equation*}
			\Phi_{-3} = \begin{pmatrix}
				f & 0 \\
				0 & f
			\end{pmatrix}
		\end{equation*}
		and commutativity relations $\lambda' f=f\lambda$ and $\mu' f=f\mu$. Thus $f$ can also be considered as a morphism in $\Hom_{K\langle a_1, a_2 \rangle}(M, N)$. It is easy to see that the commutative squares involving $\Phi_i$ for $0<i<n$ give diagonal matrices with diagonal entries $f$. Thus, we have shown $\Phi_i$ is a block diagonal matrix with diagonal entries $f\in \Hom_{K\langle a_1, a_2 \rangle}(M, N)$ for all $i$. So
		\begin{equation*}
			\im F_{M,N} = \Hom_{KQ'}(FM,FN)
		\end{equation*}
		and hence, $F$ is a full functor. So $F$ is a strict representation embedding.
	\end{proof}
	
	Now that we have a representation embedding $F:\fin K\langle a_1,a_2 \rangle \rightarrow \Mod* KQ'$, we aim to construct a functor $G:\Mod* KQ' \rightarrow \Mod* A$ such that the composite functor $H=GF$ is a representation embedding. We give advance notice to the reader that $H$ will not be strict. Thus, we must prove that $H$ is exact, respects isomorphism classes and maps indecomposable $K\langle a_1,a_2 \rangle$-modules to indecomposable $A$-modules
	
	Recall that $Q'$ is defined using the Brauer configuration algebra $A$ associated to a Brauer configuration $\chi$ that contains a 3-gon $x$, and a string $w_\chi=\alpha_1\ldots\alpha_n$, which both satisfy Assumption~\ref{WildAssum}. Define a morphism of quivers $\pi=(\pi_0,\pi_1):Q'\rightarrow Q$. That is, we will define maps $\pi_0:Q'_0 \rightarrow Q_0$ and $\pi_1:Q'_1 \rightarrow Q_1$ such that any arrow $\alpha':i \rightarrow j$ in $Q'$ is mapped to an arrow $\pi_1(\alpha'):\pi_0(i) \rightarrow \pi_0(j)$ in $Q$. Define $\pi_0(0)=x=\pi_0(n)$. Then for all $1 \leq i \leq n$, define $\pi_1(\alpha'_i) =\alpha_i$ if $\alpha_i \in Q_1$ and $\pi_1(\alpha'_i) =\alpha^{-1}_i$ if $\alpha_i \in Q^{-1}_1$. Let $\beta_1$, $\beta_2$, $\gamma_1$ and $\gamma_2$ be the distinct arrows of $Q$ such that $\widehat{s}(\beta_1)=x^{v_2}$, $\widehat{e}(\beta_2)=x^{v_2}$, $\widehat{s}(\gamma_1)=x^{v_3}$ and $\widehat{e}(\gamma_2)=x^{v_3}$. Then define $\pi_1(\beta'_1)=\beta_1$, $\pi_1(\beta'_2)=\beta_2$, $\pi_1(\gamma'_1)=\gamma_1$ and $\pi_1(\gamma'_2)=\gamma_2$.
	
	If $\val(v_3)>2$, then $\gamma'_1\delta'_1$ and $\delta'_2\gamma'_2$ form directed paths in $Q'$. In this case, we define $\pi_1(\delta'_1)$ to be the unique arrow of $Q$ such that $\pi_1(\gamma'_1)\pi_1(\delta'_1) \not\in I$ and define $\pi_1(\delta'_2)$ to be the unique arrow of $Q$ such that $\pi_1(\delta'_2)\pi_1(\gamma'_2) \not\in I$. See Figure~\ref{PiFigure}(b) for a visual illustration.
	
	Otherwise if $\val(v_3)=2$, then there exists a polygon $y=\pi_0(e(\gamma'_1))=\pi_0(s(\gamma'_2))$. Let $u$ be any vertex distinct from $v_3$ connected to $y$. Such a choice of vertex $u$ is not necessarily unique (for example, if $\pi_0(e(\gamma'_1))$ is not a 2-gon). However, the proof is not dependent on the choice of the vertex $u$, and so any choice of $u$ will do. Then let $\delta_1$ and $\delta_2$ be the arrows of $Q$ such that $\widehat{e}(\delta_1)=y^u=\widehat{s}(\delta_2)$ and define $\pi_1(\delta'_1)=\delta_1$ and $\pi_1(\delta'_2)=\delta_2$. See Figure~\ref{PiFigure}(a) for a visual illustration.
	
	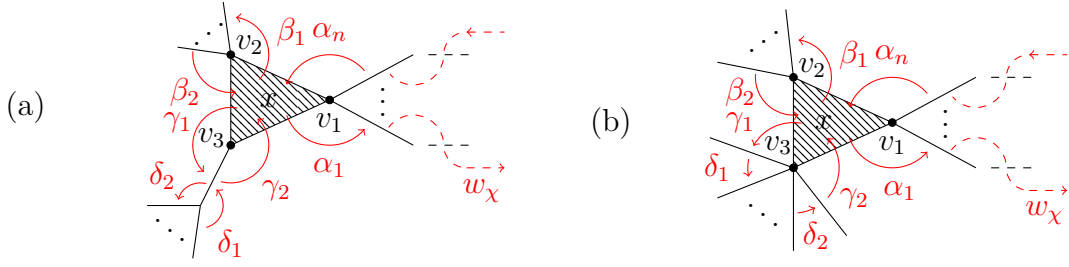
\begin{figure}
		\centering
		\begin{tikzpicture}
		\draw (-3.8,0.4) node {(a)};
		\draw [pattern=north west lines](-1.1,1.1) -- (-1.1,-0.1) -- (0.2,0.5) -- (-1.1,1.1);
		\draw [fill=black] (0.2,0.5) ellipse (0.05 and 0.05);
		\draw [fill=black] (-1.1,1.1) ellipse (0.05 and 0.05);
		\draw [fill=black] (-1.1,-0.1) ellipse (0.05 and 0.05);
		\draw (-0.6,0.5) node {$x$};
		\draw (0.2,0.2) node {$v_1$};
		\draw (-0.85,1.23) node {$v_2$};
		\draw (-1.35,0) node {$v_3$};
		
		\draw (1.3,-0.1) -- (0.2,0.5) -- (1.3,1.1);
		\draw (0.9,0.6) node {$\vdots$};
		\draw [dashed](1.5,1.1) -- (2.1,1.1);
		\draw [dashed](1.5,-0.1) -- (2.1,-0.1);
		
		\draw (-1.2,1.8) -- (-1.1,1.1);
		\draw (-1.4,1.5) node {$\iddots$};
		\draw (-1.8,1.2) -- (-1.1,1.1);
		
		\draw (-1.5,-0.9) -- (-1.1,-0.1);
		\draw (-1.9,-1.1) node {$\ddots$};
		\draw (-2.2,-0.9) -- (-1.5,-0.9);
		\draw (-1.5,-0.9) -- (-1.6,-1.6);
		
		\draw [red,->](-0.3638,0.2948) arc (-160.0006:-40:0.6);
		\draw [red,->](0.6596,0.8857) arc (40.0036:160:0.6);
		\draw [red,->](-0.7235,0.771) arc (-41.1482:82.2857:0.5);
		\draw [red,->](-1.598,1.1448) arc (174.8595:277.7143:0.5);
		\draw [red, dashed](1.6936,-0.0286) arc (10.2818:144:0.4);
		\draw [red, dashed](1.7108,-0.1798) arc (-164.574:-90:0.3);
		\draw [red, dashed](1.6939,1.0305) arc (-10.0063:-140:0.4);
		\draw [red, dashed](1.7046,1.1521) arc (169.9976:90:0.3);
		\draw [red, dashed, <-](2,1.4) -- (2.5,1.4);
		\draw [red, dashed, ->](2,-0.4) -- (2.5,-0.4);
		
		\draw [red](2.2,-0.7) node {$w_\chi$};
		\draw [red](0.2,-0.4) node {$\alpha_1$};
		\draw [red](0.2,1.4) node {$\alpha_n$};
		\draw [red](-0.3,1.4) node {$\beta_1$};
		\draw [red](-1.7,0.6) node {$\beta_2$};
		\draw [red](-1.8,0.2) node {$\gamma_1$};
		\draw [red](-0.5,-0.7) node {$\gamma_2$};
		\draw [red](-2,-0.5) node {$\delta_2$};
		\draw [red](-1.1,-1.4) node {$\delta_1$};
		
		\draw (3.9,0.2) node {(b)};
		\draw [pattern=north west lines](6.3,0.8) -- (6.3,-0.4) -- (7.6,0.2) -- (6.3,0.8);
		\draw [fill=black] (7.6,0.2) ellipse (0.05 and 0.05);
		\draw [fill=black] (6.3,0.8) ellipse (0.05 and 0.05);
		\draw [fill=black] (6.3,-0.4) ellipse (0.05 and 0.05);
		\draw (6.7,0.2) node {$x$};
		\draw (7.6,-0.1) node {$v_1$};
		\draw (6.55,0.93) node {$v_2$};
		\draw (6.1,-0.1) node {$v_3$};
		
		\draw (8.7,-0.4) -- (7.6,0.2) -- (8.7,0.8);
		\draw (8.3,0.3) node {$\vdots$};
		\draw [dashed](8.9,0.8) -- (9.5,0.8);
		\draw [dashed](8.9,-0.4) -- (9.5,-0.4);
		
		\draw (6.2,1.7) -- (6.3,0.8);
		\draw (5.9,-0.9) node {$\ddots$};
		\draw (5.3,1) -- (6.3,0.8);
		
		\draw (6.3,-0.4) -- (7,-1.3);
		\draw (5.2,0) -- (6.3,-0.4);
		\draw (5.9,1.4) node {$\iddots$};
		\draw (5.3,-0.8) -- (6.3,-0.4);
		\draw (6.3,-0.4) -- (6.3,-1.5);
		
		\draw [red,->](7.0362,-0.0052) arc (-160.0006:-40:0.6);
		\draw [red,->](8.0596,0.5857) arc (40.0036:160:0.6);
		\draw [red,->](6.4042,0.1909) arc (79.9992:150.9677:0.6);
		\draw [red,->](5.7115,-0.2829) arc (168.7462:191.25:0.6);
		\draw [red,->](6.3554,-0.9974) arc (-84.7018:-61.7143:0.6);
		\draw [red,->](6.7518,-0.7948) arc (-41.1482:41.1429:0.6);
		\draw [red,->](6.6765,0.471) arc (-41.1482:82.2857:0.5);
		\draw [red,->](5.802,0.8448) arc (174.8595:277.7143:0.5);
		\draw [red, dashed](9.0936,-0.3286) arc (10.2818:144:0.4);
		\draw [red, dashed](9.1108,-0.4798) arc (-164.574:-90:0.3);
		\draw [red, dashed](9.0939,0.7305) arc (-10.0063:-140:0.4);
		\draw [red, dashed](9.1046,0.8521) arc (169.9976:90:0.3);
		\draw [red, dashed, <-](9.4,1.1) -- (9.9,1.1);
		\draw [red, dashed, ->](9.4,-0.7) -- (9.9,-0.7);
		
		\draw [red](9.6,-1) node {$w_\chi$};
		\draw [red](7.7,-0.7) node {$\alpha_1$};
		\draw [red](7.6,1.1) node {$\alpha_n$};
		\draw [red](7.1,1.1) node {$\beta_1$};
		\draw [red](5.6,0.6) node {$\beta_2$};
		\draw [red](5.6,0.2) node {$\gamma_1$};
		\draw [red](7.1,-0.8) node {$\gamma_2$};
		\draw [red](5.3,-0.4) node {$\delta_1$};
		\draw [red](6.6,-1.3) node {$\delta_2$};
		\draw [red, ->](-1.0132,0.3924) arc (80.0026:225:0.5);
		\draw [red, ->](-1.2294,-0.583) arc (-104.9979:45:0.5);
		\draw [red, ->](-1.4224,-1.1898) arc (-75.0095:45:0.3);
		\draw [red, ->](-1.4224,-0.6102) arc (75.0095:165:0.3);
		\end{tikzpicture}
		\caption{The image of the map $\pi$ in the cases where (a) $\val(v_3)=2$ and (b) $\val(v_3)>2$. Primed arrows in $Q'$ are mapped to the corresponding unprimed arrow in the figure. For example, $\pi_1(\alpha'_1)=\alpha_1$.} \label{PiFigure}
	\end{figure}
	
	\begin{rem}	\label{PiPreImages}
		We have the following notes on the image of $\pi$.
		\begin{enumerate}[label=(\alph*)]
			\item The image of any directed path in $Q'$ under $\pi$ avoids the relations in $I$.
			\item For any $i,j \in Q'_0$, define an equivalence relation $i \sim j \Leftrightarrow \pi_0(i)=\pi_0(j)$ and denote the equivalence class of $i \in Q'_0$ by $[i]$. Then for example, we have the following.
				\begin{align*}
					[-3] &=
					\begin{cases}
						\{-3, n+3\}	& \text{if } \val(v_2)=2, \\
						\{-3 \}			& \text{otherwise,}
					\end{cases} \\
					[0] &= [n] \\
					[n+3] &=
					\begin{cases}
						\{-3, n+3\}	& \text{if } \val(v_2)=2, \\
						\{n+3 \}			& \text{otherwise.}
					\end{cases}
				\end{align*}
			\item For any $\zeta',\eta' \in Q'_1$, define an equivalence relation $\zeta' \sim \eta' \Leftrightarrow \pi_1(\zeta')=\pi_1(\eta')$ and denote the equivalence class of $\zeta' \in Q'_1$ by $[\zeta']$. Then by Remark~\ref{StringNotInTree},
				\begin{equation*}
					\beta'_1,\beta'_2,\gamma'_1,\gamma'_2, \delta'_1, \delta'_2 \not \in [\alpha'_i] \text{ for any } 1\leq i \leq n.
				\end{equation*}
				In addition, since it follows from Assumption~\ref{WildAssum} that $\val(v_2), \val(v_3) \geq 2$, we have
				\begin{align*}
					[\beta'_1] 	&= \{\beta'_1\},		& [\beta'_2] 		&= \{\beta'_2\}, \\
					[\gamma'_1]	&= \{\gamma'_1\},	& [\gamma'_2]	&= \{\gamma'_2\}, \\
					[\delta'_1] 	&=
					\begin{cases}
						\{\delta'_1, \delta'_2\}	&	\text{if } \val(v_3)=3, \\
						\{\delta'_1\}	&	\text{otherwise,}
					\end{cases}
					& [\delta'_2] &=
					\begin{cases}
						\{\delta'_1, \delta'_2\}	&	\text{if } \val(v_3)=3, \\
						\{\delta'_2\}	&	\text{otherwise.}
					\end{cases}.
				\end{align*}
		\end{enumerate}
	\end{rem}
	
	 Let $M'=(M'_i, \varphi_{\zeta'})_{i \in Q'_0, \zeta' \in Q'_1}$ be a representation of the quiver $Q'$ over the field $K$. Define a representation $GM'=((GM')_y, \phi_\zeta)_{y \in Q_0, \zeta \in Q_1}$ of the quiver $Q$ over $K$ in the following way. For each $i\in Q'_0$, we say that the vector space $M'_i$ is a direct summand of the vector space $(GM')_{\pi_0(i)}$. If $y \not\in \im\pi_0$, then we define $(GM')_y=0$. Consider an arrow $\zeta:y \rightarrow z$ in $Q$ such that $(GM')_y$ and $(GM')_z$ are non-zero. Suppose
	\begin{equation*}
		(GM')_y=\bigoplus_{\pi_0(i)=y} M'_i \qquad \text{and} \qquad (GM')_z=\bigoplus_{\pi_0(k)=z} M'_k
	\end{equation*}
	Then the linear map $\phi_\zeta$ is given by a block matrix $((\phi_\zeta)_{k,i})_{k \in R, i \in C}$ with row and column index sets
	\begin{equation*}
		R= \{k : \pi(k)=z\} \qquad \text{ and } \qquad C= \{i : \pi(i)=y\}
	\end{equation*}
	respectively, where $(\phi_\zeta)_{k, i}:M'_i \rightarrow M'_k$ is a linear map defined by $\varphi_{\zeta'}$ if $\zeta':i \rightarrow k$ is an arrow in $Q'$ such that $\pi_1(\zeta')=\zeta$, and is zero otherwise. Since the image of any directed path in $Q'$ under $\pi$ avoids the relations in $I$, $\phi_\zeta \phi_\eta=0$ for any path $\zeta\eta \in I$. Thus the representation $GM'$ respects the relations in $I$ and hence corresponds to an $A$-module.
	
	\begin{rem}	\label{phiMaps}
		We have the following notes on the $K$-linear maps $\phi_\zeta$ in the $K$-representation $GM'$.
		\begin{enumerate}[label=(\alph*)]
			\item Let $\zeta \in \{\pi_1(\beta'_j), \pi_1(\gamma'_j) : j=1,2\}$. Then by Remark~\ref{PiPreImages}(c), the map $\phi_\zeta$ contains at most one non-zero entry when considered as a block matrix.
			\item For any arrow $\alpha'_i \in Q'_1$ ($1 \leq i \leq n$), each row and column of the block matrix $\phi_{\pi_1(\alpha'_i)}$ contains at most one non-zero entry. This follows from the fact that the arrows $\pi_1(\alpha'_i)$ follow a string of length $n$.
			\item It follows from Remark~\ref{PiPreImages}(c) that if $\val(v_3) = 3$, then $\phi_{\pi_1(\delta'_1)}=\phi_{\pi_1(\delta'_2)}$ is a diagonal $2 \times 2$ block matrix with possibly non-zero diagonal entries. Otherwise, $\phi_{\pi_1(\delta'_1)}\neq \phi_{\pi_1(\delta'_2)}$ and both $\phi_{\pi_1(\delta'_1)}$ and $\phi_{\pi_1(\delta'_2)}$ contain at most one non-zero entry.
		\end{enumerate}
	\end{rem}
		
	Let $f'=(f'_i)_{i \in Q'_0}:M' \rightarrow N'$ be a morphism of representations of $Q'$ over $K$. We define a morphism $Gf'=((Gf')_y)_{y \in Q_0}:GM' \rightarrow GN'$ of representations of $Q$ over $K$ as follows. If $y \not\in \im \pi_0$ (and hence, $(GM')_y=0$) then $(Gf')_y=0$. Otherwise, suppose 
	\begin{equation*}
		(GM')_y=\bigoplus_{\pi_0(i)=y} M'_i  \qquad \text{and} \qquad (GN')_y=\bigoplus_{\pi_0(i)=y} N'_i
	\end{equation*}
	Then we define $(Gf')_y$ to be a block diagonal matrix $((Gf')_{i,j})_{i,j \in C}$ in which each $(i,i)$-th diagonal entry is precisely the linear map $f'_i:M'_i \rightarrow N'_i$. One can verify that $Gf'$ is a genuine morphism of $A$-modules by considering commutative squares
	\begin{equation*}
		\xymatrix{
			(GM')_y \ar[d]_{(Gf')_y} \ar[r]^{\phi_{\zeta}}		&	(GM')_z \ar[d]^{(Gf')_z} \\
			(GN')_y \ar[r]^{\phi'_{\zeta}}							&	(GN')_z
		}
	\end{equation*}
	which induce commutativity relations $f'_{k} (\phi_\zeta)_{k, i} = (\phi'_\zeta)_{k, i} f'_{i}$ for each arrow $\zeta:y \rightarrow z$ in $Q$. By definition, either $(\phi_\zeta)_{k, i}=(\phi'_\zeta)_{k, i}=0$, or $(\phi_\zeta)_{k, i}=\varphi_{\zeta'}$ and $(\phi'_\zeta)_{k, i}=\varphi'_{\zeta'}$ for some arrow $\zeta': i \rightarrow k$ in $Q'$. Since $f'$ is a morphism of $KQ'$-modules (and hence the commutativity relation $f'_{k} \varphi_{\zeta'} = \varphi'_{\zeta'} f'_{i}$ is satisfied), we deduce that $Gf'$ is a morphism of $A$-modules. Thus, we have defined a functor $G:\Mod*KQ'\rightarrow \Mod* A$ which maps a $KQ'$-module $M'$ to an $A$-module $GM'$ and a morphism $f':M' \rightarrow N'$ to a morphism $Gf':GM' \rightarrow GN'$.
	
	\begin{rem}
		Given $K\langle a_1, a_2 \rangle$-modules $M$ and $N$ and a morphism $f=(f_0):M \rightarrow N$, the linear map $(GFf)_y$ is a block diagonal matrix with diagonal entries $f_0$ for all $y \in Q_0$.
	\end{rem}
	
	\begin{exam}
		Consider the Brauer configuration algebra $A$ associated to the following Brauer configuration.
		
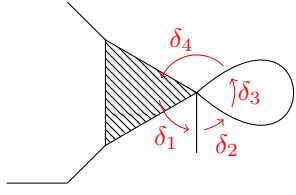
\begin{center}
			\begin{tikzpicture}
				\draw[pattern = north west lines] (1.6,-4.2) -- (2.8,-3.5) -- (1.6,-2.8) -- (1.6,-4.2);
				\draw (1.1,-2.3) -- (1.6,-2.8);
				\draw (1.1,-4.7) -- (1.6,-4.2);
				\draw (0.3,-4.7) -- (1.1,-4.7);
				\draw (2.8,-3.5) -- (2.8,-4.3);
				\draw (2.8,-3.5) .. controls (4.5,-5) and (4.5,-2) .. (2.8,-3.5);
				
				\draw [red, ->](2.3096,-3.5975) arc (-168.7552:-101.25:0.5);
				\draw [red, ->](2.8975,-3.9904) arc (-78.7552:-45:0.5);
				\draw [red, ->](3.2619,-3.6913) arc (-22.4973:22.5:0.4999);
				\draw [red, ->](3.1536,-3.1464) arc (45:157.5:0.5001);
			
				\draw [red] (2.4,-4.1) node{\footnotesize$\delta_1$};
				\draw [red] (3.2,-4.2) node{\footnotesize$\delta_2$};
				\draw [red] (3.5,-3.5) node{\footnotesize$\delta_3$};
				\draw [red] (2.6,-2.8) node{\footnotesize$\delta_4$};
			\end{tikzpicture}
		\end{center}
		Let $w_\chi=\delta_1 \delta_2 \delta^{-1}_3\delta_4$. The quiver $Q'$ associated to $(A, w_\chi)$ is the following orientation of $\mathbb{X}_4$.
		
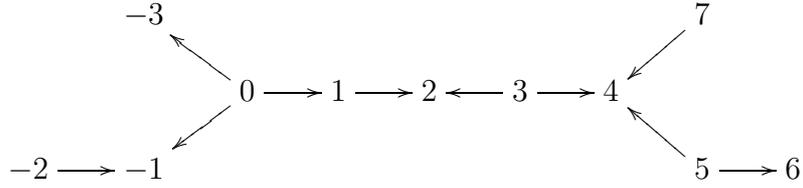
\begin{equation*}
			\xymatrix@R=1.2em@C=1.8em{
										&	-3 \ar@{<-}[dr]	&	 			&				&		&					&										&	7	\\
										&						&	0	\ar[r]	&	1 \ar[r]	&	2 	&	\ar[l] 3 \ar[r]	&	4 \ar@{<-}[ur] \ar@{<-}[dr] 	&		\\
				-2 \ar[r]				&	-1 \ar@{<-}[ur]					&	 			&		&					&						&	&	5 \ar[r]	&	6
			}
		\end{equation*}
		Consider the following $K$-representation of $Q'$.
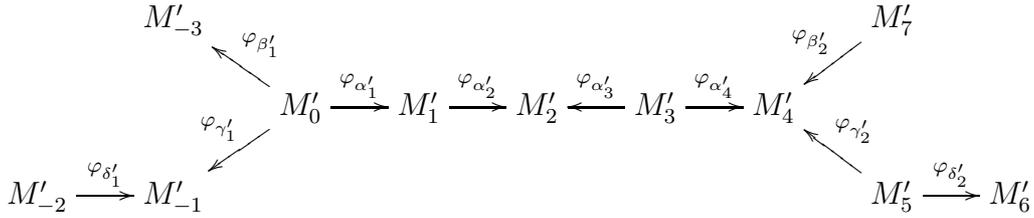
		\begin{equation*}
			\xymatrix@R=1.2em@C=1.8em{
				&	M'_{-3}
					\ar@{<-}[dr]^-{\varphi_{\beta'_1}}
				&	 & &	&	&	&	M'_{7} \\
				&	&	M'_{0}
					\ar[r]^-{\varphi_{\alpha'_1}}
				& M'_{1}
					\ar[r]^-{\varphi_{\alpha'_2}}
				& M'_{2}
					\ar@{<-}[r]^-{\varphi_{\alpha'_3}}
				& M'_{3}
					\ar[r]^-{\varphi_{\alpha'_4}}
				&	M'_{4}
					\ar@{<-}[ur]^-{\varphi_{\beta'_2}}
					\ar@{<-}[dr]^-{\varphi_{\gamma'_2}} \\
				M'_{-2}
					\ar[r]^-{\varphi_{\delta'_1}}
				&	M'_{-1}
					\ar@{<-}[ur]^-{\varphi_{\gamma'_1}}
				&	 &	&	& & &	M'_{5}
					\ar[r]^-{\varphi_{\delta'_2}}	&
				M'_{6}
			}
		\end{equation*}
		Then the $A$-module $GM'$ is given by the following $K$-representation.
		
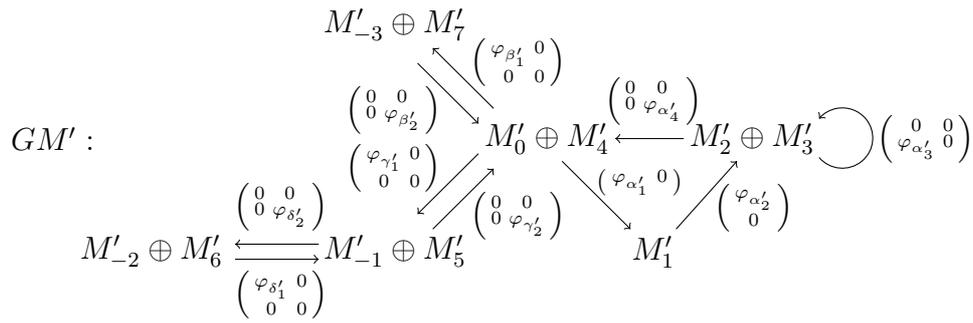
\begin{center}
			\begin{tikzpicture}
				\draw (-6.4,0) node {$GM':$};
				\draw (0.1,0) node {$M'_0 \oplus M'_4$};
				\draw (-1.9,-1.5) node {$M'_{-1} \oplus M'_5$};
				\draw [->](-0.8,-0.2) -- (-1.6,-1);
				\draw [->](-1.4,-1.2) -- (-0.6,-0.4);
				\draw (-5.1,-1.5) node {$M'_{-2} \oplus M'_6$};
				\draw [->](-2.9,-1.4) -- (-4,-1.4);
				\draw [->](-4,-1.6) -- (-2.9,-1.6);
				\draw (-1.9,1.5) node {$M'_{-3} \oplus M'_7$};
				\draw [->](-0.6,0.4) -- (-1.4,1.2);
				\draw [->](-1.6,1) -- (-0.8,0.2);
				\draw (1.5,-1.5) node {$M'_1$};
				\draw [->](0.3,-0.3) -- (1.2,-1.2);
				\draw [->](1.8,-1.2) -- (2.6,-0.3);
				\draw [->](1.9,0) -- (1,0);
				\draw (2.8,0) node {$M'_2 \oplus M'_3$};
				\draw [->](3.68,-0.25) arc (-140:140:0.4);
				
				\draw (-1.92,-0.38) node {\footnotesize$\left(\begin{smallmatrix} \varphi_{\gamma'_1} & 0 \\ 0 & 0 \end{smallmatrix}\right)$};
				\draw (-3.4,-2.07) node {\footnotesize$\left(\begin{smallmatrix} \varphi_{\delta'_1} & 0 \\ 0 & 0 \end{smallmatrix}\right)$};
				\draw (-0.28,1.02) node {\footnotesize$\left(\begin{smallmatrix} \varphi_{\beta'_1} & 0 \\ 0 & 0 \end{smallmatrix}\right)$};
				\draw (-0.28,-1.02) node {\footnotesize$\left(\begin{smallmatrix} 0 & 0 \\ 0 & \varphi_{\gamma'_2} \end{smallmatrix}\right)$};
				\draw (-3.4,-0.9) node {\footnotesize$\left(\begin{smallmatrix} 0 & 0 \\ 0 & \varphi_{\delta'_2} \end{smallmatrix}\right)$};
				\draw (-1.92,0.38) node {\footnotesize$\left(\begin{smallmatrix} 0 & 0 \\ 0 & \varphi_{\beta'_2} \end{smallmatrix}\right)$};
				\draw (1.32,-0.58) node {\footnotesize$\left(\begin{smallmatrix} \varphi_{\alpha'_1} & 0 \end{smallmatrix}\right)$};
				\draw (2.82,-0.92) node {\footnotesize$\left(\begin{smallmatrix} \varphi_{\alpha'_2} \\ 0 \end{smallmatrix}\right)$};
				\draw (5.08,0) node {\footnotesize$\left(\begin{smallmatrix} 0 & 0 \\ \varphi_{\alpha'_3} & 0 \end{smallmatrix}\right)$};
				\draw (1.5,0.5) node {\footnotesize$\left(\begin{smallmatrix} 0 & 0 \\ 0 & \varphi_{\alpha'_4} \end{smallmatrix}\right)$};
			\end{tikzpicture}
		\end{center}
		Given an $A$-module $N'$ and a morphism $f'=(f'_{-3}, \ldots, f'_7):M' \rightarrow N'$, we have a morphism $Gf':GM' \rightarrow GN'$ of the form 
		\begin{equation*}
			Gf'=\left(
			\begin{pmatrix}
				f'_{-3}	&	0		\\
				0		&	f'_7	\\
			\end{pmatrix},
			\begin{pmatrix}
				f'_{-2}	&	0		\\
				0		&	f'_6	\\
			\end{pmatrix},
			\begin{pmatrix}
				f'_{-1}	&	0		\\
				0		&	f'_5	\\
			\end{pmatrix},
			\begin{pmatrix}
				f'_0	&	0		\\
				0		&	f'_4	\\
			\end{pmatrix},
			f'_1,
			\begin{pmatrix}
				f'_2	&	0		\\
				0		&	f'_3	\\
			\end{pmatrix}
			\right)
		\end{equation*}
		Given a $K \langle a_1, a_2 \rangle$-module $M$ associated to the $K$-representation $(M_0, \lambda, \mu)$, the composite functor $H=GF$ gives us an $A$-module $HM$ associated to the following $K$-representation.
		
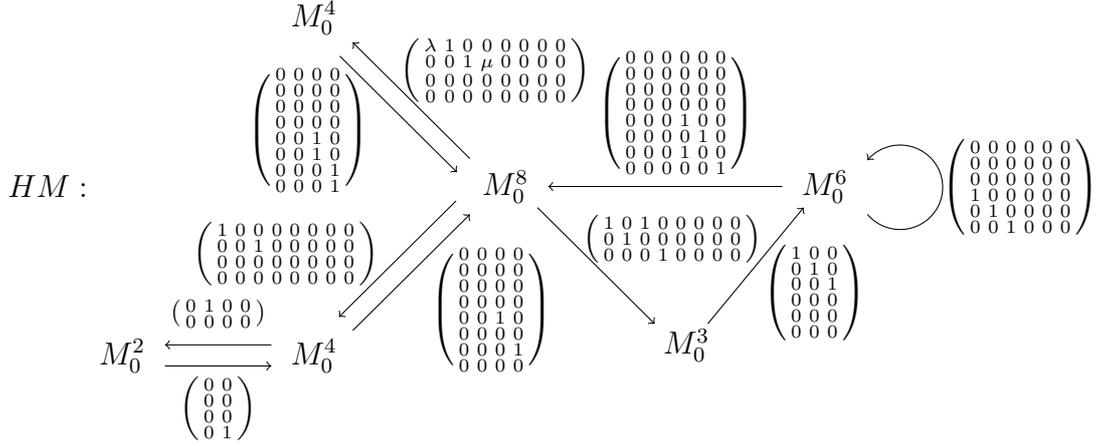
\begin{center}
			\begin{tikzpicture}[scale=1.4]
				\draw (-4.5,0) node {$HM:$};
				\draw (-0.2,0) node {$M_0^8$};
				\draw (-2,-1.6) node {$M_0^4$};
				\draw [->](-0.6645,-0.1355) -- (-1.7645,-1.2355);
				\draw [->](-1.6355,-1.3645) -- (-0.5355,-0.2645);
				\draw (-3.8,-1.6) node {$M_0^2$};
				\draw [->](-2.4,-1.5) -- (-3.4,-1.5);
				\draw [->](-3.4,-1.7) -- (-2.4,-1.7);
				\draw (-2,1.6) node {$M_0^4$};
				\draw [->](-0.5355,0.2645) -- (-1.6355,1.3645);
				\draw [->](-1.7574,1.2355) -- (-0.6574,0.1355);
				\draw (1.5,-1.5) node {$M_0^3$};
				\draw [->](0.1,-0.2) -- (1.2,-1.3);
				\draw [->](1.7,-1.3) -- (2.6,-0.2);
				\draw [->](2.4,0) -- (0.2,0);
				\draw (2.8,0) node {$M_0^6$};
				\draw [->](3.2,-0.2645) arc (-140:140:0.4);
				
				\draw (-2.2645,-0.6355) node {\footnotesize$\left(\begin{smallmatrix} 1 & 0 & 0 & 0 & 0 & 0 & 0 & 0 \\ 0 & 0 & 1 & 0 & 0 & 0 & 0 & 0 \\ 0 & 0 & 0 & 0 & 0 & 0 & 0 & 0 \\ 0 & 0 & 0 & 0 & 0 & 0 & 0 & 0 \end{smallmatrix}\right)$};
				\draw (-2.9,-2.1) node {\footnotesize$\left(\begin{smallmatrix} 0 & 0 \\ 0 & 0 \\ 0 & 0 \\ 0 & 1 \\  \end{smallmatrix}\right)$};
				\draw (-0.3,1.1) node {\footnotesize$\left(\begin{smallmatrix} \lambda & 1 & 0 & 0 & 0 & 0 & 0 & 0 \\ 0 & 0 & 1 & \mu & 0 & 0 & 0 & 0 \\ 0 & 0 & 0 & 0 & 0 & 0 & 0 & 0 \\ 0 & 0 & 0 & 0 & 0 & 0 & 0 & 0 \end{smallmatrix}\right)$};
				\draw (-0.3355,-1.1645) node {\footnotesize$\left(\begin{smallmatrix} 0 & 0 & 0 & 0 \\ 0 & 0 & 0 & 0 \\ 0 & 0 & 0 & 0 \\ 0 & 0 & 0 & 0 \\ 0 & 0 & 1 & 0 \\ 0 & 0 & 0 & 0 \\ 0 & 0 & 0 & 1 \\ 0 & 0 & 0 & 0 \end{smallmatrix}\right)$};
				\draw (-2.9,-1.2) node {\footnotesize$\left(\begin{smallmatrix} 0 & 1 & 0 & 0 \\ 0 & 0 & 0 & 0 \end{smallmatrix}\right)$};
				\draw (-2.0574,0.5355) node {\footnotesize$\left(\begin{smallmatrix} 0 & 0 & 0 & 0 \\ 0 & 0 & 0 & 0 \\ 0 & 0 & 0 & 0 \\ 0 & 0 & 0 & 0 \\ 0 & 0 & 1 & 0 \\ 0 & 0 & 1 & 0 \\ 0 & 0 & 0 & 1 \\ 0 & 0 & 0 & 1 \end{smallmatrix}\right)$};
				\draw (1.3645,-0.4935) node {\footnotesize$\left(\begin{smallmatrix} 1 & 0 & 1 & 0 & 0 & 0 & 0 & 0 \\ 0 & 1 & 0 & 0 & 0 & 0 & 0 & 0 \\ 0 & 0 & 0 & 1 & 0 & 0 & 0 & 0  \end{smallmatrix}\right)$};
				\draw (2.7,-1) node {\footnotesize$\left(\begin{smallmatrix} 1 & 0 & 0 \\ 0 & 1 & 0 \\ 0 & 0 & 1 \\ 0 & 0 & 0 \\ 0 & 0 & 0 \\ 0 & 0 & 0 \end{smallmatrix}\right)$};
				\draw (4.6355,0) node {\footnotesize$\left(\begin{smallmatrix} 0 & 0 & 0 & 0 & 0 & 0 \\ 0 & 0 & 0 & 0 & 0 & 0 \\ 0 & 0 & 0 & 0 & 0 & 0 \\ 1 & 0 & 0 & 0 & 0 & 0 \\ 0 & 1 & 0 & 0 & 0 & 0 \\ 0 & 0 & 1 & 0 & 0 & 0 \end{smallmatrix}\right)$};
				\draw (1.4,0.7) node {\footnotesize$\left(\begin{smallmatrix} 0 & 0 & 0 & 0 & 0 & 0 \\ 0 & 0 & 0 & 0 & 0 & 0 \\ 0 & 0 & 0 & 0 & 0 & 0 \\ 0 & 0 & 0 & 0 & 0 & 0 \\ 0 & 0 & 0 & 1 & 0 & 0 \\ 0 & 0 & 0 & 0 & 1 & 0 \\ 0 & 0 & 0 & 1 & 0 & 0 \\ 0 & 0 & 0 & 0 & 0 & 1 \\  \end{smallmatrix}\right)$};
			\end{tikzpicture}
		\end{center}
		Let $N=(N_0, \lambda', \mu')$. Given a morphism $f=(f_0):M \rightarrow N$, we have a morphism $Hf=((Hf)_y)_{y \in Q_0}:HM \rightarrow HN$ such that each $(Hf)_y$ is a block diagonal matrix with diagonal entries $f_0$.
	\end{exam}
	
	It is clear that the functor $G$ is not a representation embedding, since it does not respect isomorphism classes. A counterexample is formed with the simple modules $S'_0$ and $S'_n$ associated to the vertices $0$ and $n$ of $Q'$, respectively. Clearly, it follows from Remark~\ref{PiPreImages}(b) that we have $GS'_0 \cong GS'_n$ but $S'_0 \not\cong S'_n$. However, the functor $G$ should respect isomorphism classes for all modules $M' \in \Mod* KQ'$ that do not contain a string module as a direct summand. Indeed, we will show with the following lemmata that $G$ satisfies the necessary properties of a representation embedding under the image of the functor $F$. That is, we will show that the composite functor $H=GF$ is a representation embedding.
	
	\begin{lem} \label{ExactEmbed}
		The functor $H=GF:\fin K \langle a_1, a_2 \rangle \rightarrow \Mod* A$ is exact.
	\end{lem}
	\begin{proof}
		Let $M=(M_0, \lambda, \mu)$ and $N=(N_0, \lambda', \mu')$ be $K$-representations of $K \langle a_1, a_2 \rangle$-modules $M$ and $N$. Note that each vertex in the $K$-representation associated to $HM$ (resp. $HN$) is a direct sum of copies of $M_0$ (resp. $N_0$), and for any morphism $f=(f_0) \in \Hom_{K \langle a_1, a_2 \rangle} (M, N)$, the linear map $(Hf)_y$ is a block diagonal matrix with diagonal entries $f_0$ for all $y \in Q_0$. Thus, the proof of the exactness of $H$ is identical to the proof of the exactness of the functor $F$ in Lemma~\ref{FunctorFStrict}.
	\end{proof}
	
	\begin{lem} \label{IndecEmbed}
		The functor $H=GF:\fin K \langle a_1, a_2 \rangle \rightarrow \Mod* A$ maps indecomposable $K \langle a_1, a_2 \rangle$-modules to indecomposable $A$-modules.
	\end{lem}
	\begin{proof}
		We will prove that $G$ maps indecomposables to indecomposables, since then the same is true for $H$. Let $M'=(M'_i, \varphi_{\zeta'})_{i \in Q'_0, \zeta' \in Q'_1}$ be a $K$-representation of $Q$ such that $M'$ is indecomposable. Then there exist no non-trivial idempotents in $\End_{KQ'}(M')$. Suppose there exists an idempotent map $\Phi=(\Phi_y)_{y \in Q_0} \in \End_{A}(GM')$. Then for each arrow $\zeta:y \rightarrow z$ in $Q$, there exists a commutative square
		\begin{equation*}
			\xymatrix{
				\bigoplus_{\pi_0(i)=y}M'_i
					\ar[r]^{\phi_\zeta}	\ar[d]^{\Phi_y}
				&	\bigoplus_{\pi_0(k)=z}M'_k
					\ar[d]^{\Phi_z} \\
				\bigoplus_{\pi_0(i)=y}M'_i
					\ar[r]^{\phi_\zeta}
				&	\bigoplus_{\pi_0(k)=z}M'_k.
			}
		\end{equation*}
		View $\Phi_y$ and $\Phi_z$ as block matrices $((\Phi_y)_{i,j})_{\pi_0(i)=\pi_0(j)=y}$ and $((\Phi_z)_{k,l})_{\pi_0(k)=\pi_0(l)=z}$. By Remark~\ref{phiMaps}(a), (b) and (c), the map $\phi_\zeta$ contains at most one non-zero entry in each row and column when viewed as a block matrix $((\phi_\zeta)_{i,k})_{\pi_0(i)=y,\pi_0(k)=z}$. This implies the existence of commutativity relations
		\begin{equation*}
			(\phi_\zeta\Phi_y)_{k,i} = \varphi_{\zeta'}(\Phi_y)_{i,i} = (\Phi_z)_{k,k}\varphi_{\zeta'} = (\Phi_z \phi_\zeta)_{k,i}.
		\end{equation*}
		for each arrow $\zeta':i \rightarrow k$ in $Q'$ such that $\pi_1(\zeta')=\zeta$. Ranging over all $\zeta \in Q_1$, we obtain precisely the commutativity relations that determine a morphism in $\End_{KQ'}(M')$.
		
		Since $\Phi_y$ is idempotent for all $y \in Q_0$, each block diagonal entry $(\Phi_y)_{i,i}$ is idempotent (for each $y$ and each $i$). But since any idempotent in $\End_{KQ'}(M')$ is trivial, it follows that either every block entry $(\Phi_y)_{i,i}$ is an identity matrix or every block entry $(\Phi_y)_{i,i}$ is a zero matrix.
		
		Recall that an idempotent matrix is non-singular if and only if it is an identity matrix. Also recall that the trace of an idempotent matrix is equal to its rank. If every block entry $(\Phi_y)_{i,i}$ is an identity, then $\Phi_y$ is full rank for all $y \in Q_0$ (such that $(GM')_y \neq 0$). But then every $\Phi_y$ is non-singular and thus must each be an identity matrix. On the other hand, if every block entry $(\Phi_y)_{i,i}$ is a zero matrix, then the rank of each $\Phi_y$ is zero, and thus must be a zero matrix. Hence, if $\Phi \in \End_{A}(GM')$ is idempotent, then $\Phi$ must be a trivial idempotent. So $GM'$ must be indecomposable.
	\end{proof}
	
	\begin{lem} \label{IsoEmbed}
		The functor $H=GF:\fin K \langle a_1, a_2 \rangle \rightarrow \Mod* A$ respects isomorphism classes. That is, if $HM \cong HN$ then $M \cong N$ for all $M,N \in \fin K \langle a_1, a_2 \rangle$.
	\end{lem}
	\begin{proof}
		Let $M$ and $N$ be finite-dimesnional $K \langle a_1, a_2 \rangle$-modules such that $GFM \cong GFN$. Let
		\begin{align*}
			FM&=((FM)_i, \varphi_{\zeta'})_{i \in Q'_0, \zeta' \in Q'_1},	&
			FN&=((FN)_i, \varphi'_{\zeta'})_{i \in Q'_0, \zeta' \in Q'_1}	\\
			GFM&=((GFM)_y, \phi_\zeta)_{y \in Q_0, \zeta \in Q_1},		&
			GFN&=((GFN)_y, \phi'_\zeta)_{y \in Q_0, \zeta \in Q_1}.
		\end{align*}
		Then there exists a bijective $K$-linear map $(\Phi_y)_{y \in Q_0}$ such that the squares 
		\begin{equation*}
			\xymatrix{
				(GFM)_y \ar[d]_{\Phi_y} \ar[r]^{\phi_{\zeta}}		&	(GFM)_z \ar[d]^{\Phi_z} \\
				(GFN)_y \ar[r]^{\phi'_{\zeta}}						&	(GFN)_z
			}
		\end{equation*}
		commute. Write
		\begin{align*}
			(GFM)_y&=\bigoplus_{\pi_0(i)=y} (FM)_i,	&	(GFM)_z&=\bigoplus_{\pi_0(k)=z} (FM)_k, \\
			(GFN)_y&=\bigoplus_{\pi_0(i)=y} (FN)_i,	&	(GFN)_z&=\bigoplus_{\pi_0(k)=z} (FN)_k.
		\end{align*}
		Then each $\Phi_y$ can be viewed as a block matrix, where each entry $(\Phi_y)_{j,i}:(FM)_i \rightarrow (FN)_j$ is a $K$-linear map such that $\pi_0(i)=\pi_0(j)=y$. Recall further that for any arrow $\zeta: y \rightarrow z$ in $Q$, the maps $\phi_\zeta$ and $\phi'_\zeta$ can be viewed as a block matrices such that the entries $(\phi_\zeta)_{k,i}:(FM)_i \rightarrow (FM)_k$ and $(\phi'_\zeta)_{k,i}:(FN)_i \rightarrow (FN)_k$ are $K$-linear maps such that $\pi_0(i)=y$ and $\pi_0(k)=z$. If $(\phi_\zeta)_{k,i}$ (resp. $(\phi'_\zeta)_{k,i}$) is a non-zero entry of $\phi_\zeta$ (resp. $\phi'_\zeta$), then $(\phi_\zeta)_{k,i}=\varphi_{\zeta'}$ and $(\phi'_\zeta)_{k,i}=\varphi'_{\zeta'}$ for some arrow $\zeta':i \rightarrow k$ in $Q'$ such that $\pi_1(\zeta')=\zeta$. By Remark~\ref{phiMaps}(a), (b) and (c), each row and column of $\phi_\zeta$ contains at most one non-zero entry. Thus for each arrow $\zeta':i \rightarrow k$ in $Q'$, we have commutativity relations of the form
		\begin{equation} \label{CommRelationFM} \tag{$\ast$}
			(\phi'_{\pi_1(\zeta')}\Phi_{\pi_0(i)})_{k,i} = \varphi'_{\zeta'}(\Phi_{\pi_0(i)})_{i,i} = (\Phi_{\pi_0(k)})_{k,k}\varphi_{\zeta'} = (\Phi_{\pi_0(k)} \phi_{\pi_1(\zeta')})_{k,i}.
		\end{equation}
		
		Note that the relations (\ref{CommRelationFM}) are precisely the commutativity relations that determine the space $\Hom_{KQ'}(FM,FN)$. So $\Phi'=(\Phi'_i)_{i \in Q'_0} \in \Hom_{KQ'}(FM, FN)$, where $\Phi'_i = (\Phi_{\pi_0(i)})_{i,i}$. Recall that the space $\Hom_{KQ'}(FM, FN)$ was calculated in the proof of Lemma~\ref{FunctorFStrict}. Specifically, each $K$-linear map $\Phi'_i:(FM)_i \rightarrow (FN)_i$ is a block diagonal matrix with diagonal entries $f_0$, where $(f_0)=f\in \Hom_{K\langle a_1, a_2 \rangle}(M,N)$.
		
		Recall that $Q'$ is constructed using a string. It follows from Remark~\ref{PiPreImages}(a) that
		\begin{align*}
			(GFM)_{\pi_0(n+3)} &=
			\begin{cases}
				(FM)_{n+3}						&	\text{if } \val(v_2)>2 \\
				(FM)_{-3} \oplus (FM)_{n+3}	&	\text{if } \val(v_2)=2.
			\end{cases} \\
			(GFN)_{\pi_0(n+3)} &=
			\begin{cases}
				(FN)_{n+3}						&	\text{if } \val(v_2)>2 \\
				(FN)_{-3} \oplus (FN)_{n+3}	&	\text{if } \val(v_2)=2.
			\end{cases}
		\end{align*}
		In the case where $\val(v_2)>2$, we have $\Phi_{\pi_0(n+3)}=\Phi'_{n+3}$, which by assumption is bijective. But since $\Phi'_{n+3}$ is a block diagonal matrix with diagonal entries $f_0=f \in \Hom_{K\langle a_1, a_2 \rangle}(M,N)$ it follows that $\Phi'_{n+3}$ is bijective only if $f_0$ is bijective. Thus in this case, there exists an isomorphism $f:M \rightarrow N$.
		
		In the case where $\val(v_2)=2$, we have
		\begin{equation*}
			\Phi_{\pi_0(n+3)}=
			\begin{pmatrix}
				(\Phi_{\pi_0(n+3)})_{-3,-3}	&	(\Phi_{\pi_0(n+3)})_{-3,n+3}		\\
				(\Phi_{\pi_0(n+3)})_{n+3,-3}	&	(\Phi_{\pi_0(n+3)})_{n+3,n+3}
			\end{pmatrix}.
		\end{equation*}
		Firstly, we will make the observation that $\dim_K M = \dim_K N$, since $GFM \cong GFN$, which implies
		\begin{equation*}
			\dim_K GFM = (3n+15) \dim_K M_0 = (3n+15) \dim_K N_0 = \dim_K GFN,
		\end{equation*}
		where $n$ is the length of the string $w_\chi$ used in the construction of $Q'$. Secondly, note from the definition of $F$ and the above observation that
		\begin{equation*}
			\dim_K (FM)_{-3} = \dim_K (FM)_{n+3}=\dim_K (FN)_{-3} = \dim_K (FN)_{n+3},
		\end{equation*}
		so each block entry of $\Phi_{\pi_0(n+3)}$ is a square matrix. Thirdly, note from the definition of $F$ that the map
		\begin{equation*}
			\varphi_{\beta'_2}=\varphi'_{\beta'_2}=
			\begin{pmatrix}
				1	&	0	\\
				1	&	0	\\
				0	&	1	\\
				0	&	1
			\end{pmatrix}
		\end{equation*}
		is injective (where $\beta'_2: n+3 \rightarrow n \in Q'_1$, as defined earlier in this section). Fourthly, by Remark~\ref{phiMaps}(a), the map $\phi_{\pi_1(\beta'_2)}$ (resp. $\phi'_{\pi_1(\beta'_2)}$) has at most one non-zero entry, which is $(\phi_{\pi_1(\beta'_2)})_{n,n+3}=\varphi_{\beta'_2}$ (resp. $(\phi'_{\pi_1(\beta'_2)})_{n+3,n}=\varphi'_{\beta'_2}$). Thus, we have a commutative square of the form
		\begin{equation*}
			\xymatrix{
				(FM)_{-3} \oplus (FM)_{n+3}
					\ar[d]_{\Phi_{\pi_0(n+3)}}
					\ar[r]^-{\left(\begin{smallmatrix} 0 & \varphi_{\beta'_2} \\ 0 & 0 \end{smallmatrix} \right)}
				&	(FM)_n \oplus X \ar[d]^{\Phi_{\pi_0(n)}} \\
				(FN)_{-3} \oplus (FN)_{n+3}
					\ar[r]^-{\left(\begin{smallmatrix} 0 & \varphi'_{\beta'_2} \\ 0 & 0 \end{smallmatrix} \right)}
				&	(FN)_n \oplus X'
			}
		\end{equation*}
		where $X = (GFM)_{\pi_0(n)} / (FM)_n$ and $X' = (GFN)_{\pi_0(n)} / (FN)_n$. From this, we obtain the relation $\varphi'_{\beta'_2}(\Phi_{\pi_0(n+3)})_{n+3,-3} = 0$. Since $\varphi'_{\beta'_2}$ is injective, it has a left inverse. This implies that $(\Phi_{\pi_0(n+3)})_{n+3,-3} = 0$, so $\Phi_{\pi_0(n+3)}$ is a block triangular matrix. Thus,
		\begin{equation*}
			\det \Phi_{\pi_0(n+3)} = \det (\Phi_{\pi_0(n+3)})_{-3,-3} \det (\Phi_{\pi_0(n+3)})_{n+3,n+3}.
		\end{equation*}
		But $\Phi_{\pi_0(n+3)}$ is bijective by assumption and hence has non-zero determinant. Thus, both $(\Phi_{\pi_0(n+3)})_{-3,-3}$ and $(\Phi_{\pi_0(n+3)})_{n+3,n+3}$ have non-zero determinant, and hence both are bijections. But both $(\Phi_{\pi_0(n+3)})_{-3,-3}$ and $(\Phi_{\pi_0(n+3)})_{n+3,n+3}$ are diagonal matrices with diagonal entries $f_0$, where $(f_0)=f \in \Hom_{K\langle a_1, a_2 \rangle}(M,N)$. Thus, there must exist an isomorphism $f \in \Hom_{K\langle a_1, a_2 \rangle}(M,N)$, as required.
	\end{proof}
	
	We can now prove Proposition~\ref{WildCycleMult}.
	\begin{proof}
		Suppose $\chi$ is not of the form given in Proposition~\ref{WildCycleMult}. If $x$ is self-folded then $A$ is wild by Proposition~\ref{CrossBand}, so assume that this is not the case. Suppose then that $\chi$ is locally of the form
		
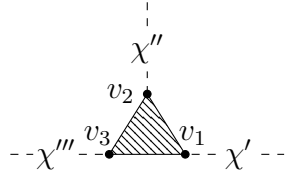
\begin{center}
			\begin{tikzpicture}
				\draw[pattern = north west lines] (0.3,-0.6) -- (1.3,-0.6) -- (0.8,0.2) -- (0.3,-0.6);
				\draw[dashed] (1.7,-0.6) -- (1.3,-0.6);
				\draw[dashed] (2.3,-0.6) -- (2.6,-0.6);
				\draw (2,-0.6) node{$\chi'$};
				\draw[dashed] (0.8,0.2) -- (0.8,0.5);
				\draw[dashed] (0.8,1.1) -- (0.8,1.5);
				\draw (0.8,0.8) node{$\chi''$};
				\draw[dashed] (-1,-0.6) -- (-0.7,-0.6);
				\draw[dashed] (-0.1,-0.6) -- (0.3,-0.6);
				\draw (-0.4,-0.6) node{$\chi'''$};
				\draw [fill=black] (1.3,-0.6) ellipse (0.05 and 0.05);
				\draw [fill=black] (0.8,0.2) ellipse (0.05 and 0.05);
				\draw [fill=black] (0.3,-0.6) ellipse (0.05 and 0.05);
				\draw (1.4,-0.35) node{$v_1$};
				\draw (0.45,0.2) node{$v_2$};
				\draw (0.15,-0.35) node{$v_3$};
			\end{tikzpicture}
		\end{center}
		where $\chi'$, $\chi''$ and $\chi'''$ are subconfigurations of $\chi$. If any two of the subconfigurations $\chi'$, $\chi''$ and $\chi'''$ contain a polygons belonging to a cycle or vertices of multiplicity strictly greater than one, then $A$ is also wild by Proposition~\ref{CrossBand}, so assume that this is not the case either. Thus, $\chi'$, $\chi''$ and $\chi'''$ are pairwaise disjoint and all cycles and vertices of $\chi$ of higher multiplicity must belong to precisely one of $\chi'$, $\chi''$ and $\chi'''$. Let $\chi'$ be this subconfiguration, which we may assume without loss of generality. Then $\chi''$ and $\chi'''$ are multiplicity-free trees. Since $\chi$ is not of the form given in the Proposition statement, there necessarily exists at least two polygons in $\chi''$ or $\chi'''$. Suppose (without loss of generality) that $\chi'''$ contains at least two polygons. Then $\chi$ satisfies Assumption~\ref{WildAssum}. Thus, there exists a $K$-linear functor $H:\fin K \langle a_1, a_2 \rangle \rightarrow \Mod* A$ (defined above) that is exact (Lemma~\ref{ExactEmbed}), maps indecomposable $K \langle a_1, a_2 \rangle$-modules to indecomposable $A$-modules (Lemma~\ref{IndecEmbed}), and respects isomorphism classes (Lemma~\ref{IsoEmbed}), and hence, is a representation embedding. Thus, $A$ is wild.
	\end{proof}
	
	\section{Brauer Configurations with at Least Two 3-gons} \label{Multiple3GonSubsection}
	We will address the final case of wild symmetric special triserial algebras, which is where the Brauer configuration contains multiple 3-gons. We begin with the following lemma.
	
	\begin{lem} \label{WildHeredSubQuiver}
		Let $A=KQ/I$ be a Brauer configuration algebra. Suppose there exists a connected acyclic subquiver $Q' \subset Q$ such that $KQ'$ is a wild hereditary algebra and every (directed) path $\alpha_1\ldots\alpha_n$ in $Q'$ is not in $I$. Then $A$ is wild.
	\end{lem}
	\begin{proof}
		Define a functor $F:\Mod* KQ' \rightarrow \Mod* A$ in the following way. For any $KQ'$-module $M$ defined by a quiver representation $(M_x, \varphi_\alpha)_{x \in Q'_0, \alpha \in Q'_1}$, we define $FM$ to be the $A$-module given by the quiver representation $((FM)_x, \phi_\alpha)_{x \in Q_0, \alpha \in Q_1}$ such that
		\begin{equation*}
			(FM)_x =
			\begin{cases}
				M_x,	&	\text{if } x \in Q'_0 \\
				0,		&	\text{otherwise}
			\end{cases}
			\qquad \text{and} \qquad
			\phi_\alpha =
			\begin{cases}
				\varphi_\alpha,	&	\text{if } \alpha \in Q'_1 \\
				0,					&	\text{otherwise.}
			\end{cases}
		\end{equation*}
		For any $KQ'$-modules $M$ and $N$ and any morphism $\Phi = (\Phi_x)_{x \in Q'_0} :M \rightarrow N$, we define a morphism $F\Phi = ((F\Phi)_x)_{x \in Q_0}:FM \rightarrow FN$ by $(F\Phi)_x=\Phi_x$ if $x \in Q'_0$ and $(F\Phi)_x = 0$ otherwise. It is easy to see that this functor is exact and fully faithful, and hence, is a (strict) representation embedding. Since, $KQ'$ is a wild algebra, this implies $A$ is also a wild algebra.
	\end{proof}
	
	\begin{prop} \label{Multiple3GonCase}
		Let $A=KQ/I$ be a Brauer configuration algebra associated to a Brauer configuration $\chi$. Suppose $\chi$ contains at least two 3-gons. Suppose further that $\chi$ is not of the form
		
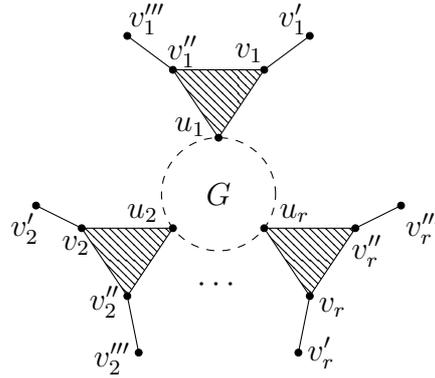
\begin{center}
			\begin{tikzpicture}[scale=1.5, ]
			\draw [dashed] (0,0) ellipse (0.5 and 0.5);
			\draw (0,0) node {$G$};
			\draw [fill=black] (0,0.5) ellipse (0.03 and 0.03);
			\draw (-0.25,0.6) node {$u_1$};
			\draw [fill=black] (-0.4,-0.3) ellipse (0.03 and 0.03);
			\draw (-0.667,-0.1665) node {$u_2$};
			\draw [fill=black] (0.4,-0.3) ellipse (0.03 and 0.03);
			\draw (0.6665,-0.1665) node {$u_r$};
			
			\draw [fill=black] (0.4,1.1) ellipse (0.03 and 0.03);
			\draw (0.2665,1.25) node {$v_1$};
			\draw [fill=black] (0.8,1.4) ellipse (0.03 and 0.03);
			\draw (0.65,1.55) node {$v'_1$};
			\draw [fill=black] (-0.4,1.1) ellipse (0.03 and 0.03);
			\draw (-0.3165,1.2634) node {$v''_1$};
			\draw [fill=black] (-0.8,1.4) ellipse (0.03 and 0.03);
			\draw (-0.6299,1.55) node (o) {$v'''_1$};
			\draw [fill=black] (-1.2,-0.3) ellipse (0.03 and 0.03);
			\draw (-1.25,-0.45) node {$v_2$};
			\draw [fill=black] (-1.6,-0.1) ellipse (0.03 and 0.03);
			\draw (-1.7,-0.3) node {$v'_2$};
			\draw [fill=black] (-0.8,-0.9) ellipse (0.03 and 0.03);
			\draw (-1,-0.95) node {$v''_2$};
			\draw [fill=black] (-0.7,-1.4) ellipse (0.03 and 0.03);
			\draw (-0.9335,-1.45) node {$v'''_2$};
			\draw [fill=black] (0.8,-0.9) ellipse (0.03 and 0.03);
			\draw (1,-1) node {$v_r$};
			\draw [fill=black] (0.7,-1.4) ellipse (0.03 and 0.03);
			\draw (0.9,-1.4) node {$v'_r$};
			\draw [fill=black] (1.2,-0.3) ellipse (0.03 and 0.03);
			\draw (1.3,-0.5) node {$v''_r$};
			\draw [fill=black] (1.6,-0.1) ellipse (0.03 and 0.03);
			\draw (1.8,-0.3) node {$v'''_r$};
			
			\draw [pattern=north west lines](0,0.5) -- (-0.4,1.1) -- (0.4,1.1) -- (0,0.5);
			\draw (-0.8,1.4) -- (-0.4,1.1);
			\draw (0.8,1.4) -- (0.4,1.1);
			\draw [pattern=north west lines](-1.2,-0.3) -- (-0.4,-0.3) -- (-0.8,-0.9) -- (-1.2,-0.3);
			\draw (-1.6,-0.1) -- (-1.2,-0.3);
			\draw (-0.7,-1.4) -- (-0.8,-0.9);
			\draw [pattern=north west lines](0.4,-0.3) -- (1.2,-0.3) -- (0.8,-0.9) -- (0.4,-0.3);
			\draw (1.6,-0.1) -- (1.2,-0.3);
			\draw (0.7,-1.4) -- (0.8,-0.9);
			\draw (0,-0.8) node {$\cdots$};
			\end{tikzpicture}
		\end{center}
		where $G$ is a Brauer graph connecting the (not necessarily distinct) vertices $u_1,\ldots,u_r$ and $\mathfrak{e}_{v_i}=\mathfrak{e}_{v'_i}=\mathfrak{e}_{v''_i}=\mathfrak{e}_{v'''_i}=1$ for all $i$. Then $A$ is wild.
	\end{prop}
	\begin{proof}
		If $\chi$ contains a 3-gon and an $n$-gon with $n>3$, then $A$ is wild by Theorem~\ref{TameQuad}. If $\chi$ contains a cycle or a vertex with multiplicity strictly greater than one, then the result follows from Proposition~\ref{WildCycleMult}. So assume that $\chi$ contains no $n$-gon with $n>3$, no cycles and no vertices with multiplicity strictly greater than one. (Note that this implies that no 3-gon of $\chi$ is self-folded.) Under this assumption, if $\chi$ is not of the form given in the proposition statement, then $\chi$ contains a 3-gon $x$ that is locally of the form
		
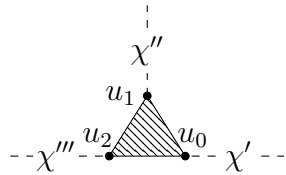
\begin{center}
			\begin{tikzpicture}
				\draw[pattern = north west lines] (0.3,-1) -- (1.3,-1) -- (0.8,-0.2) -- (0.3,-1);
				\draw[dashed] (1.7,-1) -- (1.3,-1);
				\draw[dashed] (2.3,-1) -- (2.6,-1);
				\draw (2,-1) node{$\chi'$};
				\draw[dashed] (0.8,-0.2) -- (0.8,0.1);
				\draw[dashed] (0.8,0.7) -- (0.8,1);
				\draw (0.8,0.4) node{$\chi''$};
				\draw[dashed] (-1,-1) -- (-0.7,-1);
				\draw[dashed] (-0.1,-1) -- (0.3,-1);
				\draw (-0.4,-1) node{$\chi'''$};
				\draw [fill=black] (1.3,-1) ellipse (0.05 and 0.05);
				\draw [fill=black] (0.8,-0.2) ellipse (0.05 and 0.05);
				\draw [fill=black] (0.3,-1) ellipse (0.05 and 0.05);
				\draw (1.4,-0.75) node{$u_0$};
				\draw (0.45,-0.2) node{$u_1$};
				\draw (0.15,-0.75) node{$u_2$};
			\end{tikzpicture}
		\end{center}
		where $\chi'$, $\chi''$ and $\chi'''$ are pairwise disjoint subconfigurations of $\chi$ such that at least two of $\chi'$, $\chi''$ and $\chi'''$ contain more than one polygon distinct from $x$.		
		
		Let $y$ be some 3-gon in $\chi$ distinct from $x$. Since $Q$ is connected, there exists a string $w=\beta_1\ldots\beta_n$ such that $s(w)=x$ and $e(w)=y$. Recall that since $x$ is a 3-gon in $\chi$, $x$ is the source of 3 distinct arrows and the target of 3 distinct arrows in $Q$. Since at least two of $\chi'$, $\chi''$ and $\chi'''$ contain more than one polygon distinct from $x$, there exist pairwise distinct symbols $\alpha_1,\alpha_2,\alpha_3 \in Q_1 \cup Q^{-1}_1$ such that $\alpha_1$, $\alpha_2$ and $\alpha_3$ are not symbols of $w$ and $\alpha_1w$ and $\alpha_2\alpha_3 w$ are strings. Since $y$ is the source of 3 distinct arrows and the target of 3 distinct arrows in $Q$, there exist pairwise distinct symbols $\gamma_1,\gamma_2, \in Q_1 \cup Q^{-1}_1$ such that $\gamma_1$ and $\gamma_2$ are not symbols of $w$ and $w\gamma_1$ and $w\gamma_2$ are strings. Thus, there exists a wild subquiver $Q'$ of $Q$ with underlying graph
		
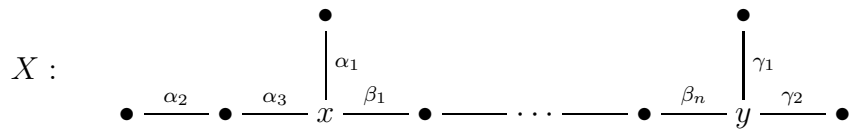
\begin{center}
			\begin{tikzpicture}
				\node[anchor=east] (1) at (0,0) {$X:$};
				\node[anchor=west] (2) at (0.4,0) {$\xymatrix{	& & \bullet \ar@{-}[d]^{\alpha_1} & & & & \bullet \ar@{-}[d]^{\gamma_1} \\
				\bullet \ar@{-}[r]^{\alpha_2} & \bullet \ar@{-}[r]^{\alpha_3} & x \ar@{-}[r]^{\beta_1} & \bullet \ar@{-}[r] & \cdots \ar@{-}[r] & \bullet \ar@{-}[r]^{\beta_n} & y \ar@{-}[r]^{\gamma_2} & \bullet}$};
			\end{tikzpicture}
		\end{center}
		The orientation of $Q'$ is determined by the string $w$ and the symbols $\alpha_1$, $\alpha_2$, $\alpha_3$, $\gamma_1$ and $\gamma_2$. Since   $Q'$ is constructed from strings, every directed path of $Q'$ avoids the relations in $I$. Thus $KQ'$ is a wild hereditary algebra and by Lemma~\ref{WildHeredSubQuiver}, $A$ is wild.
	\end{proof}	
	
	\section{The Proof of the Main Theorem} \label{MainProof}
	We can now bring together all of the results from previous sections to prove the main theorem (Theorem~\ref{Result1}). Throughout, let $A$ be a Brauer configuration algebra associated to a Brauer configuration $\chi$.
	
	($\Rightarrow$:) We prove the contrapositive. Namely, that if $\chi$ is none of the Brauer configurations presented in the list of Theorem~\ref{Result1}, then $A$ is wild. Since $\chi$ is not a Brauer graph, $A$ is $n$-serial with $n>2$. If $A$ is $n$-serial with $n>4$, then $A$ is wild by Theorem~\ref{TameQuad}. But if $n=4$ and not of the form given in Theorem~\ref{Result1}(iv), then $A$ is wild again by Theorem~\ref{TameQuad}.
	
	So suppose $A$ is triserial. This automatically implies that $\chi$ is not of the form given in Theorem~\ref{Result1}(i) or (iv). If $\chi$ contains at least two 3-gons, then $\chi$ cannot be of the form given in Theorem~\ref{Result1}(iii). By Proposition~\ref{Multiple3GonCase}, $A$ is wild if it is also not of the form given in Theorem~\ref{Result1}(ii).
	
	Now suppose $\chi$ contains precisely one 3-gon. If $\chi$ contains a cycle or a vertex of multiplicity strictly greater than one, then $\chi$ is necessarily not of the form in Theorem~\ref{Result1}(iii). By Proposition~\ref{WildCycleMult}, if $\chi$ is also not of the form in Theorem~\ref{Result1}(ii) (with $r=1$), then $A$ is wild.
	
	The only remaining case is where $\chi$ contains precisely one 3-gon and $\chi$ is a multiplicity-free tree. If $\chi$ is not of the form given in Theorem~\ref{Result1}(ii), then at least two of the disjoint subtrees connected to the unique 3-gon in $\chi$ contain at least two polygons each. By Theorem~\ref{ExceptionalAlgebras}, if $\chi$ is not of the form in Theorem~\ref{Result1}(iii), then $A$ is wild. This completes the proof.
	
	($\Leftarrow$:) If $\chi$ is of the form given in (i), then $A$ is biserial and hence tame. Cases (ii), (iii) and (iv) of Theorem~\ref{Result1} follow from Theorems~\ref{TameQC}, \ref{ExceptionalAlgebras} and \ref{TameQuad} respectively.
	
\end{document}